\documentclass[11pt]{amsart}
\usepackage{amsmath}
\usepackage{amsfonts}
\usepackage{amssymb}
\usepackage{amsthm}
\usepackage{mathrsfs}
\usepackage{mathtools}
\usepackage{latexsym}
\usepackage{enumerate}
\usepackage{cancel}
\usepackage{cases}
\usepackage{empheq}
\usepackage{multicol}
\usepackage{algorithm}
\usepackage{algpseudocode}
\usepackage{caption}
\usepackage{wrapfig}
\usepackage{subcaption}
\usepackage{float}
\usepackage{tikz}
\usepackage{pgfplots}
\pgfplotsset{compat=1.13}
\usepackage{color}

 \usepackage[backgroundcolor=gray!60,linecolor=black]{todonotes}
\usepackage{mathrsfs}
\usepackage{fontenc}
\usepackage{inputenc}
\usepackage{verbatim}
\theoremstyle{plain}
\newtheorem{theorem}{Theorem}[section]

\newtheorem{lemma}[theorem]{Lemma}

\theoremstyle{definition}
\newtheorem{definition}[theorem]{Definition}

\theoremstyle{remark}
\newtheorem{remark}[theorem]{Remark}
\numberwithin{equation}{section}
\numberwithin{figure}{section}   
\usepackage[square,comma,numbers,sort]{natbib}
\usepackage[colorlinks=true, pdfborder={ 0 0 0}]{hyperref}
\hypersetup{urlcolor=blue, citecolor=red}
\usepackage{url}
\newcommand{\vect}[1]{\mathbf{#1}}

\newcommand{\bk}{\vect{k}}

\newcommand{\bc}{\vect{c}}
\newcommand{\bB}{\vect{B}}

\newcommand{\bu}{\vect{u}}
\newcommand{\bv}{\vect{v}}
\newcommand{\bw}{\vect{w}}

\newcommand{\bx}{\vect{x}}
\newcommand{\by}{\vect{y}}

\newcommand{\bbf}{\vect{f}}

\newcommand{\field}[1]{\mathbb{#1}}
\newcommand{\nN}{\field{N}}
\newcommand{\nZ}{\field{Z}}

\newcommand{\nR}{\field{R}}

\DeclareSymbolFont{bbold}{U}{bbold}{m}{n}
\DeclareSymbolFontAlphabet{\mathbbold}{bbold}

\newcommand{\bvphi}{\boldsymbol{\vphi}}

\newcommand{\bomega}{\boldsymbol{\omega}}
\newcommand{\nT}{\mathbb T}
\newcommand{\vphi}{\varphi}

\newcommand{\cnj}[1]{\overline{#1}}
\newcommand{\pd}[2]{\frac{\partial #1}{\partial #2}}

\newcommand{\set}[1]{\left\{#1\right\}}
\newcommand{\ip}[2]{\left<#1,#2\right>}

\newcommand{\pnt}[1]{\left(#1\right)}
\newcommand{\pair}[2]{\left(#1,#2\right)}

\newcounter{my_counter}
\usepackage{placeins}
\title[Rotational Burgers, Kuramoto-Sivashinsky]
{Burgers equation with a twist: \\A study on rotational-form equations}
\date{\today}
\author{Adam Larios}
\address[Adam Larios]{Department of Mathematics, 
                University of Nebraska--Lincoln,
        Lincoln, NE 68588-0130, USA}
\email[Adam Larios]{alarios@unl.edu}

\begin{document}

\begin{abstract}
 A new three-dimensional (3D) equation is proposed, which is formed like Burgers' equation by starting with the 3D incompressible Navier-Stokes equations (NSE) and eliminating the pressure and the divergence-free constraint, but instead the Bernoulli pressure is eliminated, leaving only the rotational form of the nonlinearity.  This results in a globally well-posed 3D equation which has exactly the same energy balance as the 3D NSE.  Moreover, we show in simulations that the system seems to exhibit chaotic dynamics.  
 In the viscous case, we prove the global existence, uniqueness, and higher-order regularity of solutions to this equation with no restriction on the initial data other than smoothness.  In the inviscid case, local existence holds, but we give an example of a class of solutions with smooth initial data that develop a singularity in finite time in both 2D and 3D.  Moreover, a new numerical algorithm is presented in the 2D case, and simulations are included to illustrate the dynamics.  In addition, a rotational-form modification for the 2D Kuramoto-Sivashinsky equations (KSE) is proposed, and global well-posedness is also established.  We also discuss several related ``rotational form'' equations, and some pedagogical considerations.  Global well-posedness for the original 3D NSE and 2D KSE remains a challenging open problem, but it is hoped that by focusing on the rotational term, new insight may be gained.  
\end{abstract}

\maketitle

\thispagestyle{empty}

\section{Introduction}\label{secInt}
\noindent
A central difficulty in fluid dynamics is to understand how to handle the nonlinear advective term $(\bu\cdot\nabla)\bu$ in the equations of motion, where $\bu$ is the fluid velocity.  However, it is well-known that the Navier-Stokes equations (NSE) of incompressible, constant density flows, can be rewritten in an equivalent form, replacing the advective term by a rotational term; namely $(\nabla\times\bu)\times\bu$, and replacing the pressure $p$ by the Bernoulli pressure (also called the \emph{dynamic pressure}); namely  $p+\frac12|\bu|^2$.  The rotational term has identical scaling properties to the advective term, but it is arguably more geometric in nature, having the appealing property that it is automatically pointwise orthogonal to both the velocity and the vorticity $\bomega:=\nabla\times\bu$.  Momentarily neglecting viscous effects and external forcing, the incompressible NSE can therefore be thought of as a system which evolves by a kind of length-preserving local rotation, orthogonal to its own velocity and vorticity, and a length-deforming dynamic pressure gradient, which only acts to enforce the divergence-free condition in the face of this local rotation.  In the present work, we focus on this rotational evolution isolated from the length-deforming effects of the Bernoulli pressure. 
For over a century, people have tried to understand aspects of the Navier-Stokes equations by considering a simplified model based on dropping the pressure term $p$ and the divergence-free constraint from the equations, resulting in what is called the Burgers equation (first introduced by H. Bateman in 1915 \cite{Bateman_1915_burgers}, and later popularized by J. Burgers \cite{Burgers_1948}).  The study of the Burgers equation has led to so many fruitful insights that we cannot possibly list them all here (although see \cite{bonkile2018burgers,Pooley_Robinson_2016} for historical discussions, and \cite{Bec_Khanin_2007_Burgers_turbulence} for an engaging discussion on ``Burgers turbulence'').  We also mention in passing another approach to eliminating the pressure: the so-called ``cheap Navier-Stokes equation'' $u_t =\triangle u + (-\triangle)^{1/2}u^2$, proposed in \cite{Montgomery_Smith_2001_cheapNSE}, where it was shown to experience finite-time blow-up.   Instead of these scalar approaches, we consider a different equation, where we drop the full Bernoulli pressure (and the divergence-free constraint), resulting in a ``rotational form'' Burgers-like equation (this is the ``twist'' referred to in the title).  To be clear, our interest here is not to model a physical phenomenon, but to study the dynamical effect of the rotational nonlinearity.  As we show below, this system is dynamically non-trivial and globally well-posed in the viscous case, and has several properties that make it dynamically distinct from Burgers equation. It can also develop a finite-time singularity in the inviscid case.
We also consider a similar modification of the vector form of the multi-dimensional Kuramoto-Sivashinsky equation (KSE) of flame fronts.  
Unlike in the NSE case, replacing the nonlinear term by a rotational form does not seem to have a clear physical interpretation (e.g., we cannot say we are ``dropping the Bernoulli pressure''); however, this form is quite interesting from the perspective of global well-posedness of equations.  To illustrate this, we note that proving global well-posedness of the 2D KSE (without smallness restrictions on the initial data or domain size) is a long-standing open problem, despite much work on this problem (see discussions in, e.g, \cite{Ambrose_Mazzucato_2018,Ambrose_Mazzucato_2021,CotiZelati_Dolce_Feng_Mazzucato_2021,Feng_Mazzucato_2021,Larios_Martinez_2024,Larios_Rahman_Yamazaki_2021_JNLS_KSE_PS,Larios_Yamazaki_2020_rKSE}).  However, as we show below, replacing the advective term $(\bu\cdot\nabla)\bu$ by the rotational term $(\nabla\times\bu)\times\bu$ immediately allows for a proof of global well-posedness.\footnote{Chronologically, I first thought of the rotational modification of the KSE as a way to have a well-posed multi-dimensional KSE-like equation, and only later realized that a similar modification could be made for the Navier-Stokes equations, and that one could interpret the result as ``Navier-Stokes without the Bernoulli pressure.''}  Moreover, as we show in simulations, the dynamical behavior of this system is non-trivial.  
It is hoped that by studying the rotational modifications proposed here, new perspectives and insight into original equations will be gained, as in the case of the Burgers equation which has led to so much understanding.  Moreover, the dynamics of these rotational-form equations appear interesting in their own right, and may provide new objects of study in the theory of dynamical systems of partial differential equations.
We note that there have been several works (see, e.g.,  \cite{Layton_Manica_Neda_Olshanskii_Rebholz_2009rotational,
Gardner_Larios_Rebholz_Vargun_Zerfas_2020_VVDA,
Manica_Neda_Olshanskii_Rebholz_2011,
Moser_Moin_1987effects,
Lube_Olshanskii_2002,
Palha_Gerritsma_2016,
Rukavishnikov_Rukavishnikov_2023_corner_JCAM} and the references therein), that have exploited the rotational form of Navier-Stokes for improved numerical properties (see, e.g., \cite{Canuto_Hussaini_Quarteroni_Zang_2006,Canuto_Hussaini_Quarteroni_Zang_2007_complex_geom} for a detailed discussion).  If one does not use so-called ``div-grad stabilization'' as in \cite{Layton_Manica_Neda_Olshanskii_Rebholz_2009rotational}, this can sometimes lead to a decrease in accuracy, possibly due to discretization near the wall \cite{Horiuti_1987_JCP,Horiuti_Itami_1998_JCP,Zang_1991_rotation_skew_sym_ANM}, although   \cite{Olshanskii_2002_low,Olshanskii_Reusken_2002_multigrid_SJSC} suggest that these inaccuracies may be due to difficulties in computing the Bernoulli pressure.
The present work is organized as follows.
In Section \ref{sec_relations_NSE}, we discuss relations to the Navier-Stokes equations, and also discuss some other related equations.
In Section \ref{secPre}, we lay out some preliminary results and notation.
In Section \ref{sec_rB_well_posedness} we prove global well-posedness for \eqref{Burgers_twist}.
In Section \ref{sec_global_attractor}, for a damped-driven version of \eqref{Burgers_twist}, we establish the existence of a global attractor in $H^1$, but in the $L^2$ topology, and also a global attractor in the restricted space $H^1\cap L^\infty$.  
In Section \ref{sec_inviscid_case}, we establish local well-posedness of analytic solutions to the inviscid rotational Burgers equation.  We then show a family of initial data for which the corresponding solutions to the inviscid equations blow up in finite time.
In Section \ref{sec_rotKSE}, we discuss the Kuramoto-Sivashinsky equation, propose a rotational modification of the equation, and prove its global well-posedness.  
In Section \ref{subsec_rot_num}, we propose a new numerical algorithm for \eqref{Burgers_twist}, in the 2D case, based on the geometry of the rotational form.
In Section \ref{sec_simulations},  show some simulations in different regimes to illustrate the dynamics of the viscous  equations.
In Section \ref{sec_additional_considerations}, we give some additional considerations. 
We give some concluding remarks in Section \ref{sec_conclusions}.
In subsection \ref{sec_pedagogical}, we discuss how the ``Burgers equation with a twist'' may be useful for example in an advanced PDE course, since it shares many of the features of the 3D NSE without the need to develop the divergence-free framework.  Moreover, many of the subtler details of the well-posedness theory for these types of equations is scattered about in various references or difficult to find in the literature.  Therefore, in order to keep the paper a somewhat self-contained reference (and also because estimates for Burgers-type equations are somewhat less flexible that those for the NSE, due to the lack of symmetry properties in the nonlinear term, and therefore do not always follow similarly to the NSE case, see, e.g., Remark \ref{remark_no_weak_Burgers_twist_sol}), we include the full details in most of the proofs. 
\begin{remark}
Many of the results in this paper, such as Theorem \ref{thm_strong_local} can be easily adapted to the case of a bounded domain $\Omega\subset\nR^3$ with homogeneous Dirichlet (i.e., ``no slip'') boundary conditions $\bu\big|_{\partial\Omega}=\mathbf{0}$, or in the case of the Theorem \ref{thm_GWP_vKSE}, $\bu\big|_{\partial\Omega}=\triangle \bu\big|_{\partial\Omega}=\mathbf{0}$.   (\textit{cf.} Remark \ref{remark_attractor_needs_poincare}.) However, to keep the presentation simpler, we focus on the case of periodic domains.
\end{remark}
\section{Relations to Navier-Stokes and other equations}\label{sec_relations_NSE}
 \noindent
In this section, we motivate the main equation by discussing it in the context of the Navier-Stokes equation and the usual Burgers equation.
\subsection{The Bernoulli pressure and the Lamb vector}
We recall the incompressible, constant-density Navier-Stokes equations, given by 
\begin{subequations}\label{NSE}
\begin{align}
\pd{\bu}{t}+(\bu\cdot\nabla)\bu  +\nabla p &= \nu \triangle\bu +\bbf,\\
\nabla\cdot\bu &= 0,
\end{align}
\end{subequations}
taken here with just periodic boundary conditions for simplicity.  Here, $\bu=\bu(\bx,t)=(u_1(\bx,t),u_2(\bx,t),u_3(\bx,t))$ is the fluid velocity, $p=p(\bx,t)$ is the fluid pressure, $\nu\geq0$ is the kinematic viscosity, and $\bbf=\bbf(\bx,t)$ is the body force.
Note that it is common practice, especially in velocity-vorticity formulations
(see, e.g., \cite{G91,Gardner_Larios_Rebholz_Vargun_Zerfas_2020_VVDA,GHH90,GHOR15,HOR17,Larios_Pei_Rebholz_2018,LOR11,LYM06,MF00,OR10,WB02,WWW95})
to use the vector identity, 
\begin{align}\label{vec_identity}
 (\bu\cdot\nabla)\bu = (\nabla\times\bu)\times\bu + \tfrac12\nabla|\bu|^2,
\end{align}
to rewrite equations \eqref{NSE} in the so-called ``rotational form,'' given by
\begin{subequations}\label{NSE_vor}
\begin{align}
\pd{\bu}{t}+\bomega\times\bu  +\nabla \pi &= \nu \triangle\bu+\bbf,\\
\nabla\cdot\bu &= 0,\label{NSE_vor_div}
\end{align}
\end{subequations}
where $\pi:=p + \tfrac12|\bu|^2$ is the Bernoulli pressure\footnote{The Bernoulli pressure is sometimes called the ``dynamic pressure'', but this risks confusion with other quantities, such as $\frac12\rho|\bu|^2$, so we refer to it as the Bernoulli pressure.} and $\bomega:=\nabla\times\bu$ denotes the vorticity.  
As mentioned above, this rotational form of the Navier-Stokes equations reveals that the nonlinear parts of the Navier-Stokes equations (either $(\bu\cdot\nabla)\bu+\nabla p$ or $\bomega\times\bu+\nabla \pi$) can be thought of as a ``twist plus deformation-back-to-incompressiblity'' action.  That is, the term $\bomega\times\bu$, known as the Lamb vector\footnote{\label{footnote_Lamb}The vector $(\nabla\times\bu)\times\bu$, is sometimes called the Lamb vector after Sir Horace Lamb \cite{Lamb1924hydrodynamics,Truesdell2018Vorticity}.  It is related to the Craik–Leibovich vortex force (see, e.g., \cite{Holm_1996_CL} and the references therein) and is an important term in Crocco's Theorem \cite{Crocco_1937_ZAMM}. We note that it need not be a gradient; for example, for $\bu=(y,0,x)$, $(\nabla\times\bu)\times\bu=(-x,-y,y)$, which has non-zero curl.  The vector $\bu=(\sin(y),0,\sin(x))$ provides a similar example in the periodic case.  Note that the Lamb vector is identically zero in Beltrami flows (by definition), and also trivially in irrotational (i.e., $\bomega\equiv\mathbf{0}$) flows.  It is frequently used in Navier-Stokes simulations (with Bernoulli pressure), as computing $(\nabla\times\bu)\times\bu$ requires fewer derivatives and fewer multiplications than directly computing $(\bu\cdot\nabla)\bu$, although alternatives such as the Basdevant formulation \cite{Basdevant_1983_JCP} are also popular.  We note that the beautiful work \cite{Hamman_Klewicki_Kirby_2008_Lamb},  discusses the importance of a geometric and dynamical understanding of the Lamb vector in the context of understanding turbulence.  In particular, the authors demonstrate that regions where the divergence of the Lamb vector switches sign are associated with strong turbulence.}, rotates the velocity orthogonally to itself and the vorticity (see Figure \ref{figs:fields}), locally creating circular motions without any stretching or compression of individual vectors.  
\begin{figure}[htp!]
\centering
\begin{subfigure}[t]{.5\textwidth}
    \includegraphics[width=\linewidth,trim=0 0 0 150,clip]{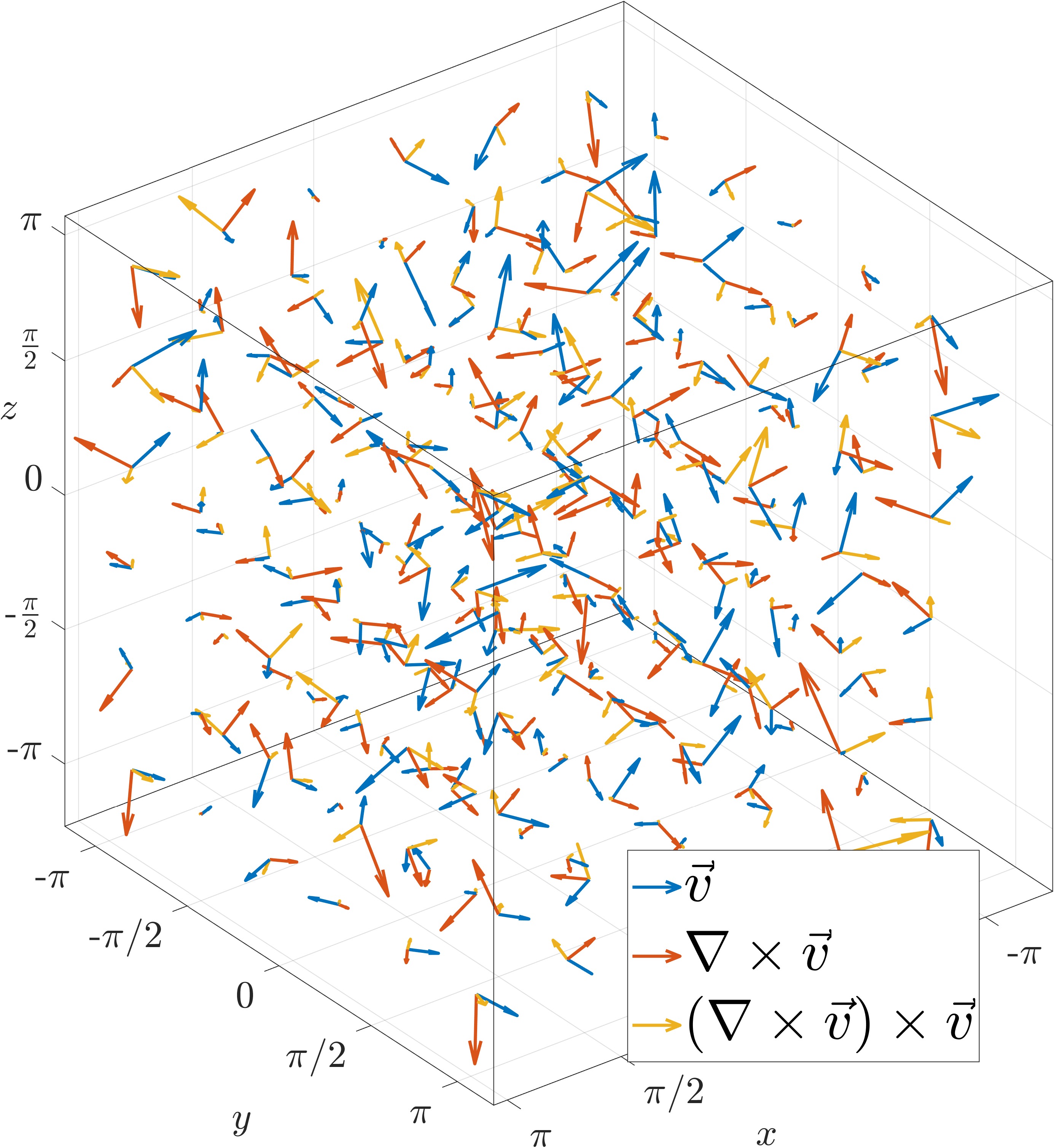}
    \caption{\label{fields_3D}\footnotesize 3D velocity, vorticity, and Lamb vector}
\end{subfigure}
\hfill
\begin{subfigure}[t]{.49\textwidth}
    \includegraphics[width=\linewidth,trim=0 0 0 0,clip]{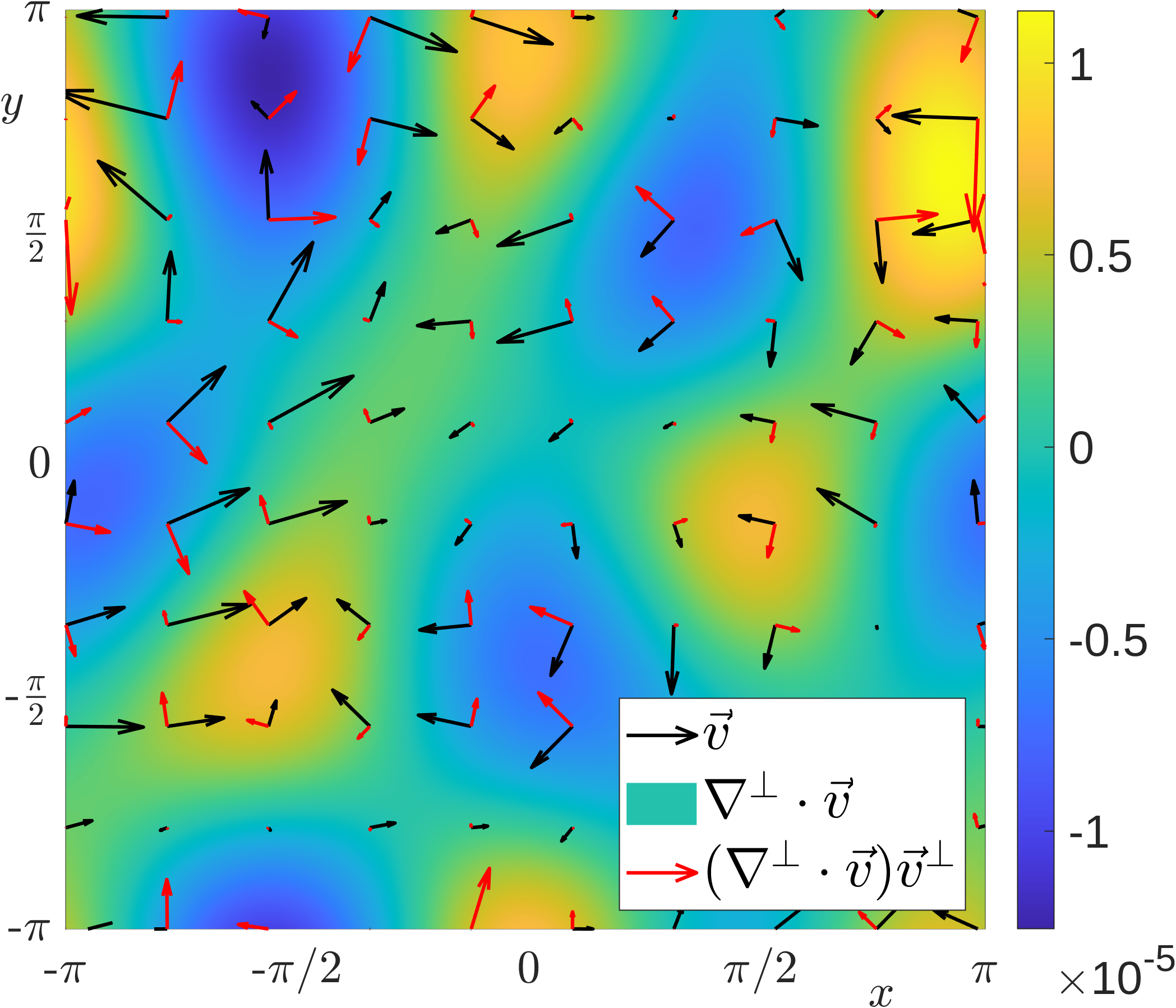}
    \caption{\label{fields_2D}\footnotesize 2D (curl shown as scalar quantity)}
\end{subfigure}
\caption{\label{figs:fields}\footnotesize 
(A): Visualization of a randomly-generated smooth 3D vector field $\vec{v}$, along with its curl $\nabla\times \vec{v}$, and the Lamb vector $(\nabla\times \vec{v})\times \vec{v}$. (B): 2D analogs of the fields in (A).  Vectors not to scale.
}
\end{figure}
However, this kind of rotational motion by itself can violate the divergence-free condition \eqref{NSE_vor_div}, so the $\nabla \pi$ term precisely acts to enforce this condition.  Moreover, since stretching or compression of individual vectors cannot arise from the rotational term, it must arise only from the $\nabla \pi$ term (of course, viscosity or external forcing may contribute as well). 
The above remarks raise questions about the dynamical effect of the rotational term alone.  To this end, let us first recall the Burgers equation, which has the multi-dimensional viscous form
\begin{align}\label{Burgers}
\pd{\bu}{t}+(\bu\cdot\nabla)\bu &= \nu \triangle\bu+\bbf.
\end{align}
The Burgers equation is often used as a test-bed for ideas, since heuristically speaking, it is in some sense ``Navier-Stokes without the pressure'' (where, of course, one must also drop the divergence-free constraint).  In the same light, we can propose the following initial value problem, which is in some sense ``Navier-Stokes without the \emph{Bernoulli pressure}'' (where we still drop the divergence-free constraint); namely,
\begin{subequations}
\label{Burgers_twist_all}
\begin{empheq}[left=\empheqlbrace]{align}
\label{Burgers_twist}
\pd{\bu}{t}+\bomega\times\bu &= \nu \triangle\bu+\bbf, \qquad \bomega:=\nabla\times\bu,
\\
\label{Burgers_twist_initial_data}
\bu(0) &= \bu_0.
\end{empheq}
\end{subequations}
Equation \eqref{Burgers_twist} is a sort-of ``Burgers equation with a twist'' or ``rotational Burgers equation.'' It is a major goal of the present work to study this equation\footnote{We note in passing that \eqref{Burgers_twist} bears a superficial resemblance to the electron-MHD equation $\bB_t + \nabla\times((\nabla\times \bB)\times \bB) = \eta\triangle \bB + \bbf$, see, e.g., \cite{Gordeev_Kingsep_Sergeevic_Rudakov_1994_electron_MHD,Dai_Krol_Lio_2022_electron_MHD,Jeong_Oh_2025_electron_MHD} and the references therein.  However, electron-MHD (E-MHD) seems to be of a fundamentally different character than that of the \eqref{Burgers_twist}.  For example, E-MHD conserves the divergence-free condition, while \eqref{Burgers_twist} does not.  Moreover, E-MHD is second-order quasilinear and allows for finite-time blow-up \cite{Dai_2025_blowup_forced_electron_MHD}, while equation \eqref{Burgers_twist}is semilinear and, as we show, is globally well-posed when $\nu>0$.  When $\nu=0$, \eqref{Burgers_twist} is first-order quasilinear and, as we show, allows for finite-time blow-up.}.
Denoting the velocity magnitude by $|\bu|:=\sqrt{u_1^2+u_2^2+u_3^2}$ and also $|\nabla\bu|:=\sqrt{\sum_{i,j}(\partial_{x_i}u_j)^2}$, and recalling the well-known vector identity,
\begin{align}
\label{id_u_dot_Laplace}
    \bu\cdot(\triangle\bu) &= \tfrac12\triangle|\bu|^2 - |\nabla\bu|^2,
\end{align}
and the standard cross-product orthogonality for vectors $\mathbf{A}$ and $\mathbf{B}$,
\begin{align}\label{cross_prod_ortho}
(\mathbf{A}\times\mathbf{B})\cdot\mathbf{B}=0,
\end{align}
one obtains the following local identity for smooth solutions to \eqref{Burgers_twist}:
\begin{align}\label{Burgers_vor_dot}
\frac12\frac{\partial}{\partial t}|\bu|^2 = \frac{\nu}{2}\triangle|\bu|^2 - \nu|\nabla\bu|^2 + \bbf\cdot\bu.
\end{align}
 Thus, we see that, formally, in the case when $\bbf\equiv\mathbf{0}$, if $|\bu(t,\bx)|^2$ is at a spatial maximum $\bx$, it must be non-increasing in time due to \eqref{Burgers_vor_dot}.  Hence, a maximum principle holds, so long as one can justify this formal argument rigorously, which we do later in this manuscript (although we will allow a non-zero force).  Such an argument was likely first made by O. Ladyzhenskaya et al. (see, e.g.,  \cite{Ladyzhenskaya_1968}) in the context of the multi-dimensional Burgers equation \eqref{Burgers}, although see \cite{Pooley_Robinson_2016} for an historical discussion and a modern proof.  In the work below, we rigorously establish a maximum principle, and use it to prove global well-posedness of \eqref{Burgers_twist} when $\nu>0$.
\subsection{Related equations and other considerations.}
Here, we highlight some differences between the standard Burgers equation \eqref{Burgers} and equation \eqref{Burgers_twist}.  Note that the multi-dimensional Burgers equation \eqref{Burgers} has the one-dimensional Burgers equation, 
\begin{align}\label{Burgers1D}
    u_t+uu_x=\nu u_{xx}+f,
\end{align} as a special case. 
If we try to do something similar for \eqref{Burgers_twist}, by examining the Lamb vector,
\begin{align}\label{lamb_vector_explicit}
\bomega\times\bu
&=
\begin{pmatrix}
u_3\partial_z u_1-u_3\partial_x u_3-u_2\partial_x u_2+u_2\partial_y u_1\\
u_1\partial_x u_2-u_1\partial_y u_1-u_3\partial_y u_3+u_3\partial_z u_2\\
u_2\partial_y u_3-u_2\partial_z u_2-u_1\partial_z u_1+u_1\partial_x u_3
\end{pmatrix},
\end{align}
we see that if, e.g., $u_1$ is independent of $y$ and $z$, and $u_2\equiv u_3\equiv0$, then the Lamb vector $\bomega\times\bu$ vanishes, and one obtain the heat equation $\partial_t u_1 = \nu \partial_{xx}u_1+f$, rather than \eqref{Burgers1D}.  However, see Section \ref{subsec_finite_time_blow_up} for a construction for which \eqref{Burgers1D} does arise.
Another difference between \eqref{Burgers_twist} and \eqref{Burgers} is that Burgers equation is Galilean invariant, whereas \eqref{Burgers_twist} is not.  Indeed, using \eqref{vec_identity}, we see that \eqref{Burgers_twist} is formally equivalent to the following system:
\begin{align}\label{Burgers_twist_adv_form}
\bu_t+\bu\cdot\nabla\bu &= \nu \triangle\bu + \tfrac12\nabla|\bu|^2+\bbf.
\end{align}
Note that \eqref{Burgers_twist_adv_form} is dimension-independent, as it does not involve the cross product.  
Next, given a constant vector $\bc$, let $\bu$ be a smooth function satisfying \eqref{Burgers_twist}, and denote $\bv(\by,t) = \bu(\by-\bc t,t)+\bc$. Straight-forward calculations show
\begin{align*}
\bv_t + \bv\cdot\nabla\bv =
\nu \triangle\bv + \tfrac12\nabla |\bv|^2 - \bc\cdot\nabla \bv+\bbf,
\end{align*}
and hence \eqref{Burgers_twist} does not enjoy Galilean invariance.
Equation \eqref{Burgers_twist} perhaps has more in common with the Navier-Stokes equations \eqref{NSE} than \eqref{Burgers}, in that smooth solutions for both systems satisfy the same energy balance relation (which does \textit{not} hold for Burgers equation \eqref{Burgers} in dimension $n\geq2$); namely,
\begin{align}\label{en_eq_heat}
\frac12\frac{d}{dt}\|\bu\|_{L^2}^2  + \nu\|\bu\|_{L^2}^2
= (\bbf,\bu).
\end{align}
This is because, in both cases the non-linear term is $L^2$-orthogonal to the flow.  However, it is worth noting that \eqref{Burgers_twist} also satisfies this property \textit{locally}; i.e.,
$(\bomega\times\bu)\cdot\bu=0$, whereas for smooth solutions to the NSE, one has only
\begin{align}\label{en_local_NSE}
    (\bu\cdot\nabla\bu + \nabla p)\cdot\bu
    =
    \nabla\cdot\pnt{(\tfrac12|\bu|^2+p)\bu}.
\end{align}
(Note again the appearance of the Bernoulli pressure.)  In the celebrated papers \cite{Caffarelli_Kohn_Nirenberg_1982,Duchon_Robert_2000}, the identity \eqref{en_local_NSE} played a central role in understanding local energy balance, and the definition of a \textit{suitable weak solution}\footnote{see the discussion in \cite{Guermond_2008_suitable} and the references therein for practical considerations regarding suitable weak solutions}.  It may be interesting to follow these in the context of \eqref{Burgers_twist} (or even \eqref{NSE_vor}).  However, note that for \eqref{Burgers_twist}, more regularity of the solution seems to be demanded, as there is no clear way to move all derivatives to test functions.
In Section \ref{sec_rotKSE}, we also propose a rotational modification of the Kuramoto-Sivashisnky (KSE) equations \eqref{KSE}.  Namely, we propose the following system:
\begin{align}\label{rotKSE_intro}
    \pd{\bu}{t}+\bomega\times\bu + \triangle\bu + \triangle^2\bu &= \mathbf{0}.
\end{align}
We think of this as a sort of ``Kuramoto-Sivashisnky with a twist.''
After describing this system in more detail in Section \ref{sec_rotKSE}, we prove its global well-posedness.  Showing the global well-posedness or finite-time blow-up of the original $n$-dimensional ($n\geq2$) KSE remains a long-standing open problem.
\section{Preliminaries}\label{secPre}
\noindent
In this section, we lay out some notation and preliminary notions that will be used later in the text.  We denote the $2\pi$-periodic box $\nT^n:=\nR^n/(2\pi\nZ)^n=[0,2\pi]^n$, and we denote its volume of the torus by $|\nT^n|=(2\pi)^n$.    We denote the set of  vector-valued real $L^2$ functions on $\nT^n$ by
\begin{align*}
L^2(\nT^n):=\set{
\bu\bigg|\bu(\bx)=\sum_{\bk\in\nZ^n}\widehat{\bu}_\bk e^{i\bk\cdot\bx},   \cnj{\widehat{\bu}_\bk} = \widehat{\bu}_{-\bk}, \text{ and }\sum_{\bk\in\nZ^n}|\widehat{\bu}_\bk|^2<\infty
}.
\end{align*}
Note that the equations discussed in this work typically do not conserve the mean of solutions, so there is no mean-free condition on these spaces. Indeed, formally integrating \eqref{Burgers_twist_adv_form}, we find only that
\[
\frac{d}{dt}\int_{\nT^3}\bu\,d\bx 
= 
\int_{\nT^3}(\nabla\cdot\bu)\bu\,d\bx 
+\int_{\nT^3}\bbf\,d\bx,
\]
indicating that, e.g., positive divergence may amplify the average (cf. \cite{Larios_Martinez_2024}).
We denote the (real) $L^2$ inner-product and $H^s$ Sobolev norm by
\[
 (\bu,\bv) := \sum_{i=1}^n\int_{\nT^n} u_i(\bx) v_i(\bx)\,d\bx,
\;\;
 \|\bu\|_{H^s} := \Big((2\pi)^n\!\!\sum_{\bk\in\nZ^n}(1+|\bk|^{2})^{s}|\widehat{\bu}_\bk|^{2}\Big)^{\frac12}\!\!,
\]
and the corresponding space $H^s\equiv H^s(\nT^n) = \set{\bu\in L^2(\nT^n)\big|\|\bu\|_{H^s}<\infty}$. We also denote the mean-free spaces
\[\textstyle
\dot{H}^s\equiv \dot{H}^s(\nT^n):=\set{\bu\in H^s(\nT^n)\Big|\frac{1}{|\nT^n|}\int_{\nT^n}\bbf\,d\bx=\hat{\bbf}_{\mathbf{0}}=\mathbf{0}}, \qquad \dot{L}^2:=\dot{H}^0.
\]
We denote the following action for negative spaces
\begin{align*}
\ip{\bbf}{\bu}\equiv\ip{\bbf}{\bu}_{H^{-1},H^1}&:=
(2\pi)^n\sum_{\bk\in\nZ^n}\hat{\bbf}_\bk\overline{\widehat{\bu}_\bk}.
\end{align*}
If $\bu\in H^1$, and $\bbf\in \dot{H}^{-1}$,
it holds from the mean-free property of $\bbf$ that
\begin{align}
  |\ip{\bbf}{\bu}| \notag
  &\leq (2\pi)^n\!\!\sum_{\substack{\bk\in\nZ^n\\\bk\neq \mathbf{0}}}\frac{|\hat{\bbf}_\bk|}{|\bk|}|\bk||\widehat{\bu}_\bk|
  \leq (2\pi)^n\Big(\sum_{\substack{\bk\in\nZ^n\\\bk\neq \mathbf{0}}}\frac{|\hat{\bbf}_\bk|^2}{|\bk|^2}\Big)^{\frac12}
\Big(\sum_{\substack{\bk\in\nZ^n\\\bk\neq \mathbf{0}}}|\bk|^2|\widehat{\bu}_\bk|^2\Big)^{\frac12}
\\&\leq \label{Hm1_bound}
   (2\pi)^n\Big(\sum_{\bk\in\nZ^n}\frac{2|\hat{\bbf}_\bk|^2}{1+|\bk|^2}\Big)^{\frac12}
\Big(\sum_{\bk\in\nZ^n}|\bk|^2|\widehat{\bu}_\bk|^2\Big)^{\frac12}
=
  \sqrt{2}\|\bbf\|_{H^{-1}}\|\nabla\bu\|_{L^2}.
\end{align}
For any $N\in\nN$ and $\bu\in H^s(\nT^n)$ given by $\bu(\bx)=\sum_{\bk\in\nZ^n}\widehat{\bu}_ke^{i\bk\cdot\bx}$, we denote the projection operator $P_N$ and its compliment $Q_N$ by
\[
 P_N \bu = \sum_{\substack{\bk\in\nZ^n\\|\bk|\leq N}}\widehat{\bu}_ke^{i\bk\cdot\bx},\qquad Q_N:=I-P_N.
\]
We recall the following well-known projection estimate for any $\bu\in H^{s}(\nT^n)$:
\begin{align}\label{PNproj}
 \|P_N\bu\|_{H^{s+r}}^2 
 \leq
 (2\pi)^n\sum_{\substack{\bk\in\nZ^n\\|\bk|\leq N}}(1+|\bk|^{2})^{s+r}|\widehat{\bu}_\bk|^{2}
 \leq
 (1+N^2)^{r}\|\bu\|_{H^{s}}^2.
\end{align}
and for any $\bu\in H^{s+r}(\nT^n)$:
\begin{align}\label{QNproj}
 \|Q_N\bu\|_{H^s}^2 
 \leq
 (2\pi)^n\sum_{\substack{\bk\in\nZ^n\\|\bk|> N}}\frac{|\bk|^{2r}}{N^{2r}}(1+|\bk|^{2})^{s}|\widehat{\bu}_\bk|^{2}
 \leq
 N^{-2r}\|\bu\|_{H^{s+r}}^2.
\end{align}
We denote the smallest positive eigenvalue of the negative Laplacian by $\lambda_1>0$.  Note that on $\nT^n$, $\lambda_1=1$, but it is a convenient notation as it carries units of $(\text{length})^{-2}$, so that $(\nu\lambda_1)^{-1}$ is a natural time scale.
We denote by $C$, $C_\nu$, $C_{\gamma,\nu}$, etc.  positive constants that may change from line to line.  
We recall Agmon's inequality for $0\leq s_1<\frac{n}{2}<s_2$, 
\begin{align}\label{agmon}
\|\bu\|_{L^\infty}\leq C\|\bu\|_{H^{s_1}}^{\theta}\| \bu \|_{H^{s_2}}^{1-\theta},
\end{align}
where $\frac{n}{2}=\theta s_1 + (1-\theta)s_2$, and Poincar\'e's inequality for mean-free functions, 
\begin{align}\label{poincare}
\int_{\nT^n}\bu\,d\bx=\mathbf{0}
\quad\Rightarrow\quad
\|\bu\|_{L^p}\leq C\|\nabla\bu\|_{L^p},\qquad 1\leq p\leq\infty.
\end{align}
However, as mentioned above, we cannot assume that solutions are mean-free, and hence we only apply Poincar\'e's inequality to derivatives of solutions.  We have the usual Gagliardo-Nirenberg-Sobolev inequalities:
\begin{align}\label{GNS_full}
\|\bu\|_{L^3(\nT^3)}\leq C\|\bu\|_{L^2(\nT^3)}^{1/2}\|\bu\|_{H^1(\nT^3)}^{1/2},
\quad
\text{and}
\quad
\|\bu\|_{L^6(\nT^3)}\leq C\|\bu\|_{H^1(\nT^3)},
\end{align}
but note carefully that the full $H^1$ norm appears on the right-hand side, since $\bu$ need not be mean-free.  However, the periodic boundary conditions imply that $\partial_\alpha\bu$ is mean-free whenever the multi-index $\alpha$ satisfies $|\alpha|\geq1$.  Hence, for $n=3$,
\begin{subequations}\label{GNS}
\begin{align}
\label{poincare_on_derivative}
|\alpha|\geq1\quad\Rightarrow\quad
 \|\partial_\alpha\bu\|_{L^p}
 &\leq 
 C\|\nabla\partial_\alpha\bu\|_{L^p}
 \quad(\text{for }1\leq p\leq\infty),
 \\
\label{GNS3}
|\alpha|\geq1\quad\Rightarrow\quad
 \|\partial_\alpha\bu\|_{L^3}
 &\leq 
 C\|\bu\|_{L^2}^{1/2}\|\nabla\partial_\alpha\bu\|_{L^2}^{1/2},
\\\label{GNS6}
|\alpha|\geq1\quad\Rightarrow\quad
\|\partial_\alpha\bu\|_{L^6}
&\leq 
C\|\nabla\partial_\alpha\bu\|_{L^2},
\end{align}
\end{subequations}
whenever the right-hand sides are finite.  We also note the following inequality holding in 3D, which is not new but recorded here for reference.
\begin{align}\label{H1_leq_H1_H2}
 \|uv\|_{H^1}
 &\leq 
 \|uv\|_{L^2} + \|u\nabla v\|_{L^2}  + \|(\nabla u) v\|_{L^2} 
 \\&\leq \notag
  \|u\|_{L^\infty}\|v\|_{L^2} + \|u\|_{L^\infty}\|\nabla v\|_{L^2}  + \|\nabla u\|_{L^6}  \|v\|_{L^3} 
\\&\leq \notag
  C\|u\|_{H^2}\|v\|_{H^1}
\end{align}
Replacing $\vect{B}$ by $\vect{B}+\vect{C}$ in  \eqref{cross_prod_ortho} yields the well-known triple product identity
\begin{align}\label{triple_product}
 (\vect{A}\times\vect{B})\cdot \vect{C}
=(\vect{C}\times\vect{A})\cdot\vect{B}
=(\vect{B}\times\vect{C})\cdot\vect{A}.
\end{align}
Identity \eqref{triple_product}, along with integration by parts, can be used to prove the following identity, which holds for smooth periodic vector fields $\bu,\bv,\bw$.
\begin{align}\label{ibp_Lamb}
    ((\nabla\times\bu)\times\bv,\bw)
    &=
    (\bv(\nabla\cdot\bw)-\bw (\nabla\cdot\bv)+(\bw\cdot\nabla)\bv- (\bv\cdot\nabla)\bw,\bu).
\end{align}
In particular, note that, unlike in the Navier-Stokes equations, there does not seem to be a way to move all the derivatives to, say, a test-function $\bw$.  This fact significantly complicates the analysis of \eqref{Burgers_twist} (see Remark \ref{remark_no_weak_Burgers_twist_sol}).
 \section{Global Well-Posedness for a Viscous Rotational Burgers Equation}\label{sec_rB_well_posedness}
  \noindent
In this section, we prove the global well-posedness of \eqref{Burgers_twist} with $\nu>0$.  
We take the boundary conditions to be periodic, and the initial data for this equation to be in $\bu_0\in H^1(\nT^n)$ ($n=2$ or $3$). 
The proof of local well-posedness for this model superficially follows similarly to that of the Navier-Stokes equations, but there are subtle, non-trivial differences due to the lack of the divergence-free condition.  
Namely, the lack of a divergence-free condition prevents one from passing the derivatives in the nonlinearity fully onto test functions, and as a result, the proof is more delicate than the standard proof of existence of solutions to the Navier-Stokes equations (see, e.g., Remark \ref{remark_no_weak_Burgers_twist_sol} below).  Moreover, solutions do not preserve the mean, and hence cannot be assumed to be mean-zero, meaning tools like the Poincar\'e inequality \eqref{poincare} are not available (however, \eqref{GNS_full} and \eqref{GNS} still hold).  Hence, we give the fully rigorous details of the proof.\footnote{The proof bears some similarity with that of proofs of existence of solutions to Burgers equation (see, e.g., \cite{Pooley_Robinson_2016}), but here, the vanishing of the nonlinear term can be exploited, leading to some differences in the proof.}  However, to avoid some technical difficulties, we assume $\bbf$ is mean-free so that \eqref{Hm1_bound} holds.
As we will see in the proof (see Remark \ref{remark_no_weak_Burgers_twist_sol}), a difficulty arises that does not appear in the Navier-Stokes equations, preventing the use of the standard methods for proving the existence of weak solutions.  Hence, we do not define or investigate a notion of weak solutions to \eqref{Burgers_twist}.  We begin with a definition of strong solutions.
\begin{definition}\label{def_strong_sol}
Let $\nu>0$, $T>0$, $\bu_0\in H^1(\nT^3)$, and $\bbf\in L^2(0,T;\dot{L}^2)$.  We say that $\bu$ is a \textit{strong solution} to \eqref{Burgers_twist} on the interval $[0,T]$ if $\bu\in C([0,T];H^1)\cap L^2(0,T;H^2)$, and  $\partial_t\bu\in L^2(0,T;L^2)$, and $\bu$ satisfies \eqref{Burgers_twist} in the sense of $L^2(0,T;L^2)$, with initial data satisfied in the sense of $C([0,T];H^1)$.
\end{definition}
\begin{theorem}[Local well-posedness: strong solutions]\label{thm_strong_local}
Given $\nu>0$, $T^*>0$, $\bu_0\in H^1(\nT^3)$, and  $\bbf\in L^2(0,T^*;\dot{L}^2)$, there exists a $T\in(0,T^*)$ and a unique strong solution to \eqref{Burgers_twist} on $[0,T]$.  Moreover, strong solutions depend continuously on the initial data in the sense of $C([0,T];H^1)\cap L^2(0,T;H^2)$, and the following energy balance holds for all $t\in[0,T]$:
\begin{align}\label{en_eq}
\|\bu(t)\|_{L^2}^2  + 2\nu\int_0^t\|\nabla\bu(s)\|_{L^2}^2\,ds
= \|\bu_0\|_{L^2}^2 + \int_0^t(\bbf(s),\bu(s))\,ds.
\end{align}
\end{theorem}
\begin{proof} Assume the hypotheses.
For each $N\in\nN$, consider the following Galerkin ODE system based on \eqref{Burgers_twist}
\begin{subequations}
\begin{align}
\label{rotBurgers_visc_Gal}
\pd{\bu^N}{t}&= -P_N((\nabla\times\bu^N)\times\bu^N) + \nu\triangle\bu^N + P_N\bbf,
\\\label{rotBurgers_visc_Gal_IC}
\bu^N(0)&=P_N(\bu_0).
\end{align}
\end{subequations}
The coefficients of $\bu^N$ clearly form an ODE determined by \eqref{rotBurgers_visc_Gal} with quadratic (and hence locally-Lipschitz) right-hand side.  Therefore, for each $N\in\nN$, there exists a $T_N>0$ and a unique solution $\bu^N\in C^1([-T_N,T_N],C^\infty)$ to \eqref{rotBurgers_visc_Gal} by the Picard-Lindel\"of Theorem\footnote{Technically, since we only have $P_N\bbf\in L^2(0,T;C^\infty)$, one needs to apply the Carath\'eodory Theorem for existence and uniquness, which holds since $P_N\bbf(\bx,\cdot)\in L^1_{\text{loc}}(0,T)$ for each $\bx\in\nT^3$.}.  Taking an inner product of \eqref{rotBurgers_visc_Gal} and integrating by parts, we obtain
\begin{align}\label{rotBurgersL2}
&\quad
    \frac12\frac{d}{dt}\|\bu^N\|_{L^2}^2 + \nu\|\nabla\bu^N\|_{L^2}^2
    \\&= \notag
    -(P_N((\nabla\times\bu^N)\times\bu^N),\bu^N) + (P_N\bbf,\bu^N),
    \\&=\notag
    (-(\nabla\times\bu^N)\times\bu^N,\bu^N)+(\bbf,\bu^N)=(\bbf,\bu^N)
    \\&\notag
    \leq 
    \sqrt{2}\|\bbf\|_{H^{-1}}\|\nabla\bu^N\|_{L^2}
    \leq
    \tfrac{1}{\nu}\|\bbf\|_{H^{-1}}^2+\tfrac{\nu}{2}\|\nabla\bu^N\|^2,
\end{align}
where we used \eqref{cross_prod_ortho}, \eqref{Hm1_bound}, and the symmetry of $P_N$.  Hence, rearranging and integrating, we obtain
\begin{align}\label{rotBurgersL2integrated}
    \|\bu^N(t)\|_{L^2}^2 + \nu\int_0^t\|\nabla\bu^N(s)\|_{L^2}^2\,ds
    &\leq 
    \|\bu^N(0)\|_{L^2}^2 + 
    \tfrac{2}{\nu}\int_0^t\|\bbf(s)\|_{H^{-1}}^2\,ds
    \\&\leq \notag
    \|\bu_0\|_{L^2}^2 + 
    \tfrac{2}{\nu}\int_0^{T^*}\|\bbf(s)\|_{H^{-1}}^2\,ds.
\end{align}
Thus, $\|\bu^N(t)\|_{L^2}$ is bounded independently of time (and $N$), so the time of existence and uniqueness of the Galerkin solution can be extended to arbitrarily.  Moreover, for any $T\in(0,T^*)$, we obtain:
\begin{align}\label{Gal_bdd_L2_Burgers_visc}
    \set{\bu^N}_{N=1}^\infty\text{ is bounded in }L^\infty(0,T;L^2)\cap L^2(0,T;H^1).
\end{align}
By the Banach-Alaoglu Theorem, there exists $\bu\in L^\infty(0,T;L^2)\cap L^2(0,T;H^1)$ and a subsequence of $\set{\bu^N}_{N\in\nN}$ (which we relabel as $\set{\bu^N}_{N\in\nN}$ if necessary) such that
\begin{subequations}
\begin{align}\label{L2weak_conv_visc_Burgers}
 \bu^N \overset{*}{\rightharpoonup}\bu &\text{ in } L^\infty(0,T;L^2),
 \\\label{L2H1weak_conv_visc_Burgers}
 \bu^N\rightharpoonup \bu &\text{ in } L^2(0,T;H^1).
\end{align}
\end{subequations}
Moreover, for any $\bvphi\in L^4(0,T;H^1)$, using H\"older's inequality, \eqref{GNS}, and $\|P_N\bbf\|_{L^2}\leq \|\bbf\|_{L^2}$, we obtain
\begin{align}\label{time_der_est}
&\quad
    \ip{\partial_t\bu^N}{\bvphi}
    \\&=\notag
    -\ip{(\nabla\times\bu^N)\times\bu^N}{P_N\bvphi} -\nu(\nabla\bu^N,\nabla\bvphi) + (P_N\bbf,\bvphi)
    \\&\leq\notag
    \|\nabla\bu^N\|_{L^2}\|\bu^N\|_{L^3}\|P_N\bvphi\|_{L^6} +\nu\|\nabla\bu^N\|_{L^2}\|\nabla\bvphi\|_{L^2}
    +\|\bbf\|_{L^2}\|\bvphi\|_{L^2}
    \\&\leq\notag
    C\|\bu^N\|_{H^1}^{3/2}\|\bu^N\|_{L^2}^{1/2}\|\bvphi\|_{H^1} 
    +\nu\|\nabla\bu^N\|_{L^2}\|\bvphi\|_{H^1}
    +\|\bbf\|_{L^2}\|\bvphi\|_{L^2}
    \\&\leq\notag
    C\|\bu^N\|_{H^1}^{2}+\|\bu^N\|_{L^2}^{2}\|\bvphi\|_{H^1}^4
    +\nu\|\nabla\bu^N\|_{L^2}\|\bvphi\|_{H^1}
    +\|\bbf\|_{L^2}\|\bvphi\|_{L^2},
\end{align}
from which it follows easily by integrating \eqref{Gal_bdd_L2_Burgers_visc} on $[0,T]$ that 
\begin{align}\label{Gal_dt_bdd_Burgers_visc}
    \set{\partial_t\bu^N}_{N=1}^\infty\text{ is bounded in }L^{4/3}(0,T;H^{-1}).
\end{align}
The Banach-Alaoglu Theorem (and a standard ``uniqueness of limits'' argument, see, e.g., \cite{Constantin_Foias_1988,Temam_1995_Fun_Anal,Robinson_2001}) implies $\partial_t\bu\in L^{4/3}(0,T;H^{-1})$, and moreover,
\begin{align}\label{weak_conv_dt_Burgers_visc}
 \partial_t\bu^N \overset{*}{\rightharpoonup}\partial_t\bu &\text{ in } L^{4/3}(0,T;H^{-1}).
\end{align}
Thus, thanks to \eqref{Gal_bdd_L2_Burgers_visc} and \eqref{Gal_dt_bdd_Burgers_visc} and the Aubin-Lions Compactness Theorem, we obtain a subsequence (again relabeled as $\set{\bu^N}_{N\in\nN}$) such that
\begin{align}\label{strong_conv_visc_Burgers}
 \bu^N\rightarrow \bu &\text{ (strongly) in } C([0,T];L^2).
\end{align}
\begin{remark}\label{remark_no_weak_Burgers_twist_sol}
It might seem as though one should be able to pass to the limit in this case to obtain a weak solution, since we have all of the ``ingredients'' we need in the corresponding 3D Navier-Stokes case.  However, as discussed above, since our solutions are not required to be divergence-free, we lack the anti-symmetry of the nonlinear term that the Navier-Stokes equations posses, which allows one to integrate by parts and move derivatives entirely onto the test functions, so that strong convergence in $L^\infty([0,T],L^2)$ is all that is needed.  However, this does not seem to be the case for \eqref{Burgers_twist} (similar issues were also encountered in \cite{Pooley_Robinson_2016} in the context of  Burgers equation \eqref{Burgers} in dimension $n\geq2$). Hence, we must look for higher-order estimates before passing to the limit.  
\end{remark}
We now establish existence and uniqueness of strong solutions.  For simplicity of presentation, we take $\bbf\equiv\mathbf{0}$, but the results hold for non-zero $\bbf\in L^2(0,T;\dot{L}^2)$ as well by using integration by parts to move derivatives off of $\bbf$. 
Let $\alpha\in\nZ_{\geq0}^3$ be a multi-index. Let us also denote $\bomega^N:=\nabla\times\bu^N$.  
From Leibniz's product rule, we obtain
\begin{align}\label{Burgers_Liebniz}
\partial_t\partial_\alpha \bu^N
&=
-\sum_{\beta\leq\alpha}\binom{\alpha}{\beta}P_N((\partial_{\alpha-\beta}\bomega^N)\times\partial_{\beta}\bu^N) + \nu\triangle\partial_\alpha\bu^N.
\end{align}
We take the inner-product of \eqref{Burgers_Liebniz} with $\partial_\alpha\bu^N$ to obtain
\begin{align}
 \frac12\frac{d}{dt}\|\partial_\alpha \bu^N\|_{L^2}^2 + \nu\|\nabla\partial_\alpha\bu^N\|_{L^2}^2
 &= \notag
 -\sum_{\beta\leq\alpha}\binom{\alpha}{\beta}((\partial_{\alpha-\beta}\bomega^N)\times\partial_{\beta}\bu^N,\partial_\alpha \bu^N)
  \\&= \label{Burgers_alpha}
 -\sum_{\beta<\alpha}\binom{\alpha}{\beta}((\partial_{\alpha-\beta}\bomega^N)\times\partial_{\beta}\bu^N,\partial_\alpha \bu^N).
\end{align}
The second equality is due to the fact that $(\bomega^N\times\partial_{\alpha}\bu^N,\partial_\alpha \bu^N)=0$, which follows immediately from the cross-product orthogonality  \eqref{cross_prod_ortho}.  
For the case when $|\alpha|=1$ in \eqref{Burgers_alpha}, using \eqref{GNS_full}, \eqref{GNS3}, and the trivial fact that $|\nabla\times\bu|\leq |\nabla\bu|$, we estimate
\begin{align}
&\quad \label{H1_bound_twist}
 \frac12\frac{d}{dt}\|\partial_\alpha \bu^N\|_{L^2}^2
 + \nu\|\nabla\partial_\alpha\bu^N\|_{L^2}^2
\\ &= \notag
-((\partial_{\alpha}\bomega^N)\times\bu^N,\partial_\alpha \bu^N)
 \\&\leq \notag
 \|\partial_{\alpha}\bomega^N\|_{L^2}\|\bu^N\|_{L^6}\|\partial_\alpha\bu^N\|_{L^3}
   \\&\leq \notag
C\|\nabla\partial_\alpha\bu^N\|_{L^2}^{3/2}\|\bu^N\|_{H^1}\|\partial_\alpha\bu^N\|_{L^2}^{1/2}
  \\&\leq \notag
 \frac{\nu}{2}\|\nabla\partial_\alpha\bu^N\|_{L^2}^2+\frac{C}{\nu^3}
 \|\bu^N\|_{H^1}^4\|\partial_\alpha\bu^N\|_{L^2}^2.
\end{align}
Rearranging and then summing over all $\alpha$ with $|\alpha|=1$, and adding\footnote{Since we do not have the Poincar\'e inequality, this step is used to obtain the same quantity, namely $\|\bu^N\|_{H^1}$ on both sides of the inequality before applying Gr\"onwall's inequality.} the result to \eqref{rotBurgersL2} yields
\begin{align}\label{rotBurgersH1}
\frac{d}{dt}\|\bu^N\|_{H^1}^2
+ \nu\|\nabla\bu^N\|_{L^2}^2
 + \nu\|\nabla\bu^N\|_{H^1}^2\leq \frac{C}{\nu^3}
 \|\bu^N\|_{H^1}^6.
\end{align}
Dropping the viscous terms on the left, we obtain for arbitrary $\epsilon>0$,
\begin{align}\label{rotBurgersH1_drop}
\frac{d}{dt}(\|\bu^N\|_{H^1}^2+\epsilon)
 \leq \frac{C}{\nu^3}(\|\bu^N\|_{H^1}^2+\epsilon)^3.
\end{align}
Applying Gr\"onwall's inequality, it follows from the monotonicity of $x\mapsto(x^{-2}-K)^{-1/2}$ (for $0\leq K<x^{-2}$) and \eqref{rotBurgers_visc_Gal_IC} that
\begin{align}\label{Gronwall_with_eps}
\|\bu^N(t)\|_{H^1}^2 
&\leq 
\pnt{(\|\bu^N(0)\|_{H^1}^2+\epsilon)^{-2}-2C\nu^{-3} t}^{-1/2} - \epsilon
\\&\leq \notag 
\pnt{(\|\bu_0\|_{H^1}^2+\epsilon)^{-2}-2C\nu^{-3} t}^{-1/2} - \epsilon.
\end{align}
Note that the right-hand side is positive and monotonically increasing and continuous in $\epsilon>0$ over its domain of definition; 
hence we may take the limit as $\epsilon\rightarrow0^+$.  If $\|\bu_0\|_{H^1}=0$, the right-hand side of \eqref{Gronwall_with_eps} approaches zero for all $t\geq0$, hence $\bu^N\equiv\mathbf{0}$, $T^*>0$ can be chosen arbitrarily.  Otherwise, we obtain
\begin{align}\label{Gronwall_short_bound}
\|\bu^N(t)\|_{H^1}^2 
&\leq 
\pnt{\|\bu_0\|_{H^1}^{-4}-2C\nu^{-3} t}^{-1/2}.
\end{align}
Letting, e.g.,  \[T^*:=\nu^{3}/(4C\|\bu_0\|_{H^1}^{4}),\] we obtain a uniform bound on $\bu^N$ in $L^\infty(0,T^*;H^1)$.  
Using this after integrating \eqref{rotBurgersH1} on $[0,T^*]$, we obtain
\begin{align}\label{Gal_bdd_H1_Burgers_visc}
    \set{\bu^N}_{N=1}^\infty\text{ is bounded in }L^\infty(0,T^*;H^1)\cap L^2(0,T^*;H^2).
\end{align}
and hence by the Banach-Alaoglu Theorem,
\begin{subequations}\label{H1weak_conv_visc_Burgers}
   \begin{align}
 \bu^N \overset{*}{\rightharpoonup}\bu &\text{ in } L^\infty(0,T^*;H^1),
 \\
 \bu^N\rightharpoonup \bu &\text{ in } L^2(0,T^*;H^2).
\end{align} 
\end{subequations}
In order to show convergence to a strong solution of \eqref{Burgers_twist}, we note that, for any $\bvphi\in L^4(0,T^*;H^1)$, one may write
\begin{align}\label{nonlinear_conv_Burgers_visc}
    &\quad
    \ip{P_N(\bomega^N\times\bu^N)}{\bvphi} - \ip{\bomega\times\bu}{\bvphi}
    \\\notag&=
    \ip{\bomega^N\times(\bu^N-\bu)}{P_N\bvphi}
    +\ip{\bomega^N\times\bu}{P_N\bvphi-\bvphi}
    \\&\notag\quad
    +\ip{(\bomega^N-\bomega)\times\bu}{\bvphi}
    := I + II + III.\notag
\end{align}
For the first and second terms, we estimate
\begin{align}\label{nonlinear_conv_Burgers_visc_I_III}
    &\quad
    |I|+|II| = |\ip{\bomega^N\times(\bu^N-\bu)}{P_N\bvphi}|+|\ip{\bomega^N\times\bu}{P_N\bvphi-\bvphi}|
    \\&\leq\notag
    \|\bomega^N\|_{L^2}\|\bu^N-\bu\|_{L^3}\|P_N\bvphi\|_{L^6}
    +\|\bomega^N\|_{L^2}\|\bu\|_{L^3}\|P_N\bvphi-\bvphi\|_{L^6}
    \\\notag&\leq
    C\|\bu^N\|_{H^1}\|\bu^N-\bu\|_{L^2}^{1/2}
    \|\bu^N-\bu\|_{H^1}^{1/2}\|P_N\bvphi\|_{H^1}
    \\&\notag\quad+C\|\bu^N\|_{H^1}\|\bu\|_{L^2}^{1/2}\|\bu\|_{H^1}^{1/2}\|P_N\bvphi-\bvphi\|_{H^1}
    \\\notag&\leq
    C\|\bu^N\|_{H^1}\|\bu^N-\bu\|_{L^2}^{1/2}
    (\|\bu^N\|_{H^1}+\|\bu\|_{H^1})^{1/2}\|\bvphi\|_{H^1}
    \\&\notag\quad+C\|\bu^N\|_{H^1}\|\bu\|_{L^2}^{1/2}\|\bu\|_{H^1}^{1/2}\|P_N\bvphi-\bvphi\|_{H^1}.
\end{align}
Due to the bounds in \eqref{Gal_bdd_H1_Burgers_visc}, the strong convergence in \eqref{strong_conv_visc_Burgers}, and the trivial strong convergence of $P_N\bvphi$ to $\bvphi$ in $L^2$, $I$ and $II$ converge to zero in $L^\infty(0,T^*)$.  We view the third term $III$ as a type of product of $\bu$ and $\bvphi$ acting on $\bomega^N-\bomega$.  Denoting $(\bu\otimes\bvphi)_{i,j}=u_i\vphi_j$, we notice that in 3D,
\begin{align*}
    \|\bu\otimes\bvphi\|_{L^2}\leq \|\bu\|_{L^\infty}\|\bvphi\|_{L^2}\leq C\|\bu\|_{H^1}^{1/2}\|\bu\|_{H^2}^{1/2}\|\bvphi\|_{L^2}.
\end{align*}
Hence, \eqref{Gal_bdd_H1_Burgers_visc} implies that the product of any components of $\bu$ and $\bvphi$ satisfy $u_i \vphi_j\in L^4(0,T^*;L^2)\subset L^2(0,T^*;L^2)$ for each $i,j\in\set{1,2,3}$.  Thus, (after passing to a subsequence and relabeling as needed) the remaining term of \eqref{nonlinear_conv_Burgers_visc}; namely $III:=\ip{(\bomega^N-\bomega)\times\bu}{\bvphi}$ converges weakly in $L^2(0,T^*)$ due to the convergence in \eqref{L2weak_conv_visc_Burgers}.  Finally, the weak convergence in \eqref{L2H1weak_conv_visc_Burgers} implies that $\triangle\bu^N\rightharpoonup\triangle\bu$ weakly in $L^2(0,T^*;H^{-1})$.  Hence, passing to a weak-$*$ limit of \eqref{rotBurgers_visc_Gal} in $L^{4/3}(0,T^*;H^{-1})$, we see that $\bu$ satisfies \eqref{Burgers_twist} in the sense of $L^{4/3}(0,T;H^{-1})$.  However, for uniqueness, we need more regularity on the time derivative.  We re-estimate \eqref{time_der_est}, this time\footnote{We could have assumed $\bvphi\in L^2(0,T^*;L^2)$ from the start, but this is not needed for existence, so the estimate is carried out in this weaker form for the sake of generality.} assuming only $\bvphi\in L^2(0,T^*;L^2)$, to obtain
\begin{align}\label{time_der_est2}
    \ip{\partial_t\bu^N}{\bvphi}
    &=
    -\ip{(\nabla\times\bu^N)\times\bu^N}{P_N\bvphi} +\nu(\triangle\bu^N,\bvphi)
    \\&\leq\notag
    \|\nabla\times\bu^N\|_{L^3}\|\bu^N\|_{L^6}\|P_N\bvphi\|_{L^2} +\nu\|\triangle\bu^N\|_{L^2}\|\bvphi\|_{L^2}
    \\&\leq\notag
    C\|\bu^N\|_{H^1}^{3/2}\|\bu^N\|_{H^2}^{1/2}\|\bvphi
    \|_{L^2} +\nu\|\triangle\bu^N\|_{L^2}\|\bvphi\|_{L^2}.
\end{align}
Hence, from \eqref{H1weak_conv_visc_Burgers}, it follows that 
\begin{align}\label{Gal_dt_bdd_Burgers_visc_L2}
    \set{\partial_t\bu^N}_{N=1}^\infty\text{ is bounded in }L^{2}(0,T^*;L^2),
\end{align}
and thus 
\begin{align}\label{dudt_bdd_L2L2}
    \partial_t\bu\in L^2(0,T^*;L^2),
\end{align} and \eqref{Burgers_twist} is actually satisfied in the stronger sense of $L^2(0,T^*;L^2)$.  Moreover, the Aubin-Lions Theorem together with  \eqref{Gal_bdd_H1_Burgers_visc} and \eqref{Gal_dt_bdd_Burgers_visc_L2} imply that
\begin{align}
\bu\in C([0,T^*];H^1)\cap L^2(0,T^*;H^2).
\end{align}
Integrating \eqref{rotBurgers_visc_Gal} on $[0,t]$ for $t<T^*$ and passing to the limit as $N\rightarrow\infty$ shows that $\bu(0)=\bu_0$.
Hence, we have established the existence of strong solutions as in Definition \ref{def_strong_sol} with $T=T^*$.
To show that strong solutions are unique, suppose that $\bu,\bv$ are strong solutions with initial data $\bu_0,\bv_0\in H^1$ respectively.  Denoting $\bw =  \bu-\bv$,
\begin{align}\label{Burgers_twist_visc_diff}
 \bw_t + (\nabla\times\bu)\times\bw + (\nabla\times\bw)\times\bv &= \nu\triangle\bw.
\end{align}
We take the inner-product of \eqref{Burgers_twist_visc_diff} with $\bw\in C(0,T^*;H^1)\cap L^2(0,T^*;H^2)$.  Using 
\eqref{cross_prod_ortho} and the Lions-Magenes Lemma (which is justified thanks to \eqref{dudt_bdd_L2L2}), we obtain a.e. on $[0,T^*]$, 
\begin{align*}
&\quad
  \frac12\frac{d}{dt}\|\bw\|_{L^2}^2 + \nu\|\nabla\bw\|_{L^2}^2 
  =  
 - ((\nabla\times\bw)\times\bv,\bw)
  \\&\leq
  \|\nabla\times\bw\|_{L^2}\|\bv\|_{L^6}\|\bw\|_{L^3}
  \leq
  C\|\nabla\bw\|_{L^2}\|\bv\|_{H^1}\|\bw\|_{L^2}^{1/2}\|\bw\|_{H^1}^{1/2}
  \\&\leq
   C\|\nabla\bw\|_{L^2}\|\bv\|_{H^1}\|\bw\|_{L^2}
  +C\|\nabla\bw\|_{L^2}^{3/2}\|\bv\|_{H^1}\|\bw\|_{L^2}^{1/2}
    \\&\leq
  \tfrac{\nu}{2}\|\nabla\bw\|_{L^2}^2 + C_\nu\|\bv\|_{H^1}^{2}\|\bw\|_{L^2}^{2} + C_\nu\|\bv\|_{H^1}^{4}\|\bw\|_{L^2}^{2}.
\end{align*}
using \eqref{GNS} and the simple inequality $\|\bw\|_{H^1}^{1/2}\leq \|\bw\|_{L^2}^{1/2}+\|\nabla\bw\|_{L^2}^{1/2}$.  
Thus,
\begin{align}\label{Burgers_twist_visc_diff_L2}
 \frac{d}{dt}\|\bw\|_{L^2}^2 + \nu\|\nabla\bw\|_{L^2}^2
 \leq
   C_\nu(\|\bv\|_{H^1}^{2}+\|\bv\|_{H^1}^{4})\|\bw\|_{L^2}^{2}.
\end{align}
Since $\bv\in L^\infty(0,T;H^1)$, Gr\"onwall's inequality implies that
\begin{align}\label{Burgers_twist_unique_Gronwall_L2}
 \|\bw(t)\|_{L^2}^2 + \nu\int_0^t\|\nabla\bw(s)\|_{L^2}^2e^{C_\nu\int_\tau^t(\|\bv(s)\|_{H^1}^{2}+\|\bv(s)\|_{H^1}^{4})\,ds}\,d\tau
 \leq
 \|\bw(0)\|_{L^2}^2.
\end{align}
Setting $\bu_0=\bv_0$, it follows that strong solutions are unique.\footnote{If one could establish the existence of weak solutions, by taking $\bv$ to be a strong solution, \eqref{Burgers_twist_unique_Gronwall_L2} would be enough to show that strong solutions are unique among the class of weak solutions.}  Continuous dependence on initial data in the sense of $C([0,T^*];L^2)\cap L^2(0,T^*;H^1)$ also follows from \eqref{Burgers_twist_unique_Gronwall_L2}.  To show higher-order continuous dependence on initial data, we take the action of \eqref{Burgers_twist_visc_diff} on $-\triangle\bw\in L^\infty(0,T^*;H^{-1})\cap L^2(0,T^*;L^2)$ and again use the Lions-Magenes Lemma (which is justified, since $\nabla\bu_t\in L^2(0,T^*;H^{-1})$ and $\nabla\bu\in L^2(0,T^*;H^{1})$) to obtain a.e. on $[0,T^*]$, 
\begin{align*}
&\quad
  \frac12\frac{d}{dt}\|\nabla\bw\|_{L^2}^2 + \nu\|\triangle\bw\|_{L^2}^2 
  =  
 ((\nabla\times\bu)\times\bw,\triangle\bw) + ((\nabla\times\bw)\times\bv,\triangle\bw)
  \\&\leq
  \|\nabla\bu\|_{L^6}\|\bw\|_{L^3}\|\triangle\bw\|_{L^2}
  +
  \|\nabla\bw\|_{L^2}\|\bv\|_{L^\infty}\|\triangle\bw\|_{L^2}
  \\&\leq
  C\|\bu\|_{H^2}\|\bw\|_{H^1}\|\triangle\bw\|_{L^2}
  +
  C\|\nabla\bw\|_{L^2}\|\bv\|_{H^2}\|\triangle\bw\|_{L^2}
  \\&\leq
  C_\nu(\|\bu\|_{H^2}^2 + \|\bv\|_{H^2}^2)\|\bw\|_{H^1}^2
  +
    \frac{\nu}{2}\|\triangle\bw\|_{L^2}^2.
\end{align*}
Rearranging this and adding the result to \eqref{Burgers_twist_visc_diff_L2} yields
\begin{align}\label{Burgers_twist_visc_diff_H1}
 \frac{d}{dt}\|\bw\|_{H^1}^2 + \nu\|\nabla\bw\|_{H^1}^2
 \leq
   C_\nu(\|\bu\|_{H^2}^2 + \|\bv\|_{H^2}^2+\|\bv\|_{H^1}^{4})\|\bw\|_{L^2}^{2}.
\end{align}
Since $\bu,\bv\in C([0,T^*];H^1)\cap L^2(0,T^*;H^2)$ and $\bu_0,\bv_0\in H^1$, Gr\"onwall's inequality yields continuous dependence on the initial data in the sense of $C([0,T^*];H^1)\cap L^2(0,T^*;H^2)$ of strong solutions, as desired.  Finally, taking the inner-product of  \eqref{rotBurgers_visc_Gal} with $\bu^N$, integrating in time, we obtain 
\begin{align}\label{en_eq_Gal}
\|\bu^N(t)\|_{L^2}^2  + 2\nu\int_0^t\|\nabla\bu^N(s)\|_{L^2}^2\,ds
= \|P_N\bu_0\|_{L^2}^2 + \int_0^t(\bbf(s),\bu^N(s))\,ds.
\end{align}
and noting that \eqref{time_der_est2} and \eqref{Gal_bdd_H1_Burgers_visc} together with the Aubin-Lions Theorem imply that (passing to a subsequence if necessary)
\begin{align*}
    \bu^N\rightarrow\bu \text{ in }L^2(0,T^*;H^1).
\end{align*}
This strong convergence is enough to pass to the limit as $N\rightarrow\infty$ in \eqref{en_eq_Gal} for a.e. $t\in[0,T^*]$, and $\bu\in C([0,T^*];H^1)$ implies \eqref{en_eq} for all $t\in[0,T^*]$.
\end{proof}
 \subsection{Higher-order estimates}\label{sec_higher_order}
In this section, we show that strong solutions enjoy some high-order regularity properties.  These will be necessary later for establishing global well-posedness.
\begin{theorem}\label{thm_Hm_bounds}
 Given $\bu_0\in H^1$, $m\in\nN$, $m\geq1$, and $\bbf\in L^2(0,T^*;\dot{H}^{m-1})$, let $\bu\in C(0,T^*;H^1)\cap L^2(0,T^*;H^2)$ be the unique strong solution given by Theorem \ref{thm_strong_local}, for some $T^*>0$.  Then, for any $\epsilon\in(0,T^*)$, 
\begin{align}
    \bu\in L^\infty(\epsilon,T^*;H^m)\cap L^2(\epsilon,T^*;H^{m+1}).
\end{align}
and moreover, in the case that $m\geq 5$, if we assume in addition that $\pd{\bbf}{t}\in L^2(0,T; \dot{H}^2)$, then
\begin{align}
    \bu\in H^1(\epsilon,T^*;H^4)\cap H^2(\epsilon,T^*;H^{2});
\end{align}
in particular, $\bu$ is a classical solution on $[\epsilon,T^*]$; that is, 
\begin{align}
    \bu\in C^0([\epsilon,T^*];C^2)\cap C^1([\epsilon,T^*];C^0).
\end{align}
\end{theorem}
\begin{proof}
We again take $\bbf\equiv0$ for simplicity of presentation, but the general results hold by using a suitable integration by parts for general $\bbf$ satisfying the hypotheses.  Let $\epsilon\in(0,T^*)$ be given. 
We rewrite \eqref{Burgers_alpha} in the form
\begin{align}\label{alpha3est}
&\quad
 \frac12\frac{d}{dt}\|\partial_\alpha \bu^N\|_{L^2}^2
 + \nu\|\nabla\partial_\alpha\bu^N\|_{L^2}^2
 \\&\notag=
-((\partial_\alpha\bomega^N)\times\bu^N,\partial_\alpha \bu^N)
\\&\notag\qquad
-\sum_{0<\beta<\alpha}\binom{\alpha}{\beta}
((\partial_{\alpha-\beta}\bomega^N)\times\partial_{\beta}\bu^N,\partial_\alpha \bu^N).
\end{align}
The cases $|\alpha|\leq1$ were handled in \eqref{H1_bound_twist}.  In the case $|\alpha|=2$, we estimate
\begin{align*}
&\quad
 \frac12\frac{d}{dt}\|\partial_\alpha \bu^N\|_{L^2}^2
 + \nu\|\nabla\partial_\alpha\bu^N\|_{L^2}^2
 \\&\leq
 \|\partial_\alpha\bomega^N\|_{L^2} \|\bu^N\|_{L^\infty}\|\partial_\alpha\bu^N\|_{L^2}
 \\&\qquad+
 \sum_{0<\beta<\alpha}\binom{\alpha}{\beta}
 \|\partial_{\alpha-\beta}\bomega^N\|_{L^2} \|\partial_\beta\bu^N\|_{L^3}\|\partial_\alpha\bu^N\|_{L^6}
  \\&\leq
 C\|\nabla\partial_\alpha\bu^N\|_{L^2}\|\bu^N\|_{H^1}^{1/2} 
 \|\bu^N\|_{H^2}^{3/2}
 +
 C\|\bu^N\|_{H^2}^{3/2} 
 \|\bu^N\|_{H^1}^{1/2}\|\nabla\partial_\alpha\bu^N\|_{L^2}
   \\&\leq
   \frac{\nu}{2}\|\nabla\partial_\alpha\bu^N\|_{L^2}^2
   +
 C_\nu \|\bu^N\|_{H^1}\|\bu^N\|_{H^2}^3,
\end{align*}
where we used \eqref{GNS}.  Rearranging, we obtain
\begin{align}
 \frac{d}{dt}\|\partial_\alpha \bu^N\|_{L^2}^2
 + \nu\|\nabla\partial_\alpha\bu^N\|_{L^2}^2
   \leq
 C_\nu \|\bu^N\|_{H^1}\|\bu^N\|_{H^2}^3.
\end{align}
Summing over all $\alpha$ with $|\alpha|=2$, and then adding the result to \eqref{rotBurgersH1}, and dropping some lower-order viscous terms on the left, we obtain
\begin{align}\label{Gal_H2_twist}
 \frac{d}{dt}\|\bu^N\|_{H^2}^2
 + \nu\|\nabla\bu^N\|_{H^2}^2
   \leq
 C_\nu (\|\bu^N\|_{H^1}\|\bu^N\|_{H^2})\|\bu^N\|_{H^2}^2
 + 
 C_\nu\|\bu^N\|_{H^1}^6.
\end{align}
By the bounds in \eqref{Gal_bdd_H1_Burgers_visc}, the second term on the right is bounded, and for the first term, the coefficient $\|\bu^N\|_{H^1}\|\bu^N\|_{H^2}$ is integrable, hence it is tempting to apply Gr\"onwall's inequality on $[0,T^*]$, but this would require $\bu_0\in H^2$. 
However, due to \eqref{Gal_bdd_H1_Burgers_visc}, the set $\bigcup_{N\in\nN}\set{t\in(0,T^*):\|\bu^N_t(t)\|_{H^2}=\infty}$ has measure zero.  Hence, for any $\delta>0$, there exists a point $t_*\in(0,\delta)$ such that $\|\bu^N(t_*)\|_{H^2}<\infty$.   
Applying Gr\"onwall's inequality to \eqref{Gal_H2_twist} on the interval $[t_*,T^*]$ and using \eqref{H1weak_conv_visc_Burgers} implies 
\begin{align}\label{statement_H2_bdd_twist}
    \set{\bu^N}_{N=1}^\infty\text{ is bounded in }L^\infty(\epsilon,T^*;H^2)\cap L^2(\epsilon,T^*;H^3),
\end{align}
since we can choose $t_*<\epsilon$.  
Returning to \eqref{alpha3est}, we estimate  for $|\alpha|\geq3$,
\begin{align*}
|((\partial_\alpha\bomega^N)\times\bu^N,\partial_\alpha \bu^N)|
&\leq
\|\partial_\alpha\bomega^N\|_{L^2}\|\bu^N\|_{L^\infty}\|\partial_\alpha \bu^N\|_{L^2}
\\&\leq
C\|\nabla\partial_\alpha\bu^N\|_{L^2}\|\bu^N\|_{H^2}\|\partial_\alpha\bu^N\|_{L^2}
\\&\leq
\frac{\nu}{4}\|\nabla\partial_\alpha\bu^N\|_{L^2}^{2}
+
C_\nu\|\bu^N\|_{H^2}^2\|\partial_\alpha\bu^N\|_{L^2}^2.
\end{align*}
We estimate the second part of \eqref{alpha3est} (using $1\leq|\alpha-\beta|,|\beta|\leq m-1$) by
\begin{align}
&\quad
    \sum_{0<\beta<\alpha}\binom{\alpha}{\beta}
|((\partial_{\alpha-\beta}\bomega^N)\times\partial_{\beta}\bu^N,\partial_\alpha \bu^N)|
\\\notag&\leq
   \sum_{0<\beta<\alpha}\binom{\alpha}{\beta}
\|\partial_{\alpha-\beta}\bomega^N\|_{L^2}\|\partial_{\beta}\bu^N\|_{L^3}\|\partial_\alpha \bu^N\|_{L^6}
\\\notag&\leq
   C\sum_{0<\beta<\alpha}\binom{\alpha}{\beta}
\|\partial_{\alpha-\beta}\nabla\bu^N\|_{L^2}\|\partial_{\beta}\bu^N\|_{L^2}^{1/2}\|\nabla\partial_{\beta}\bu^N\|_{L^2}^{1/2}\|\nabla\partial_\alpha \bu^N\|_{L^2}
\\\notag&\leq
C\|\bu^N\|_{H^{m-1}}^{1/2}\|\bu^N\|_{H^m}^{3/2}\|\nabla\partial_\alpha \bu^N\|_{L^2}
\\\notag&\leq
\frac{\nu}{4}\|\nabla\partial_\alpha \bu^N\|_{L^2}^2
+C_\nu\|\bu^N\|_{H^{m-1}}\|\bu^N\|_{H^m}^{3}.
\end{align}
Thus, rearranging and then summing all the above estimates 
over $\alpha$, $|\alpha|\leq m$, we obtain for $m\geq3$,
\begin{align}
\frac{d}{dt}\|\bu^N\|_{H^m}^2
+\nu\|\nabla\bu^N\|_{H^{m}}^2
\leq C_\nu(\|\bu^N\|_{H^2} + \|\bu^N\|_{H^{m-1}})\|\bu^N\|_{H^m}^3.
\end{align}
Following similar steps to the derivation of statement \eqref{statement_H2_bdd_twist}, and then using induction on $m\in\nN$ (and in particular, the integrability of the factor $(\|\bu^N\|_{H^2} + \|\bu^N\|_{H^{m-1}})\|\bu^N\|_{H^m}$ on $(\epsilon,T^*)$), we obtain for and $\epsilon\in(0,T^*)$,
\begin{align}\label{Gal_bdd_Hm_Burgers_visc}
    \set{\bu^N}_{N=1}^\infty\text{ is bounded in }L^\infty(\epsilon,T^*;H^{m})\cap L^2(\epsilon,T^*;H^{m+1}).
\end{align}
Next, we follow the approach of \cite{Pooley_Robinson_2016} for establishing classical solutions. Recall that $H^m(\nT^n)$ is an algebra whenever $m>n/2$.  Thus, \eqref{rotBurgers_visc_Gal} yields
\begin{align}\label{twist_ut_H4}
\|\bu_t^N\|_{H^4}
\leq 
\|\nabla\times\bu\|_{H^4}\|\bu\|_{H^4} + \nu\|\bu\|_{H^6}.
\end{align}
Squaring and integrating \eqref{twist_ut_H4} on $(\epsilon,T^*)$, and letting $m=5$ in \eqref{Gal_bdd_Hm_Burgers_visc}, we obtain for any $\epsilon\in(0,T^*)$,
\begin{align}
    \set{\bu^N_t}_{N=1}^\infty\text{ is bounded in }L^2(\epsilon,T^*;H^4).
\end{align}
Next, taking the derivative in $t$ with \eqref{rotBurgers_visc_Gal}, we obtain
\begin{align}\label{Galerkin_tt}
\bu^N_{tt}&= -P_N((\nabla\times\bu^N_t)\times\bu^N)
-P_N((\nabla\times\bu^N)\times\bu^N_t)+ \nu\triangle\bu^N_t.
\end{align}
Hence, since $n\in\set{2,3}$, so that $H^2$ is an algebra, we obtain
\begin{align}\label{twist_utt_H2}
\|\bu_{tt}^N\|_{H^2}
\leq 
\|\nabla\times\bu^N_t\|_{H^2}\|\bu^N\|_{H^2} 
+\|\nabla\times\bu^N\|_{H^2}\|\bu^N_t\|_{H^2} 
+ \nu\|\bu_t\|_{H^4}.
\end{align}
Letting $m=3$ in \eqref{Gal_bdd_Hm_Burgers_visc} and using the above bounds, we obtain for any $\epsilon\in(0,T^*)$,
\begin{align}
    \set{\bu^N_{tt}}_{N=1}^\infty\text{ is bounded in }L^2(\epsilon,T^*;H^2).
\end{align}
Hence, 
$\bu\in 
H^1(\epsilon,T^*;H^4)\cap H^2(\epsilon,T^*;H^2)
\subset
C^0([\epsilon,T^*];C^2)\cap C^1([\epsilon,T^*];C^0)
$; 
i.e., $\bu$ is a classical solution on $[\epsilon,T^*]$ for any $\epsilon\in(0,T^*)$.  
\end{proof}
\subsection{A Maximum Principle}
In this subsection, we use the existence of classical solutions to establish a maximum principle\footnote{Note that the maximum principle does not apply at the Galerkin level, and hence, we could not apply it before establishing the existence of classical solutions.}, and use this in turn to prove global well-posedness.  
In the absence of a forcing term, i.e., $\bbf\equiv\mathbf{0}$, one can use a nearly identical approach to \cite{Pooley_Robinson_2016} (in fact, it is slightly easier in our case, since the nonlinear term vanishes pointwise), which in turn follows ideas from \cite{Kiselev_Ladyzhenskaya_1957}\footnote{although see the discussion in \cite{Pooley_Robinson_2016} of a possibly fatal flaw in \cite{Kiselev_Ladyzhenskaya_1957}.}.  However, for a non-zero forcing term, a different argument seems necessary, which we provide below.  
\begin{lemma}[Maximum principle]\label{lemma_maximum_principle_viscous}
Let $a,b\in\nR$ with $a>b$ be given, and also let\footnote{This level of regularity on $\bbf$ is only to ensure the existence of a classical solution.  This part of the proof only requires that $\bbf\in L^1(a,b;L^\infty)$.} $\bbf\in L^2(a,b;\dot{H}^4)$, and $\bbf_t\in L^2(a,b;\dot{H}^2)$.  Suppose $\bu$ is classical solution to \eqref{Burgers_twist} on $[a,b]$; i.e., $\bu$ is a strong solution that also satisfies $\bu\in C^1([a,b],C^0(\nT^3))\cap C^0([a,b],C^2(\nT^3))$.  Then 
\begin{align}\label{sup_max_twist}
    \sup_{t\in[a,b]}\|\bu(t)\|_{L^\infty}\leq \|\bu(a)\|_{L^\infty} + \int_a^t\|\bbf(s)\|_{L^\infty}\,ds.
\end{align}
\end{lemma}
\begin{proof} 
Assume the hypotheses, and let $\epsilon>0$ be given.  Denote
$w=w(\bx,t):=\sqrt{\epsilon+|\bu(\bx,t)|^2}$.  Clearly, $0\leq|\bu|<w$, and $w\in C^1([a,b],C^0)\cap C^0([a,b],C^2)$.  Note that $|\bu|^2=w^2-\epsilon$, so that $\triangle|\bu|^2=\triangle w^2$.  Moreover, note that $\nabla w= \frac{\bu}{w}\cdot\nabla u$, so that $|\nabla w|\leq\frac{|\bu|}{w}|\nabla\bu|\leq|\nabla\bu|$. Hence, using the well-known identity $\bv\cdot\triangle\bv = \frac12\triangle|\bv|^2-|\nabla\bv|^2$ twice, we estimate 
\begin{align}\notag
\bu\cdot\triangle\bu 
&= \tfrac12\triangle|\bu|^2-|\nabla\bu|^2
= \tfrac12\triangle w^2-|\nabla\bu|^2
=w\triangle w+|\nabla w|^2-|\nabla\bu|^2
\\&\leq \notag
w\triangle w.
\end{align}
Note that $\bbf\in L^2(a,b;\dot{H}^4)\subset L^1(a,b;L^\infty)$.  We now estimate for $t\in[a,b]$,
\begin{align}\label{w_estimate}
w_t 
&= 
\frac{\bu}{w}\cdot\bu_t
=
\frac{\bu}{w}\cdot(-\bomega\times\bu + \nu\triangle\bu + \bbf)
=
\frac{\nu}{w}\bu\cdot\triangle\bu + \frac{\bu}{w}\cdot\bbf
\\&\leq \notag
\nu \triangle w + \|\bbf\|_{L^\infty}.
\end{align}
Denoting $v=w-\int_a^t\|\bbf(s)\|_{L^\infty}\,ds$, inequality \eqref{w_estimate} yields $v_t-\nu\triangle v\leq0$, so by the usual parabolic maximum principle, $\|v(t)\|_{L^\infty}\leq \|v(a)\|_{L^\infty}$, from which \eqref{sup_max_twist} easily follows by sending $\epsilon\rightarrow0$.
\end{proof}
\subsection{Global Well-Posedness}
Thanks to the maximum principle, we can now prove global well-posedness in the viscous case under a slightly more restrictive assumption on $\bbf$.
\begin{theorem}\label{thm_Burgers_twist_GWP}
Given $\nu>0$, $\bu_0\in H^1(\nT^3)$, an arbitrary $T\in(0,\infty)$, 
and $\bbf\in L^2(0,T;\dot{H}^4)$ with $\bbf_t\in L^2(0,T;\dot{H}^2)$, the unique strong solution to \eqref{Burgers_twist}, guaranteed by Theorem \ref{thm_strong_local}, can be extended to a solution $\bu \in C^1((0,T],C^0(\nT^3))\cap C^0((0,T],C^2(\nT^3))$.
\end{theorem}
\begin{proof}
 Let $\bu$ be a strong solution defined on $[0,T^{**})$, where $T^{**}>0$ is its maximum time interval of existence and uniqueness.  By Theorems \ref{thm_strong_local} and \ref{thm_Burgers_twist_GWP}, $\bu$ is a classical solution for positive times; i.e., for any $\epsilon\in(0,T^{**})$, $\bu\in C^1([\epsilon,T^{**}),C^0)\cap C^0([\epsilon,T^{**}),C^2)$.  By Lemma \ref{lemma_maximum_principle_viscous}, 
 \begin{align}
    \sup_{t\in[\epsilon,T^{**}]}\|\bu(t)\|_{L^\infty}\leq \|\bu(\epsilon)\|_{L^\infty} + \int_\epsilon^{T^{**}}\|\bbf(s)\|_{L^\infty}\,ds.
\end{align}
Taking the inner-product of $\eqref{Burgers_twist}$ with $-\triangle\bu$, we obtain after integration by parts
\begin{align}
 \label{H1_bound_twist_global}
\frac12\frac{d}{dt}\|\nabla \bu\|_{L^2}^2
 + \nu\|\triangle\bu\|_{L^2}^2 
&= 
((\nabla\times\bu)\times\bu,\triangle \bu)
- (\bbf,\triangle\bu)
\\&\leq\notag
 \|\nabla\bu\|_{L^2}\|\bu\|_{L^\infty}\|\triangle\bu\|_{L^2}
 +\|\bbf\|_{L^2}\|\triangle\bu\|_{L^2}
\\&\leq \notag
\frac{1}{\nu}\|\nabla\bu\|_{L^2}^2\|\bu\|_{L^\infty}^2+\frac{\nu}{2}\|\triangle\bu\|_{L^2}^2
+\frac{1}{\nu}\|\bbf\|_{L^2}^2.
\end{align}
Let $0<\epsilon< t< T^{**}$.  Rearranging \eqref{H1_bound_twist_global} and integrating on $[\epsilon,t]$, we obtain
\begin{align*}
&\quad
\|\nabla \bu(t)\|_{L^2}^2
 + \nu\int_\epsilon^t\|\triangle\bu(s)\|_{L^2}^2\,ds
\\&\leq\notag
\|\nabla \bu(\epsilon)\|_{L^2}^2
+ \frac{2}{\nu}\int_\epsilon^t\|\nabla\bu(s)\|_{L^2}^2\|\bu(s)\|_{L^\infty}^2\,ds
+ \int_\epsilon^{t}\|\bbf(s)\|_{L^\infty}\,ds
\\&\leq\notag
\|\nabla \bu(\epsilon)\|_{L^2}^2
+ \frac{2}{\nu}\|\bu(\epsilon)\|_{L^\infty}^2\int_0^{T^{**}}\|\nabla\bu(s)\|_{L^2}^2\,ds
+ \int_0^{T^{**}}\|\bbf(s)\|_{L^\infty}\,ds.
\end{align*}
If $T^{**}<\infty$, then the above inequality implies $\lim_{t\nearrow T^{**}}\|\nabla \bu(t)\|_{L^2}^2<\infty$, since $\bbf\in L^2(0,T;\dot{H}^4)\subset L^1(0,T;L^\infty)$. Hence, owing to Theorem \ref{thm_strong_local}, the solution can be continued past $T^{**}$, contradicting the maximality of $T^{**}$.
\end{proof}
\begin{remark}\label{remark_attractor_needs_poincare}
It may seem that, now that we know that solutions are global, one could prove the existence of a global attractor using the energy equality \eqref{en_eq} to establish time-independent bounds, along with higher-order estimates, which can be controlled by the maximum principle.
If one had such estimates for the 3D Navier-Stokes equations, one could use such estimates to prove the existence of a global attractor (see, e.g., \cite{Robinson_2001,Enlow_Larios_Wu_2023_NSE}) by using the Poincar\'e inequality to get time-independent bounds.  However, due to the lack of conservation of mean in \eqref{Burgers_twist}, the Poincar\'e inequality is not available.  Another option would be to consider \eqref{Burgers_twist} with no-slip (i.e., homogeneous Dirichlet) boundary conditions.  The analysis for establishing local well-posedness goes through, and is even slightly simpler in some cases due to the availability of the Poincar\'e inequality.  However, proving global well-posedness seems to require the use of a maximum principle, which in turn seems to require the establishment of classical solutions.  However, repeating the steps of the proof of Theorem \ref{thm_Hm_bounds} would not work in the case of a bounded domain, as it would require higher-order boundary conditions, and one must use instead, e.g., elliptic regularity as in, e.g., \cite{Galdi_2000_NSE_book_chapter,Temam_2001_Th_Num}, or semigroup methods.
However, for the sake of simplicity, we do not pursue such approaches here.  Instead, in the next section, we simply add a damping term, which allows for time-independent bounds.
\end{remark}
\section{The Global Attractor (Damped-Driven Case)}\label{sec_global_attractor}
 \noindent
In this section, we prove the existence of a global attractor in $L^2$ for the semigroup on $H^1$, but for the reasons discussed in Remark \ref{remark_attractor_needs_poincare}, we add a linear damping term to  \eqref{Burgers_twist_all}, as shown in \eqref{Burgers_twist_damping_all}. To avoid the trivial case, we also drive the solution by a time-independent force $\bbf$.  Moreover, for similar reasons to those discussed in Remark \ref{remark_no_weak_Burgers_twist_sol}, we are not able to pass all derivatives to test functions, which prevents us from obtaining a time-independent $H^1$ bound, unless we assume that the initial data is in $L^\infty$.  However, the $L^\infty$ norm is of course not uniformly controlled by the $H^1$ norm of the initial data, and thus we are only able to prove the existence of a global attractor in $L^2$ for initial data in $H^1$, or, restricting the phase space, we can prove the existence of a global attractor in $L^\infty\cap H^1$.  We consider both possibilities here.
Given  damping coefficient $\gamma>0$, consider 
\begin{subequations}
\label{Burgers_twist_damping_all}
\begin{empheq}[left=\empheqlbrace]{align}
\label{Burgers_twist_damping}
\pd{\bu}{t}+\bomega\times\bu +\gamma\bu&= \nu \triangle\bu+\bbf, \qquad \bomega:=\nabla\times\bu,
\\
\label{Burgers_twist_damping_initial_data}
\bu(0) &= \bu_0.
\end{empheq}
\end{subequations}
Note that all the theorems of Section \ref{sec_rB_well_posedness} hold for \eqref{Burgers_twist_damping_all} \textit{mutatis mutandis} for $\gamma>0$, and with nearly identical proofs.  In particular, for $\bu_0\in H^1(\nT^3)$, by the $\gamma>0$ analogue of Theorem \ref{thm_Burgers_twist_GWP}, there exists a unique global strong solution $\bu$ to \eqref{Burgers_twist_damping_all} which is a classical solution for $t>0$.  We note the following slight change in the maximum principle, which will be important.
\begin{lemma}[Maximum principle, damped case]\label{lemma_maximum_principle_viscous_damped}
Assume the same hypotheses as in Lemma \ref{lemma_maximum_principle_viscous}, except that $\bu$ solves \eqref{Burgers_twist_damping} instead of \eqref{Burgers_twist}.  Let $\gamma\geq0$.  Then 
\begin{align}\label{sup_max_twist_damped}
    \sup_{t\in[a,b]}\|\bu(t)\|_{L^\infty}
    \leq 
    e^{-\gamma(t-a)}\|\bu(a)\|_{L^\infty} + \int_a^te^{-\gamma(t-s)}\|\bbf(s)\|_{L^\infty}\,ds.
\end{align}
\end{lemma}
\begin{proof} 
Assume the hypotheses. 
The $\gamma=0$ case is covered by Lemma \ref{lemma_maximum_principle_viscous}, so assume $\gamma>0$. 
 Consider again $w=w(\bx,t):=\sqrt{\epsilon+|\bu(\bx,t)|^2}$ for $\epsilon>0$.  Arguing as in the proof of Lemma \ref{lemma_maximum_principle_viscous}, we obtain instead of \eqref{w_estimate},
\begin{align}\label{w_estimate_damped}
w_t +\gamma \tfrac{|\bu|^2}{w}&\leq 
\nu \triangle w + \|\bbf\|_{L^\infty}.
\end{align}
Since $\frac{|\bu|^2}{w}=w-\frac{\epsilon}{w}\geq w-\sqrt{\epsilon}$, denoting $W(t):=e^{\gamma t}w(t)$, we obtain
\begin{align}\label{w_estimate_damped2}
W_t&\leq 
\nu \triangle W + e^{\gamma t}\pnt{\|\bbf\|_{L^\infty} + \gamma\sqrt{\epsilon}}.
\end{align}
Denoting $V(t)=W(t)-\int_a^t e^{\gamma s}(\|\bbf(s)\|_{L^\infty}+\gamma\sqrt{\epsilon})\,ds$, inequality \eqref{w_estimate_damped2} yields $V_t-\nu\triangle V\leq0$, so by the usual parabolic maximum principle, $\|V(t)\|_{L^\infty}\leq \|V(a)\|_{L^\infty}$, from which \eqref{sup_max_twist_damped} easily follows by sending $\epsilon\rightarrow0$.  
\end{proof}
\subsection{\texorpdfstring{$L^2$}{} bounds in \texorpdfstring{$H^1$}{}}
We are now in a position to prove the existence of a global attractor.  From \eqref{en_eq}, we obtain as in \eqref{rotBurgersL2},
\begin{align}\label{rotBurgersL2_rigourous}
    \frac{d}{dt}\|\bu\|_{L^2}^2 + 2\gamma\|\bu\|_{L^2}^2 + \nu\|\nabla\bu\|_{L^2}^2
    &\leq
    \tfrac{2}{\nu}\|\bbf\|_{H^{-1}}^2,
\end{align}
Gr\"onwall's inequality then yields
\begin{align*}
    \|\bu(t)\|_{L^2}^2 + \nu\int_0^t e^{-2\gamma(t-s)}\|\nabla\bu(s)\|_{L^2}^2\,ds
    &\leq 
    e^{-2\gamma t}\|\bu_0\|_{L^2}^2+
    2\tfrac{1-e^{-2\gamma t}}{\gamma\nu}\|\bbf\|_{H^{-1}}^2.
\end{align*}
Denoting by $t_2(R):=\frac{-1}{2\gamma}\log\pnt{\max\set{1,\frac{\|\bbf\|_{H^{-1}}^2}{2\nu\gamma R^2}}}$ the absorption time of the $L^2$-ball of radius $R$, it holds that for all $t\geq t_2(\|\bu_0\|_{L^2})$, 
\begin{align}\label{rho_0}
    \|\bu(t)\|_{L^2}^2\leq \frac{5\|\bbf\|_{H^{-1}}^2}{2\nu\gamma}=:\rho_0.
\end{align}
Hence, the set
\[
B_0:=\set{\bv\in H^1(\nT^3)\big| \|\bv\|_{L^2}\leq\rho_0}.
\] 
is a bounded absorbing set in the $L^2(\nT^3)$ topology.  To show the existence of an attractor, we need compactness. 
Toward this end, denote $\tau:=1/\gamma$.  Integrating \eqref{rotBurgersL2_rigourous} on $[t,t+\tau]$ for $t\geq t_2(\|\bu_0\|_{L^2})$ yields
\begin{align}\label{rotBurgersL2_window}
    \nu\int_t^{t+\tau}\|\nabla\bu(s)\|_{L^2}^2\,ds
    &\leq
    \frac95\rho_0.
\end{align}
Next, we estimate the time derivative, but since we do not have a uniform-in-time estimate for $H^1$, we proceed somewhat differently from \eqref{time_der_est2}; however, the estimate is very similar to that of the 3D Navier-Stokes case.  Let $\bvphi\in H^1(\nT^3)$.  Then,
\begin{align}\label{time_der_est_43}
    &\quad
    \ip{\partial_t\bu}{\bvphi}
    \\&=\notag
    -\ip{(\nabla\times\bu)\times\bu}{\bvphi}-\gamma(\bu,\bvphi) -\nu\ip{\nabla\bu}{\nabla\bvphi} + \ip{\bbf}{\bvphi}
    \\&\leq\notag
    \|\nabla\times\bu\|_{L^2}\|\bu\|_{L^3}\|\bvphi\|_{L^6} 
    +(\gamma+\nu)\|\bu\|_{H^1}\|\bvphi\|_{H^1}
    +\sqrt{2}\|\bbf\|_{H^{-1}}\|\bvphi\|_{H^1}
    \\&\leq\notag
    C\|\bu\|_{H^1}^{3/2}\|\bu\|_{L^2}^{1/2}\|\bvphi\|_{H^1} 
    +(\gamma+\nu)\|\bu\|_{H^1}\|\bvphi\|_{H^1}
    +\sqrt{2}\|\bbf\|_{H^{-1}}\|\bvphi\|_{H^1}
\end{align}
Taking the supremum over $\bvphi\in H^1$ with $\|\bvphi\|_{H^1}=1$ and using \eqref{rho_0} and \eqref{rotBurgersL2_window} yields for any $t\geq t_2$,
\[
\partial_t u\text{ is bounded in }L^{4/3}(t,t+\tau;H^{-1}(\nT^3)),
\]
where the bound depends only on $\rho_0$, $\nu$, $\gamma$, and $\|\bbf\|_{H^{-1}}$.
Take any sequence of points $\bu^n_0\in B_0$ for initial data, and denote $\bu^n(t):=S(t)\bu^n_0$ where $S(t)$ is the semigroup operator on $H^1$ mapping initial data to the solution of \eqref{Burgers_twist_damping_all} at time $t$, which is defined for all $t\geq0$ by (the $\gamma>0$ analogue of) Theorem \ref{thm_Burgers_twist_GWP}.  For $s\in[0,\tau]$, denote $\bv^n(s):=\bu^n(t_2+s)$.  By the above bounds, the the Aubin-Lions Compactness Theorem implies that there exists a subsequence $\bv^{n_j}$ such that $\bv^{n_j}$ converges in $L^2(0,\tau;L^2)$.  By choosing a Lebesgue point $s_*\in(0,\tau)$ for the sequence, it follows that the set $\set{\bv^n(s_*)}_{n\in\nN}=\set{\bu^n(t_2+s_*)}_{n\in\nN}$ is relatively compact in $L^2(\nT^3)$.  Since  the initial sequence was arbitrary, this implies that $S(t_2+s_*)B_0$ is relatively compact in $L^2(\nT^3)$.  Thus, for any $t\geq t_2+s_*$, $S(t)B_0=S(t-(t_2+s_*))S(t_2+s_*)B_0$ is relatively compact in $L^2(\nT^3)$ by the continuity of $S(\cdot):H^1\rightarrow L^2$, which follows from (the $\gamma>0$ analogue of) Theorem \ref{thm_strong_local}.  Hence,
\begin{align}\label{def_attractor_L2}
\mathcal{A}:=\bigcap_{t\geq t_2+s_*}\overline{\bigcup_{s\geq t}S(s)B_0}^{\,L^2}
\end{align}
is a non-empty $L^2$-compact invariant set that attracts (in the $L^2$-sense) every bounded subset of $H^1$; i.e., $\mathcal{A}$ is the desired global attractor for system \eqref{Burgers_twist_damping_all}.
\subsection{\texorpdfstring{$L^\infty\cap H^1$}{} bounds}
In this section, we obtain a strong form of attraction, but on a restricted phase space.  To achieve this, we exploit the maximum  principle (which is not available in the 3D Navier-Stokes equations \eqref{NSE}) and the vanishing of the nonlinear term in $L^2$ estimates (which is not available in the 3D Burgers equation \eqref{Burgers}). 
From Theorem \ref{thm_Burgers_twist_GWP}, we know that if $\bu_0\in H^1$, then the solution to \eqref{Burgers_twist} $ \bu$ is a classical solution for $t>0$, and hence its $L^\infty$ norm is bounded.  However, it is not clear that it can be bounded uniformly in time in terms of $\|\bu_0\|_{H^1}$, so our strategy is to restrict the phase space.
 Hence, we work in the Banach space
\begin{align}\label{X_Banach_space}
    X:=H^1(\nT^3)\cap L^\infty(\nT^3), \quad\|\cdot\|_X:= \|\cdot\|_{H^1} + \|\cdot\|_{L^\infty}.
\end{align}
Note that $H^2(\nT^3)=H^2(\nT^3)\cap L^\infty(\nT^3)$ is compactly contained in $X$, which follows from the Rellich–Kondrachov Theorem and \eqref{agmon} with $s_1=1$, $s_2=2$.
Let $\bu_0\in X$ and $\bbf\in\dot{H}^4(\nT^3)$ be given (in particular, we assume $\bbf$ is time-independent).  Let $\bu$ be the classical solution of \eqref{Burgers_twist_damping_all}. From \eqref{sup_max_twist_damped}, we obtain
\begin{align}\label{sup_max_twist_damped_time_indep}
    \sup_{t\in[0,T]}\|\bu(t)\|_{L^\infty}
    \leq 
    e^{-\gamma t}\|\bu_0\|_{L^\infty} + 
    \frac{1-e^{-\gamma t}}{\gamma}\|\bbf\|_{L^\infty}
\end{align}
Denoting by $t_\infty(R) :=\frac{-1}{\gamma}\log\pnt{\max\set{1,\frac{\|\bbf\|_{L^\infty}}{2\gamma R}}}$ the $L^\infty$ the absorption time of the $L^\infty$-ball of radius $R$, it holds that for any $t\geq t_\infty(\|\bu_0\|_{L^\infty})$,
\begin{align*}
    \sup_{s\geq t}\|\bu(s)\|_{L^\infty}
    \leq 
    \frac{3}{2\gamma}\|\bbf\|_{L^\infty}=:\rho_\infty.
\end{align*}
We established an $H^1$ bound in \eqref{H1_bound_twist}, but here we estimate differently since we may now use the maximum principle.  
Denote
\[t_X:=\max\set{t_2(\|\bu_0\|_{L^2}),t_\infty(\|\bu_0\|_{L^\infty})}\in(0,\infty).\]
Taking the inner-product of \eqref{Burgers_twist_damping} with $-\triangle\bu$, for $t\geq t_X$, we obtain 
\begin{align}
&\quad
\frac12\frac{d}{dt}\|\nabla\bu\|_{L^2}^2 + \gamma\|\nabla\bu\|_{L^2}^2 + \nu\|\triangle\bu\|_{L^2}^2
\\&= \notag
((\nabla\times\bu)\times\bu,\triangle\bu)
-(\bbf,\triangle\bu)
\\&\leq \notag
\rho_\infty\|\nabla\bu\|_{L^2}\|\triangle\bu\|_{L^2}
+
\|\bbf\|_{L^2}\|\triangle\bu\|_{L^2}
\\&\leq \notag
\frac{\rho_\infty^2}{\nu}\|\nabla\bu\|_{L^2}^2 +\frac{1}{\nu}\|\bbf\|_{L^2}^2 + \frac{\nu}{2}\|\triangle\bu\|_{L^2}^2.
\end{align}
Hence,
\begin{align}\label{H1_bound_damped_Linfty}
\frac{d}{dt}\|\nabla\bu\|_{L^2}^2 + 2\gamma\|\nabla\bu\|_{L^2}^2 + \nu\|\triangle\bu\|_{L^2}^2
\leq
\frac{2\rho_\infty^2}{\nu}\|\nabla\bu\|_{L^2}^2 +\frac{2}{\nu}\|\bbf\|_{L^2}^2.
\end{align}
Taking $t\geq t_X+\tau$ and integrating \eqref{H1_bound_damped_Linfty} on $[\sigma,t]$ with $\sigma\in[t-\tau,t)$, we obtain (since that $0<t-\sigma\leq 1/\gamma$)
\begin{align}
&\quad \notag
\|\nabla\bu(t)\|_{L^2}^2 + 2\gamma\int_{\sigma}^{t}\|\nabla\bu(s)\|_{L^2}^2\,ds + \nu\int_{\sigma}^{t}\|\triangle\bu(s)\|_{L^2}^2\,ds
\\&\leq \notag
\|\nabla\bu(\sigma)\|_{L^2}^2
+\frac{2\rho_\infty^2}{\nu}\int_{t-\tau}^{t}\|\nabla\bu(s)\|_{L^2}^2\,ds +\frac{2}{\gamma\nu}\|\bbf\|_{L^2}^2
\leq
\|\nabla\bu(\sigma)\|_{L^2}^2
+\widetilde{\rho}_1,
\end{align}
where $\widetilde{\rho}_1:=\frac{18\rho_\infty^2\rho_0}{5\nu}
+\frac{2}{\gamma\nu}\|\bbf\|_{L^2}^2$, where $\rho_0$ is defined by \eqref{rho_0}.  
Integrating this bound with respect to $\sigma\in [t-\tau,t]$, dropping some terms on the left-hand side, and multiplying by $\tau^{-1}=\gamma$, we obtain, for all $t\geq t_X + \tau$,
\begin{align}
\|\nabla\bu(t)\|_{L^2}^2 
\leq
\gamma\int_{t-\tau}^t\|\nabla\bu(\sigma)\|_{L^2}^2\,d\sigma
+\widetilde{\rho}_1
\leq
\frac{9\gamma\rho_0}{5}+\widetilde{\rho}_1=:\rho_1.
\end{align}
Hence, we have shown that the set 
\[
B_X:=\set{\bv\in X\big| \|\bv\|_{L^\infty}\leq\rho_\infty,\quad \|\bv\|_{H^1}\leq \rho_1}.
\]
is a bounded absorbing set in $X$. 
Next, for compactness, we need uniform-in-time estimates in $H^2$. Semi-group estimates are perhaps the most convenient for this (see, e.g., \cite{Pazy_2012_semigroup_book}), but to keep things self-contained, we briefly rederive the necessary estimates in the periodic setting. 
Setting $A_{\gamma,\nu}:=\gamma-\nu\triangle$, we rewrite \eqref{Burgers_twist_damping} on $[\sigma,t]$ as:
\begin{align}\label{Burgers_twist_mild}
    \bu(t) = e^{-A_{\gamma,\nu}(t-\sigma)}\bu(\sigma)
    + \int_\sigma^t e^{-(t-s)A_{\gamma,\nu}}[(\bu(s)\times(\nabla\times\bu(s))+\bbf]\,ds,
\end{align}
where $\bu(0)=\bu_0\in H^1(\nT^3)$ and $\bu$ is the unique strong solution to \eqref{Burgers_twist_damping_all}.
Since for any $r>0$ and $\bv\in H^{2r}(\nT^d)$, $\min\set{\gamma,\nu}^r\|\bv\|_{H^{2r}(\nT^d)}\leq \|A_{\gamma,\nu}^r\bv\|_{L^2}\leq \max\set{\gamma,\nu}^r\|\bv\|_{H^{2r}(\nT^d)}$, and for any $\bv\in L^2(\nT^d)$, since $x^{r}e^{-tx}\leq \pnt{\frac{r}{et}}^r$ for $r>0$, $x\geq0$, and $t>0$,
\begin{align*}
    \|A_{\gamma,\nu}^r e^{-t A_{\gamma,\nu}}\bv\|_{L^2}^2
    &=
    \sum_{\bk\in\nZ^d}(\gamma+\nu|\bk|^2)^{2r}e^{-2t(\gamma+\nu|\bk|^2)}|\widehat{\bv}_\bk|^2
    \leq
    \pnt{\frac{r}{et}}^{2r}\|\bv\|_{L^2}^2.
\end{align*}
Using this first with $r=1$, then with $r=1/2$, we obtain from \eqref{Burgers_twist_mild} applied on the interval $[t_X,t]$, with $t\geq t_X+\tau$ (so that $(t-t_X)^{-1}\leq\gamma$),
\begin{align}
\|\bu(t)\|_{H^2}
    &\leq \notag
    C_{\gamma,\nu} \|A_{\gamma,\nu} \bu(t)\|_{L^2}
    \\&\leq \notag
    C_{\gamma,\nu}\|A_{\gamma,\nu} e^{-A_{\gamma,\nu}(t-t_X)}\bu(t_X)\|_{L^2}
    \\&\qquad+ \notag
    C_{\gamma,\nu} \Big\|\int_{t_X}^t A_{\gamma,\nu}e^{-(t-s)A_{\gamma,\nu}}[\bu(s)\times(\nabla\times\bu(s))+\bbf]\,ds\Big\|_{L^2}
    \\&\leq \notag
    C_{\gamma,\nu}\pnt{(t-t_X)^{-1}\|\bu(t_X)\|_{L^2}
    +\|\bbf\|_{L^2}}
    \\&\qquad+ \notag
    C_{\gamma,\nu}\int_{t_X}^t\pnt{\frac{1/2}{e(t-s)}}^{1/2}\|\bu(s)\times(\nabla\times\bu(s))\|_{H^1}\,ds
    \\&\leq \notag
    C_{\gamma,\nu}\pnt{\gamma\sqrt{\rho_0}
    +\|\bbf\|_{L^2}}
    +
    C_{\gamma,\nu}\rho_1^{1/2}\int_{t_X}^t(t-s)^{-1/2}\|\bu(s)\|_{H^2}\,ds,
\end{align}
where we used \eqref{H1_leq_H1_H2} and $\int_\sigma^t e^{-(t-s)A_{\gamma,\nu}}\bbf\,ds = A_{\gamma,\nu}^{-1}(I-e^{-(t-\sigma)A_{\gamma,\nu}})\bbf$ so that 
\[\Big\|\int_\sigma^t A_{\gamma,\nu}e^{-(t-s)A_{\gamma,\nu}}\bbf\,ds\Big\|_{L^2}\leq \|(I-e^{-(t-\sigma)A_{\gamma,\nu}})\bbf\|_{L^2}\leq\|\bbf\|_{L^2}.\]
By Gr\"onwall's inequality in integral form, we obtain for any $t\geq t_X+\tau$,
\begin{align}\label{attractor_H2_est}
\|\bu(t)\|_{H^2}
&\leq
    C_{\gamma,\nu}\pnt{\gamma\sqrt{\rho_0}
    +\|\bbf\|_{L^2}}
    \exp\pnt{
    2C_{\gamma,\nu}\rho_1^{1/2}\sqrt{t-t_X}}=:\rho_2(t).
\end{align}
Note that for any $\bu_0\in B_X$, we may take the absorption time to be $t_X=0$.  
Thus, \eqref{attractor_H2_est} implies that  $\sup_{\bu_0\in B_X}\|S(\tau)\bu_0\|_{H^2}\leq\rho_2(\tau)<\infty$.  It follows that $S(\tau)B_X\subset H^2\subset\subset X$.  Hence, the set 
\begin{align}\label{def_attractor_X}
\mathcal{A}_X:=\bigcap_{t\geq 0}\overline{\bigcup_{s\geq t}S(s)B_X}^{\,X}
\end{align}
is a non-empty compact invariant set in $X$ that attracts every bounded set in $X$ uniformly.  Hence, we have proved the following theorem.
\begin{theorem}\label{thm_global_sttractor} Let $\gamma>0$, $\nu>0$, $\bbf\in\dot{H}^4(\nT^3)$.  Then the system \eqref{Burgers_twist_damping_all} has a global attractor defined by \eqref{def_attractor_L2} in the phase space $H^1(\nT^3)$ equipped with the $L^2(\nT^3)$ topology. Restricting to the space $X$  (see \eqref{X_Banach_space}) with the stronger norm $\|\cdot\|_{X}$, there exists a global attractor in the $X$ topology, given by \eqref{def_attractor_X}.
\end{theorem}

\section{The Inviscid Case}\label{sec_inviscid_case}
 \noindent
In this section, we consider the inviscid unforced version of \eqref{Burgers_twist}; i.e., 
\begin{align}\label{Burgers_twist_inviscid}
 \pd{\bu}{t}+\bomega\times\bu &= \mathbf{0},\qquad 
 \bomega:=\nabla\times\bu,\qquad
 \bu(0)=\bu_0.
\end{align}
We show that there exist solutions to this equation which blow up in finite time.  However, the question of local well-posedness for arbitrary smooth initial data is less clear, as we now discuss.
\subsection{Local well-posedness for the inviscid case}\label{sec_well_posedness_inviscid}
The inviscid equation \eqref{Burgers_twist_inviscid} has several analytical difficulties associated with it.  We discuss them here, as \eqref{Burgers_twist_inviscid} seems to provide a simple example for which many textbook approaches fail, and one must use more powerful theorems.  At first glance, \eqref{Burgers_twist_inviscid} may seem like a very nice equation, as the analogue of \eqref{Burgers_vor_dot} is simply $\frac12\pd{}{t}|\bu|^2=0$, and hence the equation is length-preserving, and in particular, smooth solutions conserve every $L^p$ norm, $1\leq p\leq\infty$, which may allow one to extract nice properties.  However, such an argument assumes that solutions exist.  How might one establish the existence of solutions?  If one tries Galerkin methods (as we did in the viscous case), an immediate obstacle is that the length-preserving property does not hold for the Galerkin system.  Moreover, the lack of an inviscid term prevents the use of Aubin-Lions compactness, but one can try to show that the sequence of Galerkin approximations forms a Cauchy sequence, as is commonly done for the Euler equations (i.e., \eqref{NSE} with $\nu=0$).  Roughly speaking, this works for the Euler equations, because one has the vanishing of higher-order derivatives in energy estimates, i.e., 
$(\bu\cdot\nabla\partial_\alpha\bu,\partial_\alpha \bu)=0$, since $\nabla\cdot\bu=0$.  However, 
for \eqref{Burgers_twist_inviscid}, one instead has only 
$((\nabla\times\bu)\times\partial_\alpha\bu,\partial_\alpha \bu)=0$, but
$((\nabla\times\partial_\alpha\bu)\times\bu,\partial_\alpha \bu)\neq0$; hence one must deal directly with the higher-order terms, as there is evidently no Kato-Ponce-type inequality to close the estimates on derivatives.  This fact causes difficulties with other approaches to \eqref{Burgers_twist_inviscid}, such as fixed-point methods, Picard iteration, and vanishing viscosity approaches.  Note that the inviscid Burgers equation (\eqref{Burgers} with $\nu=0$), suffers from the same difficulties, but it also has an advective structure, and hence is amenable to the method of characteristics.  Therefore, one might be tempted to exploit the geometric structure of \eqref{Burgers_twist_inviscid} directly.  In the 2D case, one might take the representation \eqref{inviscid_implicit_2D} in Section \ref{subsec_rot_num} as a starting place, but we do not do this here.  One can establish local-in-time analytic solutions for analytic initial data immediately by a direct application of the Cauchy-Kovalevskaya Theorem, yielding the following.
\begin{theorem}
    Let $\bu_0\in C^\omega(\nT^3)$ (i.e., it is real-analytic).  Then there exists a time $T>0$ such that the Cauchy problem \eqref{Burgers_twist_inviscid} has a unique solution $\bu\in C^\omega([-T,T]\times\nT^3))$.
\end{theorem}
For rougher initial data (for example, $\bu_0\in H^s(\nT^n)$, $s>5/2$), system \eqref{Burgers_twist_inviscid} is significantly more challenging that the standard inviscid Burgers equation (\eqref{Burgers} with $\nu=0$), as it exhibits unavoidable derivative loss\footnote{One can avoid derivative loss in Burgers equation, thanks to manipulations such as  $((\bu\cdot\nabla)\partial_{\alpha}\bu,\partial_{\alpha}\bu)=\frac12\int ((\bu\cdot\nabla)|\partial_{\alpha}\bu|^2=-\frac12\int (\nabla\cdot\bu)|\partial_{\alpha}\bu|^2$, so that the highest-order derivative is order $|\alpha|$, but for corresponding rotational quantity $((\nabla\times\partial_\alpha\bu)\times\bu, \partial_\alpha\bu)$, this does not seem possible, and the highest-order derivative is order $|\alpha|+1$.}, indicating that the Nash-Moser framework may be the only hope, but the linearized problem still suffers derivative loss, and the 
A maximum principle then still holds for these smooth solution on a short time interval, but (as we will see in the next section), it is not enough to prevent finite-time blow-up of solutions.  

We next show that there exist solutions for certain initial data which blow up in finite time.
\subsection{A finite-time singularity}\label{subsec_finite_time_blow_up} We seek a smooth solution to \eqref{Burgers_twist_inviscid} in the form $\bu(\bx,t) = (v,v,w)$, where $v = v(x,t)$, $w = w(x,t)$, and we have denoted the scalar $x=x_1$.  With this notation, system \eqref{Burgers_twist_inviscid} reduces to the following system (cf. \eqref{lamb_vector_explicit}), with given smooth initial data $v(0)=v_0$ and $w(0)=w_0$.
\begin{empheq}[left=\empheqlbrace]{align}
\begin{split}\label{inviscid_vw}
     v_t =\partial_t u_1&= w\partial_x w+v\partial_x v=\tfrac12\partial_x (w^2+v^2),\\
    v_t =\partial_t u_2&=-v\partial_x v=-\tfrac12\partial_x v^2,\\
    w_t =\partial_t u_3&= -v\partial_x w.   
\end{split}
\end{empheq}
Note that the second equation is exactly the form of the usual 1D inviscid Burgers equation, but it is possibly constrained by the first and third equations.  We check that $w$ can be chosen so as not to further constrain $v$. Subtracting the second equation from the first and rearranging, we obtain
\begin{align}\label{inviscid_conserved_vw}
    \partial_x(2v^2 + w^2) = 0.
\end{align}
Hence, $2v^2 + w^2$ is constant in space, say $c(t)=2v^2 + w^2$.  Moreover, 
\begin{align*}
    c'(t) 
    &= 
    4vv_t + 2ww_t
    = 
    2v\partial_x (w^2+v^2) - 2wv\partial_x w
    =
    v\partial_x (2v^2 + w^2) = 0
\end{align*}
by \eqref{inviscid_conserved_vw}.  Hence, $c(t)\equiv c$ is constant in time as well, so its value is fixed by the initial data; i.e., $c=2v_0^2 + w_0^2$.  It follows that we must have
\begin{align}\label{inviscid_conserved_vw_w_form}
    w = \pm\sqrt{2(v_0^2-v^2) + w_0^2}.
\end{align}
It is straight-forward to check that this choice of $w$ satisfies the first and third equation, so long as $v$ satisfies the second equation.  
We now construct a smooth solution to \eqref{inviscid_vw} which blows up in finite time as follows.  Let $v_0=v_0(x)$ be a non-constant smooth periodic function (e.g., $v_0(x)=\sin(x)$).  Let $v$ be the solution to the Burgers equation $v_t=-vv_x$ with initial data $v_0$.  It is well-known that $v$ is smooth for a short time, but that it develops a shock in finite time; i.e., $\|v_x(t)\|_{L^\infty}$ become infinite in finite time.   
Define $w$ by \eqref{inviscid_conserved_vw_w_form} (choosing, say, the sign to be positive), and choose $w_0=w_0(x)$ to be smooth, periodic, and bounded away from zero, so that $w$ is periodic and is defined and smooth for at least a short time.  Then either $v$ or $w$ becomes singular in finite time, as desired.  Hence, letting with solution $\bu=(v,v,w)$, we have proved the 3D case of the following result.
\begin{theorem}\label{thm_inviscid_blow_up}
There exists a smooth solution to \eqref{Burgers_twist_inviscid} in the 2D or 3D case which becomes singular in finite time.
\end{theorem}
The 2D case of course implies the 3D case, but it can proceed along slightly different lines. We present it separately, as it also allows us to exhibit a slightly different exact solution.  
In 2D, system \eqref{Burgers_twist_inviscid} becomes
\begin{align}\label{Burgers_twist_inviscid_2D}
\pd{\bu}{t}+\omega\bu^\perp=\mathbf{0}, \quad \omega := \partial_xu_2-\partial_y u_1,\quad \bu^\perp:=\binom{-u_2}{u_1}.
\end{align}
Let us look for a 1D solution of the form $u_1(\bx,t)=u(x,t)$, $u_1(\bx,t)=v(x,t)$.  System \eqref{Burgers_twist_inviscid_2D} becomes
\begin{subequations}\label{Burgers_twist_inviscid_2D_1D}
\begin{empheq}[left=\empheqlbrace]{align}
\label{Burgers_twist_inviscid_2D_1Du}
u_t- vv_x&=0,\\
\label{Burgers_twist_inviscid_2D_1Dv}
v_t+uv_x&=0.
\end{empheq}
\end{subequations}
Consider this system with smooth periodic initial data, $u(x,0)=u_0(x)=\cos(x)$, $v(x,0)=v_0(x)=\sin(x)$.  Then $u_0^2+v_0^2=1$. From \eqref{Burgers_twist_inviscid_2D_1D}, a quick calculation shows that $\pd{}{t}(u^2+v^2)=0$; hence, $(u(x,t))^2+(v(x,t))^2 = (u_0(x))^2+(v_0(x))^2=1$.  Differentiating in $x$ yields $uu_x+vv_x=0$.  Hence, \eqref{Burgers_twist_inviscid_2D_1Du} becomes the standard Burgers equation $u_t+uu_x=0$, which develops a singularity in finite time with the given initial data.  The function $v$ is then determined by \eqref{Burgers_twist_inviscid_2D_1Dv}, noting that the constraint $u^2+v^2=1$ is preserved by system \eqref{Burgers_twist_inviscid_2D_1D}, and hence not an additional constraint.  Hence, we have also shown the 2D case of Theorem \ref{thm_inviscid_blow_up}.
\section{Kuramoto-Sivashinsky with a twist}\label{sec_rotKSE}
\noindent
The Kuramoto-Sivashinsky equation (KSE) is a well-known model for many physical phenomena, including instabilities is laminar flame fronts \cite{Sivashinsky_1977}, flow of fluids down inclined planes \cite{sivashinsky1980vertical}, plasmas \cite{cohen1976non,laquey1975nonlinear}, reaction-diffusion systems \cite{Kuramoto_Tsuzuki_1975,Kuramoto_Tsuzuki_1976}, and a wide variety of generic instabilities after the underlying system undergoes a critical bifurcation \cite{misbah1994secondary}.  However, in dimensions higher than one, fundamental questions about its solutions remain open, such as whether solutions are unique, and whether they exist globally in time.  Even in one dimension, major questions remain open about the large-time dynamics of the system.  Note that physically, the KSE are 2D equations, as they model the evolution of a surface.  Hence, we confine our study to the 2D case, but the results here could surely be extended to the 3D case.
Besides the many physical applications of the KSE, the KSE is often studied by researchers interested in the incompressible Navier-Stokes equations (NSE) in hopes of gaining insight into the NSE, due in part to the similar-looking non-linear term.  However, strong differences exist between these equations.  For instance, while obtaining $L^2$ bounds on solutions of the NSE is trivial as the nonlinear term vanishes due to the incompressiblity, obtaining an $L^2$ bound on solutions to the multi-dimensional KSE is the main barrier preventing progress, as the nonlinear term does not vanish in energy estimates \cite{Ambrose_Mazzucato_2018,Biswas_Swanson_2007_KSE_Rn}.  Moreover, the 1D KSE is highly diffusive at small scales, and the stability of the equation relies on a strong forward cascade of energy due to the nonlinear term; hence, approaches to constructing better-behaved approximate equations that work for the NSE, such as adding higher-order diffusion or weakening the nonlinearity via filtering or smoothing, fail for the KSE (see, e.g.,  \cite{Kostianko_Titi_Zelik_2018}, and the discussions in \cite{Enlow_Larios_Wu_2023_NSE}).  (On the other hand, note that linear damping has a stabilizing effect on the 1D KSE, as investigated in \cite{misbah1994secondary}.)   Thus, construction of a well-posed approximate or phenomenological model of the KSE requires new approaches.  
Some modified versions of the KSE have shown to be ammeniable to analysis.  For example, in  \cite{Larios_Yamazaki_2020_rKSE}, a linear, anisotropic modification of the 2D KSE was proposed and shown to be globally well-posed.  This model modified one component of the vector form of the 2D KSE (breaking the gradient-structure of the equation), allowing for a maximum principle to be established for that component.  The authors were able to exploit this fact, along with the structure of the nonlinearity, obtain their global well-posedness results, although the attractor theory for this modified system remains open.  In \cite{Enlow_Larios_Wu_2023_NSE}, the authors proposed and studied the first known globally-well-posed PDE model which approximates solutions to the 2D KSE with arbitrary precision as the ``calming'' parameter $\epsilon\rightarrow0^+$. This was done by introducing a smooth truncation operator to the nonlinear term.  However, again the global attractor theory remains elusive for this modified model.
\subsection{A Rotational Version of the KSE}
  In the present work, a different modification of the 2D and 3D KSE (in vector form) is proposed.  Rather than modifying the linear term, the nonlinear term is modified to the Lamb vector.  While this modification also breaks the nonlinear structure, the new nonlinear term retains many of the qualities of the nonlinear term for the 2D KSE: it is quadratic,  and it is given by a product of the solution and a spatial derivative. However, as opposed to the 2D KSE itself, we are able to prove that this model is globally well-posed, and that the resulting system has a global attractor.  Hence, it is hoped that this new model may yield insight into the multi-dimensional KSE.
The $n$-dimensional Kuramoto-Sivashinsky system in dimensionless vector form is given by
\begin{align}
\label{KSE}
\pd{\bu}{t}+(\bu\cdot\nabla)\bu + \lambda\triangle\bu + \triangle^2\bu &= \mathbf{0},
\end{align}
with appropriate boundary and initial conditions, where $\lambda\geq0$ can be thought of as a dimensionless number\footnote{The notation ``$\lambda$'' goes back to at least \cite{Liu_Krstic_1999_lambda_parameter_BC_KSE}, where the authors refer to it as an ``anti-diffusion'' parameter. Some papers, such as \cite{Frisch_She_Thual_1986_JFM_viscoelastic}, use a dimensionless formulation, but with a dimensionless hyperviscosity parameter on the fourth-order term.  Many authors take $\lambda=1$ and instead consider the $L$-periodic domain, with unstable modes occurring for $|\bk|\geq \frac{L}{2\pi}$, but the dimensionless $\lambda$ notation has the advantage of acting more like a Reynolds number, fixing the domain, and scaling the coefficients, with larger values of $\lambda$ corresponding to greater instability.} controlling the stability of the linear term; namely, for wave number $\bk$,  $|\bk|>\sqrt{\lambda}$ implies linear stability, whereas $|\bk|<\sqrt{\lambda}$ correspond to linearly stable modes.  Unstable modes are a source of energy for the equation, and hence the equation is often considered without forcing.  The scalar form of the equations is given by
\begin{align}
\label{KSE_scalar}
\pd{\phi}{t}+\tfrac{1}{2}|\nabla\phi|^2 + \lambda\triangle\phi + \triangle^2\phi &= 0,
\end{align}
By setting $\bu=\nabla\phi$ in \eqref{KSE_scalar}, one formally obtains \eqref{KSE}.  
The KSE was first derived in \cite{Kuramoto_Tsuzuki_1976,Sivashinsky_1977} (see also \cite{Kuramoto_1978,Sivashinsky_1980_stoichiometry,sivashinsky1980vertical}).  
As discussed above, this system is notoriously difficult to handle in dimension $n>1$, and very little is known about its well-posedness.  Motivated by the identity \eqref{vec_identity}, we propose the following modification of system \eqref{KSE} on the periodic domain $\nT^2$.  Namely,
\begin{subequations}\label{vKSE}
\begin{empheq}[left=\empheqlbrace]{align}
\label{vKSE_mo}
\pd{\bu}{t}+\bomega\times\bu + \lambda\triangle\bu + \triangle^2\bu &= \mathbf{0},
\qquad
\bomega:=\nabla\times\bu,
\\
\bu(\bx,0)&=\bu_0(\bx).
\end{empheq}
\end{subequations}
This formulation has several points of interest.  Comparing with \eqref{Burgers_twist}, the hyperdiffusion from the term $\triangle^2\bu$ will allow us to prove global well-posedness without the need for a maximum principle, thanks to the vanishing of the nonlinear term in energy estimates (a property not enjoyed by the usual 2D or 3D KSE, for which global well-posedness remains a challenging open problem).  We will also be able to show the existence of a compact global attractor in $L^2(\Omega)$ for sufficiently small $\lambda>0$ (a property shared with the 1D KSE).  Notice also that \eqref{vKSE} is not in general to be the vectorial form of any scalar equation such as \eqref{KSE_scalar}, and hence it should be thought of as at most a phenomenological cousin of \eqref{KSE}, not a direct analogue.  The higher-order boundary conditions, along with the hyper-diffusion imply that we do not need especially high-order estimates, unlike in the case of \eqref{Burgers_twist}. 
Note that thanks to identity \eqref{vec_identity}, the original KSE system \eqref{KSE} can be written in the form 
\begin{align*}
 \pd{\bu}{t}+\bomega\times\bu + \tfrac12\nabla|\bu|^2 + \lambda\triangle\bu + \triangle^2\bu &= \bbf,
\end{align*}
Thus, the new system \eqref{vKSE} can be thought of as the original KSE system \eqref{KSE} minus the gradient term $\tfrac12\nabla|\bu|^2$.  If instead we were to throw away the term $\bomega\times\bu$, the resulting system would simply be a vectorial version of equation \eqref{KSE_scalar}; hence, our interest lies in examining system \eqref{vKSE}.
\begin{definition}
Let $\bu_0\in L^2(\nT^2)$ and let $T>0$.  We say that $\bu$ is a \textit{solution} to \eqref{vKSE} on the interval $[0,T]$ if $\bu\in L^\infty([0,T];H^1(\nT^2))\cap L^2(0,T;H^3(\nT^2))$, $\partial_t\bu\in L^2(0,T;H^{-1}(\nT^2))$, and $\bu$ satisfies \eqref{vKSE_mo} in the sense of $L^2(0,T;H^{-1}(\nT^2))$, with initial data satisfied in the sense of $C([0,T];L^2(\nT^2))$.
\end{definition}
\begin{theorem}\label{thm_GWP_vKSE}
 Let $\bu_0\in L^2(\nT^2)$ and let $T>0$. Then  solutions to \eqref{vKSE} on $[0,T]$ exist and are unique.
\end{theorem}
 \subsection{Well-Posedness for Rotational KSE}\label{sec_vKSE_well_posedness}
Part of the interest in studying the higher-dimensional Kuramoto-Sivashinsky equations is the (superficial) resemblance of the nonlinear term $\bu\cdot\nabla\bu$ to the nonlinear term in \eqref{NSE}.
\noindent
In this section, we prove that system \eqref{vKSE} is globally well-posed.  
For each $N\in\nN$, consider the following Galerkin ODE system based on \eqref{vKSE}
\begin{subequations}
\begin{align}
\label{vKSE_Gal}
\pd{\bu^N}{t}&= -P_N((\nabla\times\bu^N)\times\bu^N) -\lambda\triangle\bu^N - \triangle^2\bu^N,
\\
\bu^N(\bx,0)&=P_N(\bu_0(\bx)),
\end{align}
\end{subequations}
For any fixed $N$, this system represents a finite-dimensional ODE with quadratic right-hand side (technically, it is the coefficients of $\bu^N$ that form the ODE).  Hence, by the Picard-Lindel\"of it has a unique solution $\bu_N(t)\in C^1([-T_N,T_N];C^\infty(\nT^3))$ for some $T_N>0$.  Taking a justified inner-product of both sides of \eqref{vKSE_Gal} with $\bu^N$, and using
\begin{align}\label{nonlinear_term_vanishes}
 (P_N((\nabla\times\bu^N)\times\bu^N),\bu^N) 
 &= ((\nabla\times\bu^N)\times\bu^N,\bu^N) 
 \\&= \notag
 \int_\Omega ((\nabla\times\bu^N)\times\bu^N)\cdot\bu^N\,d\bx
 = 0,
\end{align}
we find that
\begin{align*}
  \frac12\frac{d}{dt}\|\bu^N\|_{L^2}^2 + \|\triangle\bu^N\|_{L^2}^2
  &= -\lambda(\triangle\bu^N,\bu^N)
  \leq
  \lambda\|\triangle\bu^N\|_{L^2}\|\bu^N\|_{L^2}
  \\&\leq
  \frac{1}{2}\|\triangle\bu^N\|_{L^2}^2 + \frac{\lambda^2}{2}\|\bu^N\|_{L^2}^2.
\end{align*}
Thus,
\begin{align}\label{L2_bound}
  \frac{d}{dt}\|\bu^N\|_{L^2}^2 + \|\triangle\bu^N\|_{L^2}^2
  &\leq
   \lambda^2\|\bu^N\|_{L^2}^2.
\end{align}
Hence, $\|\bu^N(t)\|_{L^2}^2 \leq e^{\lambda^2t}\|\bu^N(0)\|_{L^2}^2 \leq e^{\lambda^2t}\|\bu_0\|_{L^2}^2=:K_0(t)$, so that we may take the time of existence and uniqueness of the Galerkin ODE to be $T_N=\infty$.  By integrating \eqref{L2_bound} in time, we see that for any $T>0$, 
\begin{align}\label{L2bound}
 \text{$\set{\bu^N}_{N\in\nN}$ is a bounded sequence in $L^\infty(0,T;L^2)$ and $L^2(0,T;H^2)$.}
\end{align}
Thus, by the Banach-Alaoglu Theorem, there exists a function $\bu$ and a subsequence of $\set{\bu^N}_{N\in\nN}$ (which we relabel as $\set{\bu^N}_{N\in\nN}$) such that
\begin{subequations}
\begin{align}
 \bu^N \overset{*}{\rightharpoonup}\bu &\text{ in } L^\infty(0,T;L^2),
 \\
 \bu^N\rightharpoonup \bu &\text{ in } L^2(0,T;H^2).
\end{align}
\end{subequations}
Next, taking a justified inner-product of both sides with $-\triangle\bu^N$, and integrating by parts, we find
\begin{align}
&\quad\notag
 \frac12\frac{d}{dt}\|\nabla\bu^N\|_{L^2}^2 + \|\triangle\nabla\bu^N\|_{L^2}^2
  \\&= \notag
  \lambda(\triangle\nabla\bu^N,\nabla\bu^N) - ((\nabla\times\bu^N)\times\bu^N,\triangle\bu^N)
  \\&\leq\notag
  \lambda\|\triangle\nabla\bu^N\|_{L^2}\|\nabla\bu^N\|_{L^2} 
  + \|\nabla\bu^N\|_{L^2}\|\bu^N\|_{L^4}\|\triangle\bu^N\|_{L^4}
  \\&\leq\notag
  \lambda\|\triangle\nabla\bu^N\|_{L^2}\|\nabla\bu^N\|_{L^2} 
  + C\|\bu^N\|_{L^2}^{1/2}\|\nabla\bu^N\|_{L^2}^{3/2}\|\triangle\bu^N\|_{L^2}^{1/2}\|\triangle\nabla\bu^N\|_{L^2}^{1/2}
  \\&\leq\notag
  \lambda\|\triangle\nabla\bu^N\|_{L^2}\|\nabla\bu^N\|_{L^2} 
  + K_0(t)\|\nabla\bu^N\|_{L^2}^{3/2}\|\triangle\bu^N\|_{L^2}^{1/2}\|\triangle\nabla\bu^N\|_{L^2}^{1/2}
  \\&\leq\notag
  \frac14\|\triangle\nabla\bu^N\|_{L^2}^2 + \lambda^2\|\nabla\bu^N\|_{L^2}^2
  + K_0\|\nabla\bu^N\|_{L^2}^2\|\triangle\bu^N\|_{L^2}^{2/3}
   + \frac{1}{4}\|\triangle\nabla\bu^N\|_{L^2}^2
  \\&=\notag
  \frac12\|\triangle\nabla\bu^N\|_{L^2}^2 + 
  (\lambda^2+ K_0(t)\|\triangle\bu^N\|_{L^2}^{2/3})\|\nabla\bu^N\|_{L^2}^2.
\end{align}
Thus,
\begin{align}\label{H1_bound_KSE}
 \frac{d}{dt}\|\nabla\bu^N\|_{L^2}^2 + \|\triangle\nabla\bu^N\|_{L^2}^2
  \leq
  2K_1(t)\|\nabla\bu^N\|_{L^2}^2,
\end{align}
where $K_1(t):=(\lambda^2+ K_0(t)\|\triangle\bu^N(t)\|_{L^2}^{2/3})$.  Thus, 
\begin{align}
&\quad
 \frac{d}{dt}\pnt{\|\nabla\bu^N(t)\|_{L^2}^2\exp\pnt{-\int_0^t 2K_1(s)\,ds}}
 \\&\quad\notag+ 
 \|\triangle\nabla\bu^N\|_{L^2}^2\exp\pnt{-\int_0^t 2K_1(s)\,ds}
  \leq
  0.
\end{align}
Next, let $t-1\leq s < t$, and integrate \eqref{H1_bound_KSE} on $[s,t]$ to obtain
\begin{align}
&\quad
 \|\nabla\bu^N(t)\|_{L^2}^2\exp\pnt{-\int_0^t 2K_1(\sigma)\,d\sigma}
 \\&\quad+ \notag
 \int_0^t\|\triangle\nabla\bu^N(\tau)\|_{L^2}^2\exp\pnt{-\int_0^\tau 2K_1(\sigma)\,d\sigma}\,d\tau
  \\&\leq\notag
\|\nabla\bu^N(s)\|_{L^2}^2\exp\pnt{-\int_0^s 2K_1(\sigma)\,d\sigma},
\end{align}
so that
\begin{align}
&
 \|\nabla\bu^N(t)\|_{L^2}^2
 + \int_0^t\|\triangle\nabla\bu^N(\tau)\|_{L^2}^2\exp\pnt{\int_\tau^t 2K_1(\sigma)\,d\sigma}\,d\tau
  \\\leq&\notag
\|\nabla\bu^N(s)\|_{L^2}^2\exp\pnt{\int_s^t 2K_1(\sigma)\,d\sigma}.
\end{align}
Next, we integrate again between $s=t-1$ and $s=t$ to obtain
\begin{align}\label{H1_bound_Gronwall}
&
 \|\nabla\bu^N(t)\|_{L^2}^2
 + \int_0^t\|\triangle\nabla\bu^N(\tau)\|_{L^2}^2\exp\pnt{\int_\tau^t 2K_1(\sigma)\,d\sigma}\,d\tau
  \\\leq&\notag
\int_{t-1}^t\|\nabla\bu^N(s)\|_{L^2}^2\exp\pnt{\int_s^t 2K_1(\sigma)\,d\sigma}\,ds.
\end{align}
Thus, due to the previous bounds on $\bu^N$, we obtain that, for any $T>0$,
\begin{subequations}
\begin{align}\label{H1bound}
 \bu^N \overset{*}{\rightharpoonup}\bu &\text{ in } L^\infty(0,T;H^1),
 \\
 \bu^N\rightharpoonup \bu &\text{ in } L^2(0,T;H^3).
\end{align}  
\end{subequations}
Next, we bound the time derivative.  For $\bvphi\in H^{1}$, we estimate 
\begin{align*}
 \ip{\pd{\bu^N}{t}}{\bvphi}
 &= 
 -\ip{(\nabla\times\bu^N)\times\bu^N}{P_N\bvphi} -\ip{\triangle\bu^N}{\bvphi} - \ip{\triangle^2\bu^N}{\bvphi}
 \\&\leq
  C\|\bu^N\|_{H^1}^{3/2}\|\bu^N\|_{L^2}^{1/2}\|P_N\bvphi\|_{H^1}
  +\lambda\|\bu^N\|_{H^1}\|\bvphi\|_{H^1} 
  \\&\quad\notag
  +\|\bu^N\|_{H^3}\|\bvphi\|_{H^1}.
\end{align*}
Hence,
\begin{align*}
 \left\|\pd{\bu^N}{t}\right\|_{H^{-1}}
 \leq   
 C\|\bu^N\|_{H^1}^{3/2}\|\bu^N\|_{L^2}^{1/2}
  +\lambda\|\bu^N\|_{H^1}
  +\|\bu^N\|_{H^3}.
\end{align*}
Thus, due to \eqref{H1_bound_Gronwall}, 
\begin{align}\label{dudt_bound}
\text{$\set{\pd{\bu^N}{t}}_{N\in\nN}$ is a bounded sequence in $L^2(0,T;H^{-1})$.}
\end{align}
Therefore, by the Aubin-Lions Lemma, due to \eqref{H1_bound_Gronwall} and \eqref{dudt_bound}, we obtain 
\begin{align}\label{strong_conv}
 \bu^N\rightarrow \bu &\text{ strongly in } L^\infty(0,T;L^2),
\end{align}
and moreover, $\bu \in C([0,T];L^2)$.  Thus, taking an inner product of \eqref{vKSE_Gal} with $\bvphi\in H^1$, integrating by parts in space, and integrating on $[0,t]$ for $t\in(0,T)$, we find
\begin{align}
 \pair{\bu^N(t)}{\bvphi}
 &=
  \pair{\bu^N(0)}{\bvphi}
 -\int_0^t\ip{(\nabla\times\bu^N(s))\times\bu^N(s)}{P_N\bvphi}\,ds
 \\&\notag\quad
 +\lambda\int_0^t\pair{\nabla\bu^N(s)}{\nabla\bvphi} \,ds
 +\int_0^t\pair{\triangle\nabla\bu^N(s)}{\nabla\bvphi}\,ds.
\end{align}
Passing to the limit as $N\rightarrow\infty$, using \eqref{H1bound} and \eqref{strong_conv}, and arguing as in the proof of Theorem \ref{thm_strong_local} for the convergence of the nonlinear term, we obtain for \textit{a.e.} $t\in[0,T]$
\begin{align}\label{kse_twist_weak_form}
 \pair{\bu(t)}{\bvphi}
 &=
  \pair{\bu_0}{\bvphi}
 -\int_0^t\ip{(\nabla\times\bu(s))\times\bu(s)}{\bvphi}\,ds
 \\&\notag\quad
 +\lambda\int_0^t\pair{\nabla\bu(s)}{\nabla\bvphi} \,ds
 +\int_0^t\pair{\triangle\nabla\bu(s)}{\nabla\bvphi}\,ds.
\end{align}
It is straight-forward to see that the above bounds imply that the integrands are in $L^1((0,T))$; hence $\lim_{t\rightarrow0^+}\pair{\bu(t)}{\bvphi}=\pair{\bu_0}{\bvphi}$, so that the initial data is satisfied in the sense of weak continuity in time.  Since $\bu \in C([0,T];L^2)$, it follows that $\bu(0)=\bu_0$ in the sense of $C([0,T];L^2)$.
To show that solutions are unique, suppose that $\bu,\bv$ are solutions to \eqref{vKSE}, and let $\bw =  \bu-\bv$.  Then
\begin{align*}
 \bw_t + (\nabla\times\bu)\times\bw + (\nabla\times\bw)\times\bv + \lambda\triangle\bw + \triangle^2\bw &= \mathbf{0}.
\end{align*}
Taking the action of both sides on $\bw\in L^\infty([0,T];H^1)\cap L^2(0,T;H^3)$, and using \eqref{cross_prod_ortho}, we obtain
\begin{align*}
  \frac12\frac{d}{dt}\|\bw\|_{L^2}^2 + \|\triangle\bw\|_{L^2}^2 
  &=  
  -\lambda(\triangle\bw,\bw) - ((\nabla\times\bw)\times\bv,\bw)
  \\&\leq
  \lambda\|\triangle\bw\|_{L^2}\|\bw\|_{L^2}
  + \|\nabla\times\bw\|_{L^6}\|\bv\|_{L^3}\|\bw\|_{L^2}
  \\&\leq
  \lambda\|\triangle\bw\|_{L^2}\|\bw\|_{L^2}
  + C\|\triangle\bw\|_{L^2}\|\bv\|_{L^3}\|\bw\|_{L^2}
    \\&\leq
  \tfrac14\|\triangle\bw\|_{L^2}^2 + \lambda^2\|\bw\|_{L^2}^2
  + C\|\bv\|_{L^3}^{2}\|\bw\|_{L^2}^{2} + \tfrac14\|\triangle\bw\|_{L^2}^2.
\end{align*}
due to the Sobolev Embedding Theorem and elliptic regularity.  Thus,
\begin{align*}
 \frac{d}{dt}\|\bw\|_{L^2}^2 + \|\triangle\bw\|_{L^2}^2
 \leq
   C(\lambda^2 + \|\bv\|_{L^3}^{2})\|\bw\|_{L^2}^{2}.
\end{align*}
Since $\|\bv\|_{L^3}\leq C\|\bv\|_{H^{n/6}}<\infty$, Gr\"onwall's inequality implies that
\begin{align*}
 \|\bw(t)\|_{L^2}^2 + \int_0^t\|\triangle\bw(s)\|_{L^2}^2e^{C\int_\tau^t(1 + \|\bv(s)\|_{L^3}^{2})\,ds}\,d\tau
 \leq
 \|\bw(0)\|_{L^2}^2.
\end{align*}
Thus, solutions are unique, and moreover, depend continuously (in the $C([0,T];L^2)\cap L^2(0,T;H^2)$ sense) on the initial data.
\section{A rotational numerical scheme}\label{subsec_rot_num}
Consider equation the inviscid equation \eqref{Burgers_twist_inviscid}.
We comment on the well-posedness of this equation in Section \ref{sec_well_posedness_inviscid}. 
Smooth solutions to this equation (if they exist) do not change length, and in this section, we exploit this fact to obtain an implicit representation, as well as numerical scheme, at least in two dimensions.  The three-dimensional case is less straight-forward.  For instance, if we denote 
\begin{align}
W:=
    \begin{pmatrix}
    0&-\omega_3&\omega_2\\
    \omega_3&0&-\omega_1\\
    -\omega_2&\omega_1&0
    \end{pmatrix},
\end{align}
then \eqref{Burgers_twist_inviscid} can be written in the form
\begin{align}
    \bu_t + W\bu = 0,
\end{align}
from which one might try to use integrating factor methods.  Indeed, denoting the matrix $\Theta(t) := \int_0^tW(s)\,ds$ and the scalar $\xi(t) := \int_0^t|\bomega(t)|\,ds$ where $|\bomega|:=\sqrt{\omega_1^2+\omega_2^2+\omega_3^2}$.  Indeed, since $\Theta$ is skew-symmetric, its matrix exponential can be computed using the Euler-Rodrigues rotation formula; namely,
\begin{align}
    e^{\Theta(t)} &= 
    I + \sin(\xi)\frac{\Theta}{\xi} + (1-\cos(\xi))\pnt{\frac{\Theta}{\xi}}^2.
\end{align}
However, the following equality does \textit{not} hold in 3D
\[
\frac{d}{dt}e^{\Theta(t)} = We^{\Theta(t)} \qquad \text{[FALSE in 3D]}
\]
due to the fact that $W'(t)W(t) \neq W(t)W'(t)$, and one must instead use the Magnus expansion (see, e.g., \cite{Magnus_1954,Blanes_Casas_Oteo_Ros_2009}) to find the correct fundamental matrix of the system, an approach that seems computationally expensive.
\subsection{The two-dimensional case}\label{subsec_2D_algorithm}  Recall the 2D system \eqref{Burgers_twist_inviscid_2D}.
In this case, $W=\left(\begin{smallmatrix}0&-\omega\\\omega&0
  \end{smallmatrix}\right)$.  Hence, $W'(t)W(t)= W(t)W'(t)$, unlike in the 3D case.  Let $\theta(t):=\int_0^t\omega(s)\,ds$.  For notational clarity, we denote $\dot{\theta} = \pd{\theta}{t}=\omega$, and similarly for other time derivatives. Note that
\begin{align}
&\quad\label{2D_int_factor}
    \pd{}{t}
    \begin{pmatrix}
        u_1\cos\theta - u_2\sin\theta\\
        u_1\sin\theta + u_2\cos\theta
    \end{pmatrix}
    =
    \begin{pmatrix}
        \dot{u}_1\cos\theta -u_1\dot{\theta}\sin\theta
        - \dot{u}_2\sin\theta - u_2\dot{\theta}\cos\theta\\
        \dot{u}_1\sin\theta + u_1\dot{\theta}\cos\theta  
        + \dot{u}_2\cos\theta - u_2\dot{\theta}\sin\theta
    \end{pmatrix}
    \\&=\notag
    \begin{pmatrix}
        \cos\theta &- \sin\theta\\
        \sin\theta &\phantom{+} \cos\theta 
    \end{pmatrix}
    \pd{\bu}{t}
    +
    \dot{\theta}
    \begin{pmatrix}
        \cos\theta &- \sin\theta\\
        \sin\theta &\phantom{+} \cos\theta 
    \end{pmatrix}
    \bu^\perp
    \\&=\notag
    \begin{pmatrix}
        \cos\theta &- \sin\theta\\
        \sin\theta &\phantom{+} \cos\theta 
    \end{pmatrix}
    \pnt{
    \pd{\bu}{t}
    +
    \omega
    \bu^\perp
    }
    =\mathbf{0}.
\end{align}
Hence, integrating the above equation in time on the interval $[0,t]$, we obtain
\begin{align*}
&\quad
    \begin{pmatrix}
        u_1(t)\cos\theta(t) - u_2(t)\sin\theta(t)\\
        u_1(t)\sin\theta(t) + u_2(t)\cos\theta(t)
    \end{pmatrix}
    =
    \begin{pmatrix}
        u_1(0)\cos\theta(0) - u_2(t)\sin\theta(0)\\
        u_1(0)\sin\theta(0) + u_2(t)\cos\theta(0)
    \end{pmatrix}.
\end{align*}
Noting that $\theta(0)=0$, the above equation can be written as
\begin{align*}
&\quad
    \begin{pmatrix}
        \cos\theta(t) &- \sin\theta(t)\\
        \sin\theta(t) &\phantom{+} \cos\theta(t)
    \end{pmatrix}
    \bu(t)
    =\bu_0.
\end{align*}
Hence, inverting the above matrix, we obtain the following implicit solution:
\begin{align}\label{inviscid_implicit_2D}
\bu(t) = \begin{pmatrix}
    \phantom{+}\cos\theta(t) & \sin\theta(t)\\
       -\sin\theta(t) & \cos\theta(t)
    \end{pmatrix}
    \bu_0,
    \qquad
    \theta(t) := \int_0^t\nabla^\perp\cdot\bu(s)\,ds.
\end{align}
Thus, we see that the $\bu$ evolves by pure rotation, but the angle of rotation depends on the time history of the vorticity.  Returning to \eqref{2D_int_factor} and integrating on a time interval $[t_n,t_{n+1}]$ yields
\begin{align*}
&\quad
    \begin{pmatrix}
        \cos\theta(t_{n+1}) &- \sin\theta(t_{n+1})\\
        \sin\theta(t_{n+1}) &\phantom{+} \cos\theta(t_{n+1})
    \end{pmatrix}
    \bu(t_{n+1})
    =
    \begin{pmatrix}
        \cos\theta(t_{n}) &- \sin\theta(t_{n})\\
        \sin\theta(t_{n}) &\phantom{+} \cos\theta(t_{n})
    \end{pmatrix}
    \bu(t_{n}).
\end{align*}
Using standard trigonometric identities, this implies
\begin{align}
    \bu(t_{n+1})
    &= \notag
    \begin{pmatrix}
        \phantom{+}\cos\theta(t_{n+1}) & \sin\theta(t_{n+1})\\
        -\sin\theta(t_{n+1}) & \cos\theta(t_{n+1})
    \end{pmatrix}
    \begin{pmatrix}
        \cos\theta(t_{n}) &- \sin\theta(t_{n})\\
        \sin\theta(t_{n}) &\phantom{+} \cos\theta(t_{n})
    \end{pmatrix}
    \bu(t_{n})
    \\&= \notag
    \begin{pmatrix}
        \cos(\theta(t_{n+1})-\theta(t_{n}))
        & \sin(\theta(t_{n+1})-\theta(t_{n}))\\
        -\sin(\theta(t_{n+1})-\theta(t_{n}))
        &\cos(\theta(t_{n+1})-\theta(t_{n}))
    \end{pmatrix}
    \bu(t_{n})
    \\&= \notag
    \begin{pmatrix}
        \cos(\int_{t_n}^{t_{n+1}}\omega(s)\,ds)
        & \sin(\int_{t_n}^{t_{n+1}}\omega(s)\,ds)\\
        -\sin(\int_{t_n}^{t_{n+1}}\omega(s)\,ds)
        &\cos(\int_{t_n}^{t_{n+1}}\omega(s)\,ds)
    \end{pmatrix}
    \bu(t_{n}).
    \qquad
\end{align}
Note that at this stage, the above formula is exact.  Approximating the integral by $\int_{t_n}^{t_{n+1}}\omega(s)\,ds\approx\Delta t\omega(t_n)$, $\omega^n\approx \omega(t_n)$, $\bu^n\approx\bu(t_n)$, we obtain the following iteration scheme:
\begin{align}\label{2D_euler_rot_update}
    \bu^{n+1} = \begin{pmatrix}
        \cos(\Delta t\omega^n)
        & \sin(\Delta t\omega^n)\\
        -\sin(\Delta t\omega^n)
        &\cos(\Delta t\omega^n)
    \end{pmatrix}
    \bu^n.
\end{align}
  In the forced, viscous case, after approximating the viscous term explicitly\footnote{It is perhaps an unusual quirk of this algorithm that the viscous term is handled explicitly, while the nonlinear term is handled implicitly, a reversal of standard algorithms for, e.g., Navier-Stokes.},  similar calculations lead to the algorithm
\begin{align}\label{2D_euler_rot_update_visc}
    \bu^{n+1} = \begin{pmatrix}
        \cos(\Delta t\omega^n)
        & \sin(\Delta t\omega^n)\\
        -\sin(\Delta t\omega^n)
        &\cos(\Delta t\omega^n)
    \end{pmatrix}
    [\bu^n+(\Delta t)(\triangle\bu^n + \bbf(t_n))].
\end{align}
Of course, using higher-order quadrature rules to approximate the integral, one may obtain a wide variety of schemes.
\section{Simulations}\label{sec_simulations}
In this section, we present the results of a simulation of \eqref{Burgers_twist}.  The purpose of this section is not to give a detailed numerical study, nor to compare with related equations, but merely to give some indication of the complexity of the large-time dynamics, and to demonstrate some possible dynamical behavior of the rotational Burgers equation \eqref{Burgers_twist}.  We do not investigate the inviscid case $(\nu=0)$ computationally since, although the scheme \eqref{2D_euler_rot_update} exhibits excellent energy conservation (before blow-up), as may be expected from Theorem \ref{thm_inviscid_blow_up}, simulations blow up too quickly to allow for meaningful exploration of the large-time dynamics.  We also do not simulate the rotational Kuramoto-Sivashisnky equations \eqref{vKSE}, as handling the linear terms explicitly as in \eqref{2D_euler_rot_update_visc} would lead to the prohibitively restrictive CFL constraint $\Delta t\lesssim (\Delta x)^4$.
Simulations of \eqref{Burgers_twist} are shown in Figure \ref{figs:time1000pics}. 
Figure \ref{figs:time1000pics} (A) shows that the energy spectrum  $ E_k:=\sqrt{\sum_{k\leq|\vec{\ell}|<k+1}|\widehat{u}^{\vec{\ell}}|^2}$ is well-resolved, which was the case for the entire time evolution of the system. We make no claims about the shape of $E_k$, save that in the range $10\lesssim k\lesssim 30$,  it appears to be roughly of the form $k^{-p}$ for some $p\approx2$.  Figures \ref{figs:time1000pics} (B), (C), and (D) illustrate that the dynamics of the flow are non-trivial and contain a wide range of length scales, that seem to be organized into structures.  The video \cite{Larios_2025_Burgers_Twist_viscous_video} of the simulation indicates a chaotic or perhaps quasi-periodic evolution.  This supported by norms of the solution, pictured in Figure \ref{figs:EEplanes}, where we see the fast evolution of the system to what appears to be an absorbing set, at which point the dynamics appear to be chaotic.

\begin{figure}[htp!]
\centering
\begin{subfigure}[t]{.4\textwidth}
    \includegraphics[width=\linewidth,trim=0 14 0 0,clip]{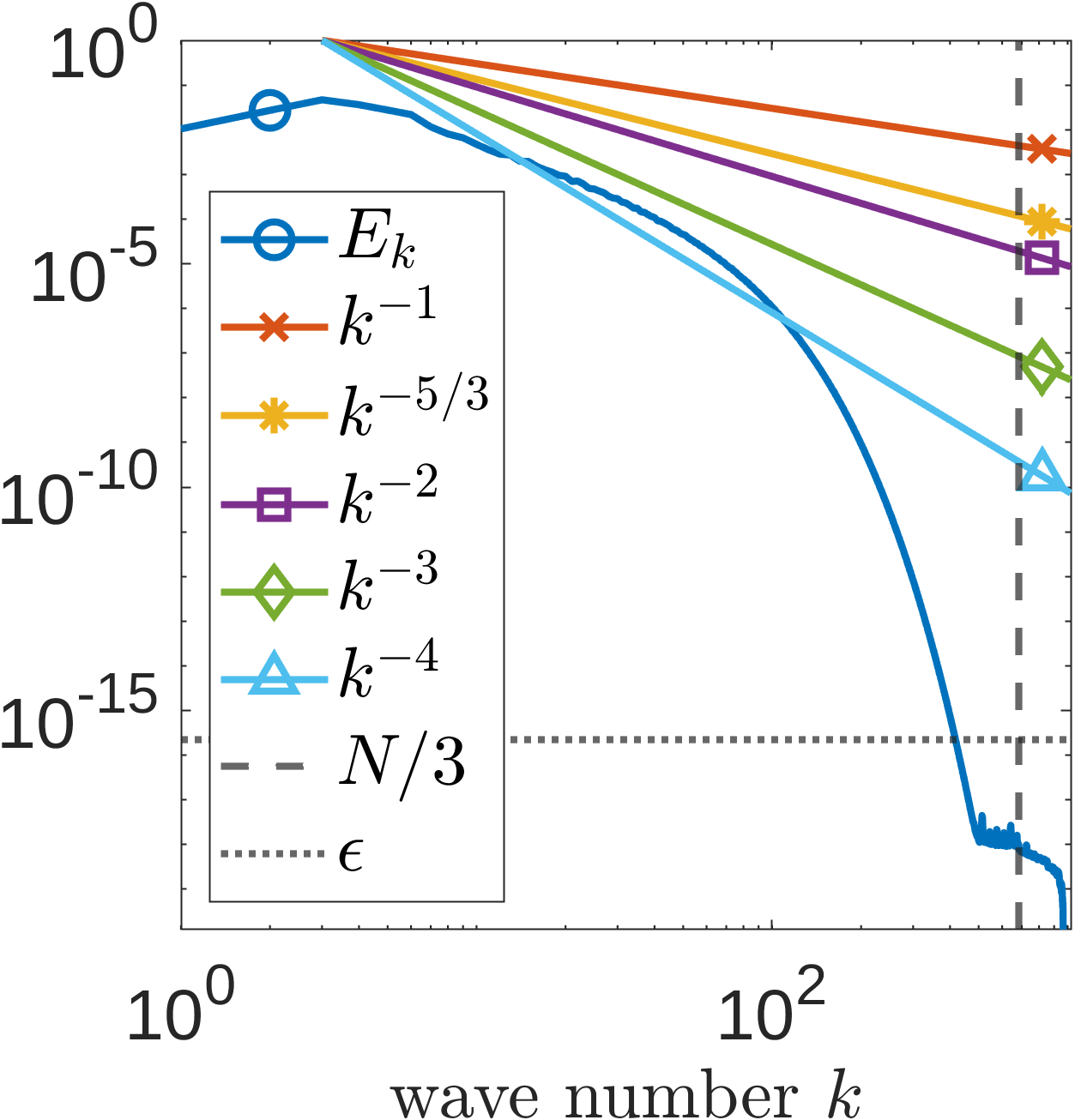}
    \caption{\footnotesize Energy spectrum $E_k$ vs. $k$}
\end{subfigure}
\hfill
\begin{subfigure}[t]{.49\textwidth}
    \includegraphics[width=\linewidth,trim=10mm 10mm 0 8mm,clip]{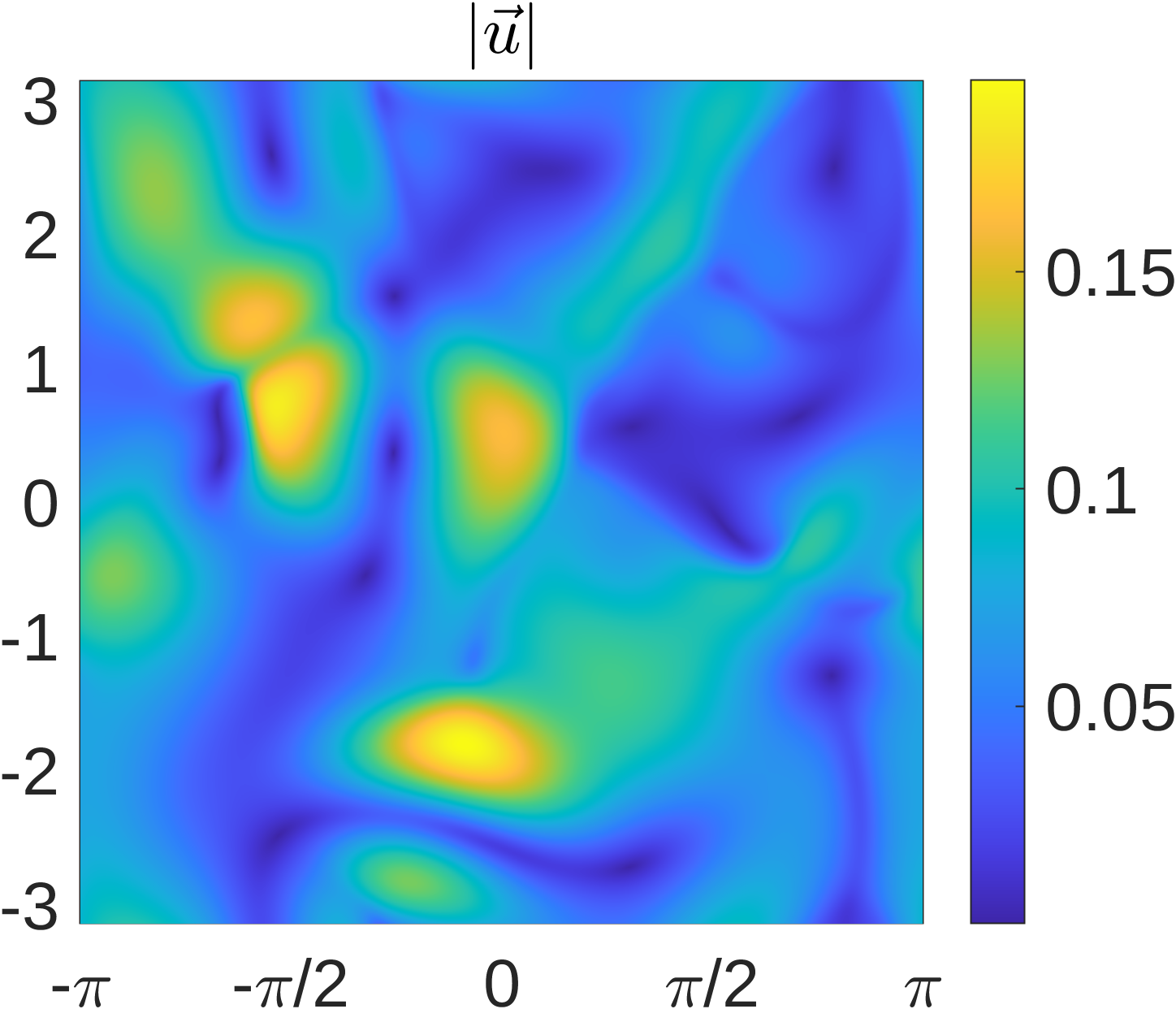}
    \caption{\footnotesize Magnitude $|\bu|$\phantom{mmm}}
\end{subfigure}
\begin{subfigure}[t]{.49\textwidth}
    \includegraphics[width=\linewidth,trim=13mm 10mm 0 7mm,clip]{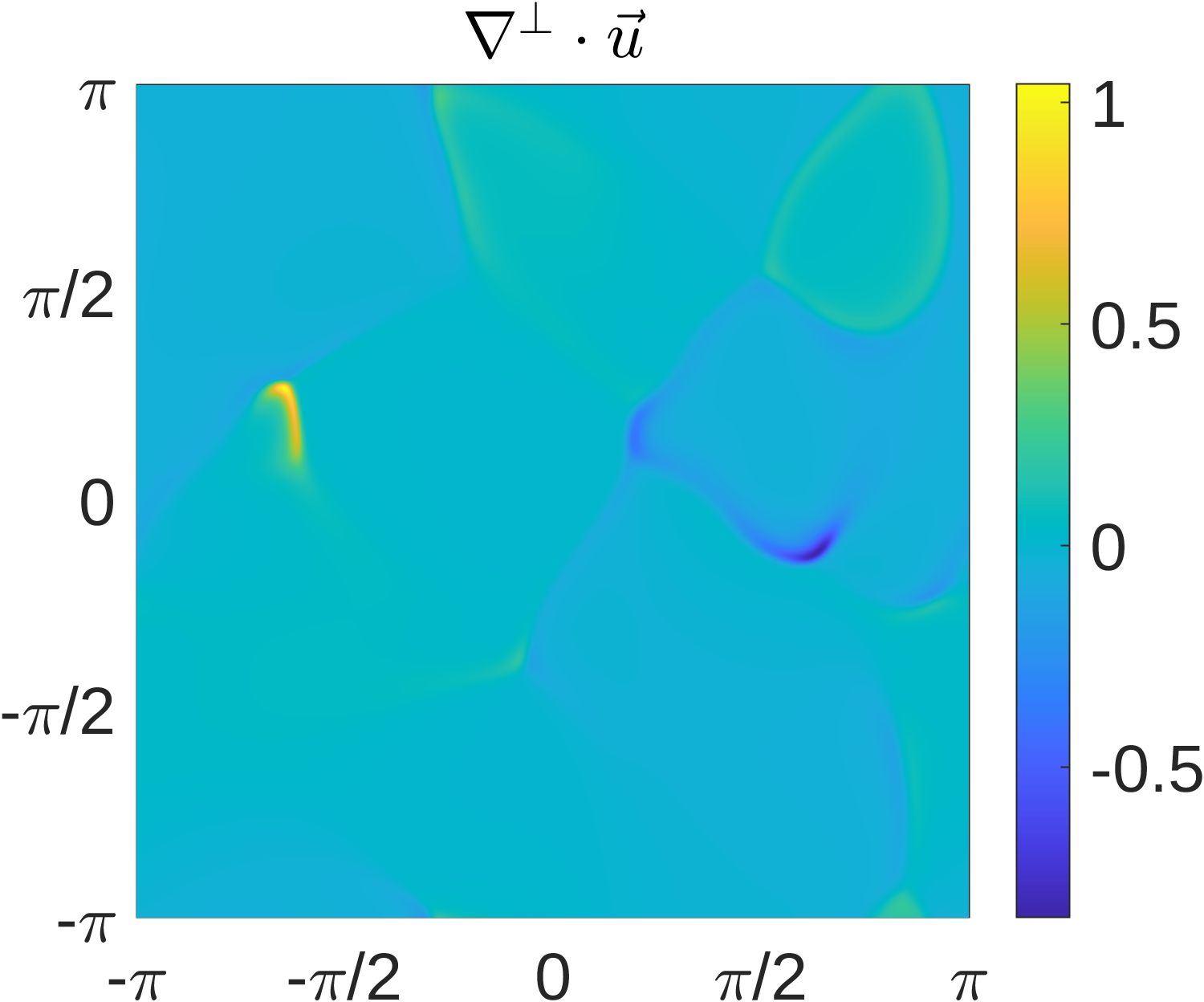}
    \caption{\footnotesize Curl $\nabla^\perp\cdot\bu$\phantom{mmm}}
\end{subfigure}
\hfill
\begin{subfigure}[t]{.49\textwidth}
    \includegraphics[width=\linewidth,trim=12mm 12mm 0 7mm,clip]{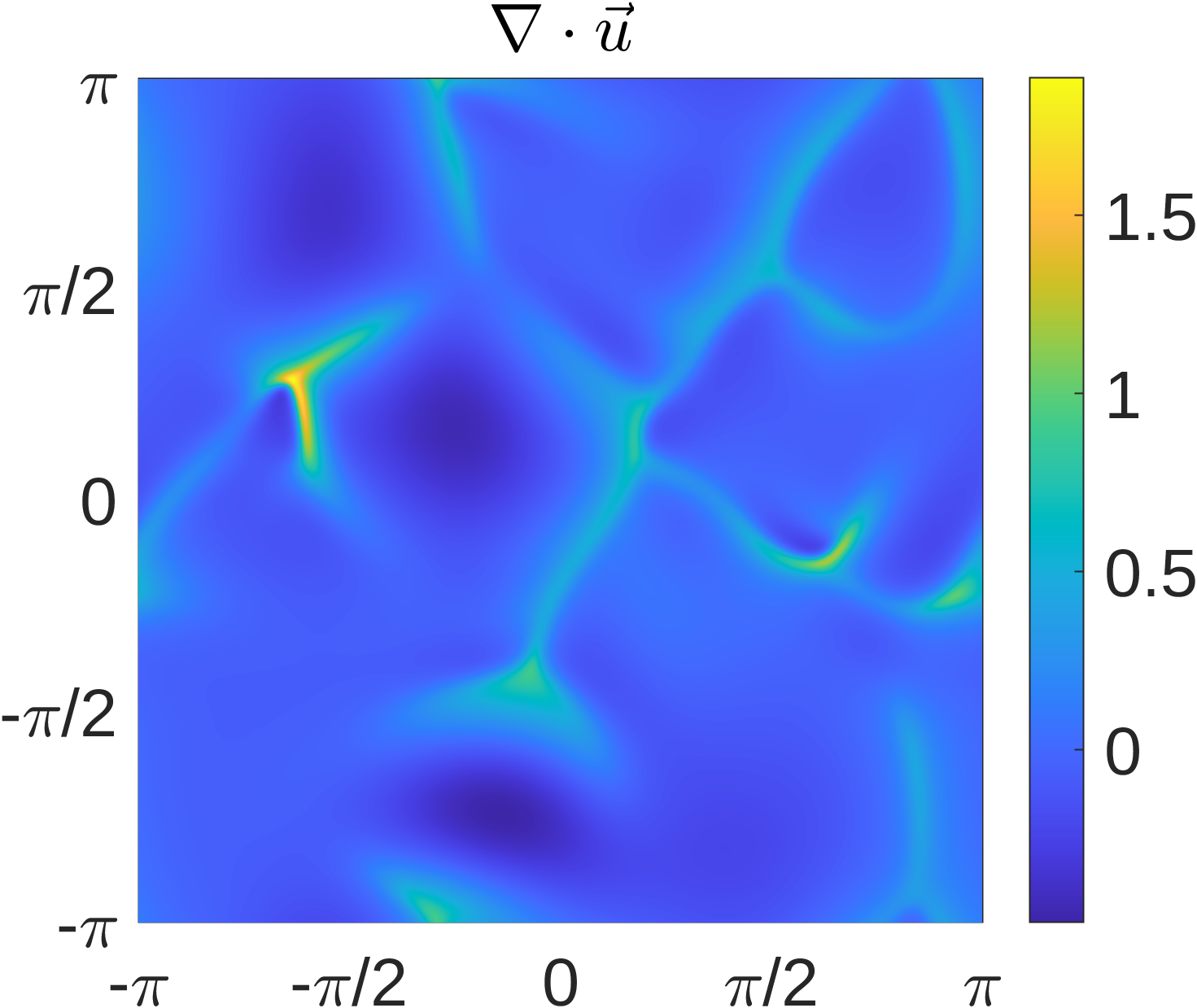}
    \caption{\footnotesize Divergence $\nabla\cdot\bu$\phantom{mm,,}}
\end{subfigure}
\caption{\label{figs:time1000pics}\footnotesize 
Snapshots of various quantities at time $T=1000$ for \eqref{Burgers_twist} (2D).  
Resolution: $N^2=2048^2$.  
(a) Energy spectrum $E_k$.  
Spectrum is resolved to machine precision ($\epsilon=2.22\times10^{-16}$) before dealiasing cutoff ($N/3$) with inertial range roughly $E_k\sim k^{-2}$.  
(b) Magnitude of the velocity.  
(c) 2D curl, showing negative and positive fine structures.  
(d) Divergence, with fine-scale positive structures, and coarse-scale negative structures.  
Unlabeled axes: (a) $E_k$ vs.~wave number $k$; (b), (c), and (d): standard $xy$-axes, $-\pi\leq x, y <\pi$.
}
\end{figure}

To generate Figures \ref{figs:time1000pics} and \ref{figs:EEplanes}, the algorithm discussed in Section \eqref{subsec_2D_algorithm} was coded and run on MATLAB version 2024a, using pseudo-spectral methods; i.e., derivatives computed in Fourier space, and products computed using collocation in physical space respecting the 2/3's dealiasing rule. The viscous term was computed explicitly from the previous time-step, using a viscosity of $\nu=0.001$. The spatial grid resolution was $N^2=2048^2$ ($\Delta x=\Delta y =2\pi/N$) on the periodic domain $[-\pi,\pi)^2$.   
The time step was $\Delta t\approx0.001534\approx 0.163(\Delta x)^2/\nu$ (i.e., respecting the viscous CFL constraint).  
The initial data was chosen as $\bu_0\equiv\mathbf{0}$.  
The force $\bbf$ was chosen to have normally-distributed real and imaginary parts (seeded initially with \texttt{rng(0)}) for Fourier coefficients $\widehat{\bbf}_{\bk}$, supported on the annulus $0.5\leq|\bk|\leq 2.5$ and scaled to have a Grashof number of $G=20$, where $G=\|\bbf\|_{L^2}/(\lambda_1\nu^2)$.  In fact, $\bbf$ is (the higher-resolution version of) the vector field $\vec{v}$ pictured in Figure \ref{figs:fields} (B).  A video of the simulation is available at \cite{Larios_2025_Burgers_Twist_viscous_video}, which includes the data for $\bbf$ in the description (data also available upon request).

\begin{figure}[htp!]
\centering
\begin{subfigure}[t]{.49\textwidth}
    \includegraphics[width=\linewidth,trim=0 0 0 0,clip]{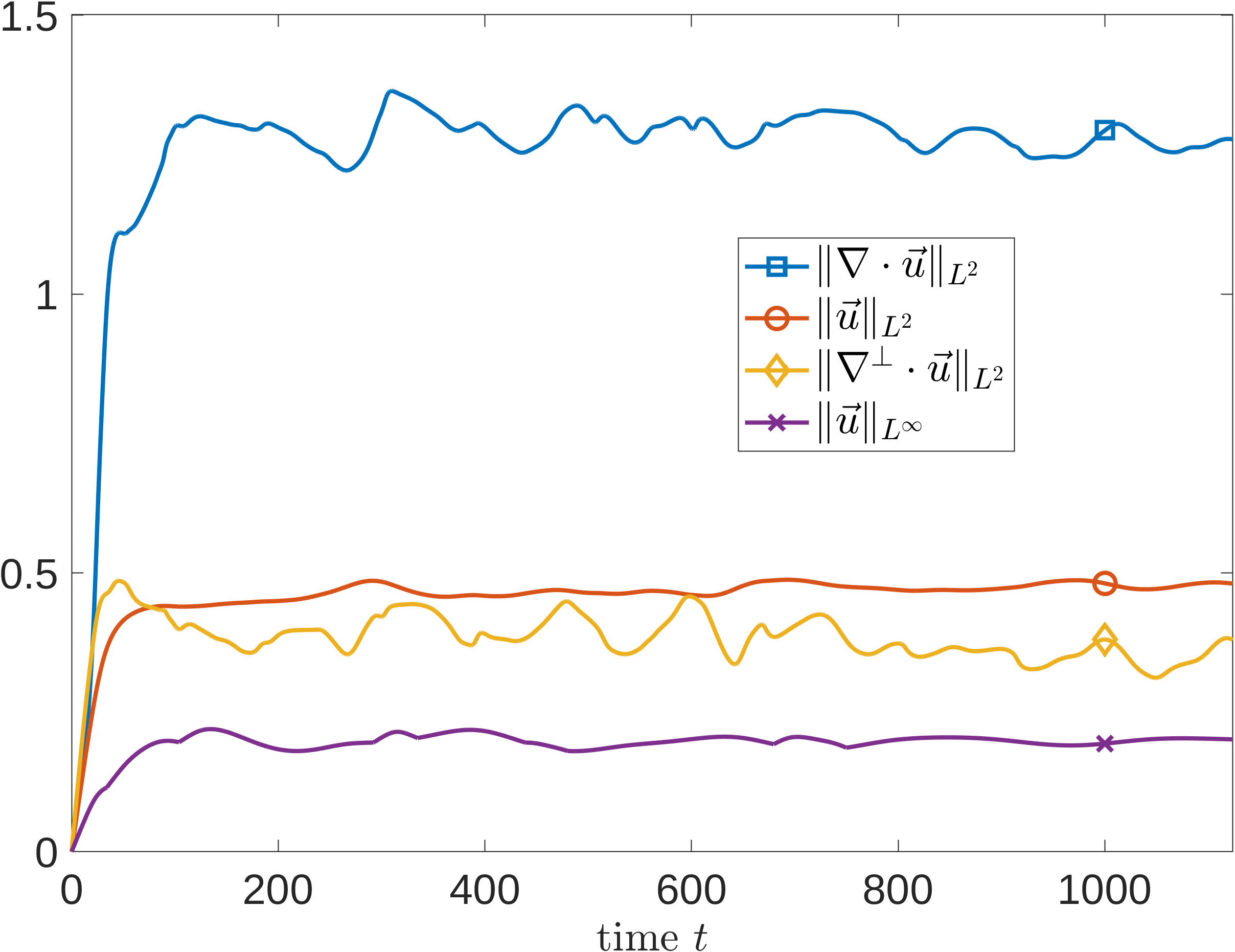}
    \caption{\footnotesize Norms of various quantities vs.~time}
\end{subfigure}
\hfill
\begin{subfigure}[t]{.49\textwidth}
    \includegraphics[width=\linewidth,trim=0 0 0 0,clip]{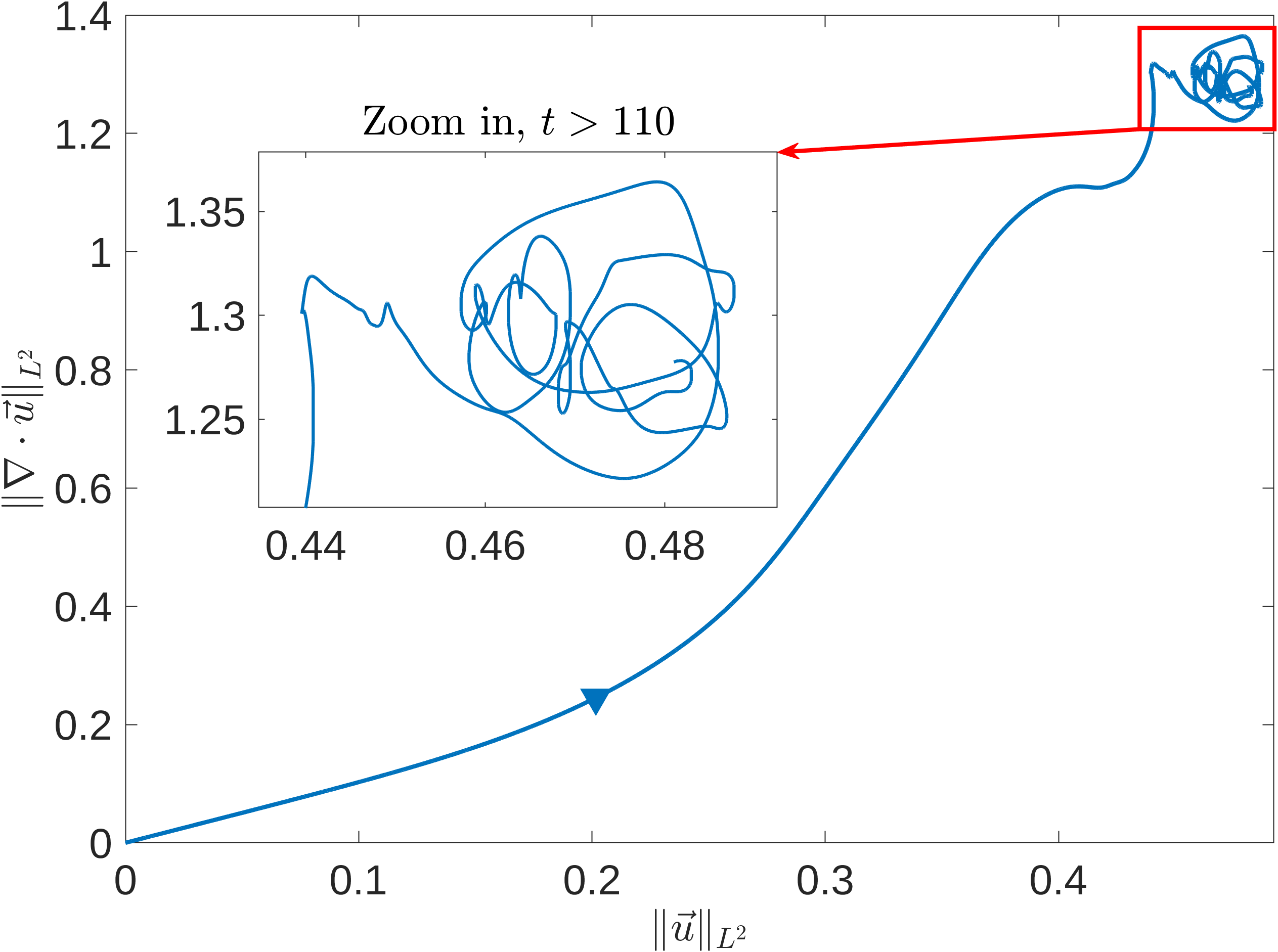}
    \caption{\footnotesize $\|\nabla\cdot \bu(t)\|_{L^2}$ vs.~$\|\bu(t)\|_{L^2}$}
\end{subfigure}
\begin{subfigure}[t]{.49\textwidth}
    \includegraphics[width=\linewidth,trim=0 0 0 0,clip]{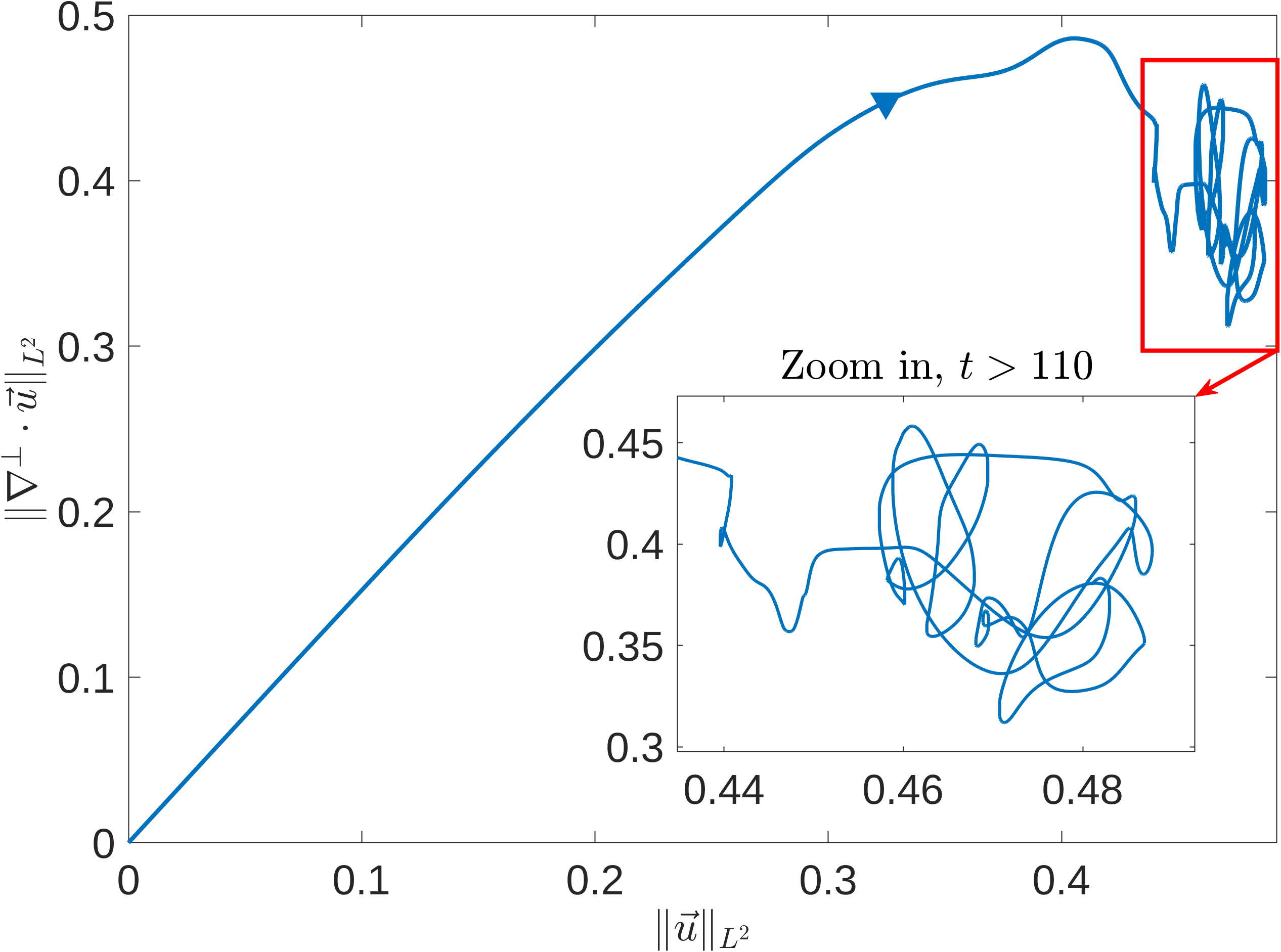}
    \caption{\footnotesize $\|\nabla^\perp\cdot \bu(t)\|_{L^2}$ vs.~$\|\bu(t)\|_{L^2}$}
\end{subfigure}
\hfill
\begin{subfigure}[t]{.49\textwidth}
    \includegraphics[width=\linewidth,trim=0 0 0 0,clip]{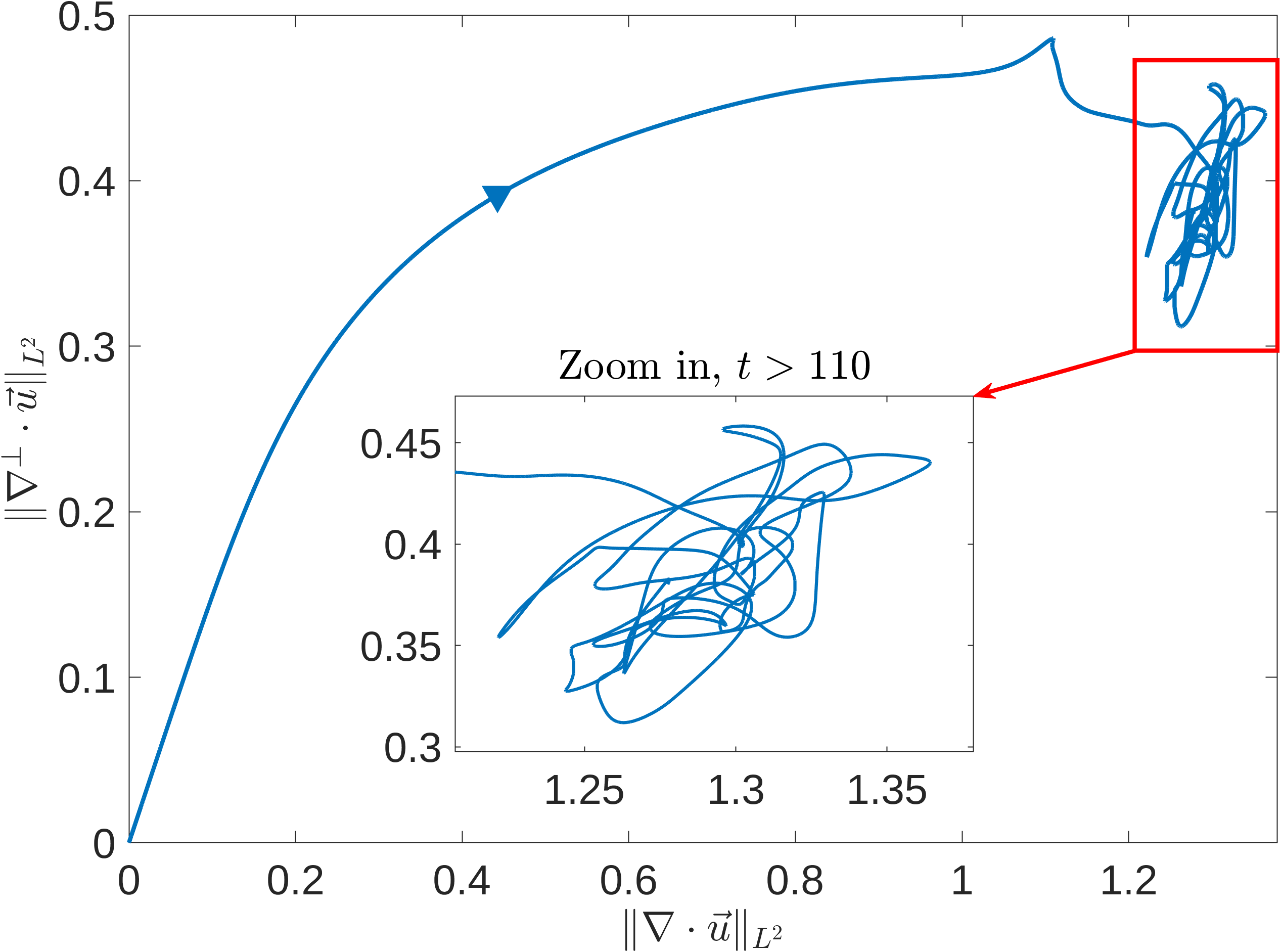}
    \caption{\footnotesize $\|\nabla^\perp\cdot \bu(t)\|_{L^2}$ vs.~$\|\nabla\cdot \bu(t)\|_{L^2}$}
\end{subfigure}
\caption{\label{figs:EEplanes}\footnotesize 
Large-time dynamics as illustrated by various norms of the solution to \eqref{Burgers_twist}.
}
\end{figure}

\FloatBarrier

\section{Additional Considerations}\label{sec_additional_considerations}
As mentioned in the introduction, equation \eqref{Burgers_twist} is distinct even from the usual one-dimensional form of the Burgers equation.  Therefore, we discuss some of the properties of this equation and related related equations.
\subsection{A possible scheme for the 2D Euler equations}
It is worth considering how one might modify the numerical scheme in Section \ref{subsec_2D_algorithm} into a scheme for solving the 2D incompressible Euler equations  (i.e., \eqref{NSE} with $\nu=0$).
For a solution $\bu(t)$ to 2D Euler, let us denote
\[
R(t) := 
    \begin{pmatrix}
        \cos\theta(t) &- \sin\theta(t)\\
        \sin\theta(t) &\phantom{+} \cos\theta(t)
    \end{pmatrix},
    \qquad
    \theta(t) := \int_0^t\nabla^\perp\cdot\bu(s)\,ds.
    \qquad
\]
The calculations in Section \ref{subsec_2D_algorithm} show that the 2D incompressible Euler equations can be formally written in the form
\begin{align}\label{rotational_Euler}
\pd{}{t}(R(t)\bu(t)) 
&= 
R(t)(\bu_t + \omega\bu^\perp)
=
-R(t)\nabla p.
\end{align}
Unfortunately, $R(t)\nabla p$ need not be a gradient, since in general the rotation matrix is spatially inhomogeneous; i.e., it typically rotates by a different angle at each point in space, so applying the Leray projection does not eliminate the pressure term.  Indeed, unlike the algorithm in \eqref{2D_euler_rot_update}, my preliminary calculations with schemes directly based on \eqref{rotational_Euler} (plus variations on solenoidal projection) exhibit bad energy conservation properties.  However, there is some hope that \eqref{rotational_Euler} could be used as part of, e.g., a predictor-corrector scheme, which may be investigated in a future work.
\subsection{Helicity}
We note that the inviscid versions of \eqref{NSE_vor}, \eqref{Burgers}, \eqref{Burgers_twist}, and \eqref{Burgers_twist_curl_curl} all conserve the helicity $(\bu,\bomega)$.  In particular, setting $\nu=0$ in any of these equations, we formally compute
\begin{align}
 \frac{d}{dt}(\bu,\bomega) 
 &= 
 (\bu_t,\bomega) + (\bu,\bomega_t) 
 = 
  -(\bomega\times\bu,\bomega) - (\bu,\nabla\times(\bomega\times\bu)) 
   \\&= \notag
  -(\bomega\times\bu,\bomega) + (\nabla\times \bu,\bomega\times\bu_t) 
  =0,
\end{align}
where the gradient term ($\nabla\pi$ in the Navier-Stokes case, $\frac{1}{2}\nabla|\bu|^2$ in the Burgers case) vanishes in the second equality due to the vector identity $\nabla\cdot(\nabla\times\mathbf{A})=0$ and integration by parts (using periodic boundary conditions), and the third equality follows by integration by parts, and \eqref{cross_prod_ortho}.
\subsection{Dynamics of the Divergence}
It might be that rotational Burgers is closer to Navier-Stokes in terms of energy transfer between scales.  Note that both the Burgers equation \eqref{Burgers} and rotational Burgers \eqref{Burgers_twist} have the same vorticity equation (taking $\nu=0$, $\bbf\equiv\mathbf{0}$, and denoting $D:=\nabla\cdot\bu$),
\begin{align}
     \partial_t \bomega + (\bu\cdot\nabla)\bomega
 &= (\bomega\cdot\nabla)\bu - \bomega D,
\end{align}
but different divergence equations.  For Burgers \eqref{Burgers},
\begin{align}
  \partial_t D + (\bu\cdot\nabla)D
  &=
  |\bomega|^2- |\nabla\bu|^2,\phantom{+\nabla\cdot((\nabla\bu)^T\bu)}
\end{align}
so, loosely speaking, the divergence is governed by a competition between ``rotation strength'' $|\bomega|^2$ and ``strain strength'' $|\nabla\bu|^2$
but for rotational Burgers \eqref{Burgers_twist},
\begin{align}\label{Burgers_twist_div_equation}
  \partial_t D + (\bu\cdot\nabla)D
  &=
  |\bomega|^2- |\nabla\bu|^2+\tfrac12\triangle|\bu|^2,\phantom{,} 
\end{align}
so that the divergence has additional an additional source term related to the curvature of the local kinetic energy, potentially creating stronger fluctuations in the divergence.  Note that we can also write \eqref{Burgers_twist_div_equation} as $\nabla\cdot \mathbf{L} = -\partial_tD$, where $\mathbf{L}:=\bomega\times\bu$ is the Lamb vector.  Hence, temporal fluctuations in $D$ directly affect the sign of $\nabla\cdot \mathbf{L}$.   Recalling Footnote \ref{footnote_Lamb} and the observations of \cite{Hamman_Klewicki_Kirby_2008_Lamb} that sign changes in $\nabla\cdot \mathbf{L}$ are correlated with strong turbulence, this may provide a mechanism behind some of the chaotic behavior observed in Section~\ref{sec_simulations}.
\subsection{A complex form in the 2D case}
We also note in passing that in the 2D inviscid ($\nu=0$) case, if we denote $\psi:=u_1 + iu_2 = \Re(\psi) + i\Im(\psi)$, then equation \eqref{Burgers_twist} can be written in complex scalar form as 
\begin{align}\label{Burgers_twist_complex}
\psi_t = i\psi L\psi, 
\end{align}
where $L$ stands for the linear differential operator defined by 
$L\psi := \partial_y\Re(\psi) + i\partial_x \Im(\psi)$.  This compact form may be of use, e.g., in computations.
\subsection{A pure rotational form}
We note in passing that one could also consider a more purely rotational modification of \eqref{Burgers_twist} by using the vector identity
\begin{align}\label{identity_curl_curl}
\triangle\bu=\nabla(\nabla\cdot\bu)-\nabla\times(\nabla\times\bu),
\end{align}
which for a divergence-free field $\bu$ becomes $\triangle\bu=-\nabla\times\bomega$.  Hence one might also consider the equation
\begin{align}\label{Burgers_twist_curl_curl}
\pd{\bu}{t}+\bomega\times\bu &= -\nu \nabla\times\bomega.
\end{align}
(There is no divergence-free assumption here, so \eqref{Burgers_twist_curl_curl} is not equivalent to \eqref{Burgers_twist}.)
A simple energy argument show that the following identity holds formally:
\begin{align}\label{pure_curl_L2}
\frac12\frac{d}{dt}\|\bu\|_{L^2}^2 + \nu\|\bomega\|_{L^2}^2 = 0.
\end{align}
While this may seem like enough to prove local well-posedness in, say, $L^\infty(0,T;L^2)\cap L^2(0,T;H(\text{curl}))$, this bound is not quite sufficient, as one would also need a bound on the time derivative $\pd{\bu}{t}$ in, e.g., $L^p(0,T;H(\text{curl})')$ (say, to get strong convergence from Aubin-Lions if using Galerkin methods).  However, doing so seems to require having some control over the full gradient $\nabla\bu$, but one does not have enough control coming from the ``viscous curl'' term alone (although it maybe be possible to work in higher-order spaces such as $L^\infty(0,T; H^3)$, etc.).  Moreover, even if one could show short-time well-posedness, it is not obvious that any analogue of \eqref{Burgers_vor_dot} holds, preventing a proof of a maximum principle, and hence of global well-posedness, at least using the approach discussed here.  In fact, using \eqref{identity_curl_curl}, one can rewrite \eqref{Burgers_twist_curl_curl} as
\begin{align}\label{Burgers_twist_curl_curl_div}
\pd{\bu}{t}+\bomega\times\bu &= \nu\triangle\bu -\nu\nabla(\nabla\cdot\bu),
\end{align}
and the analogue of \eqref{pure_curl_L2} is (formally)
\begin{align}\label{pure_curl_L2_div}
\frac12\frac{d}{dt}\|\bu\|_{L^2}^2 + \nu\|\nabla\bu\|_{L^2}^2 = \nu\|\nabla\cdot\bu\|_{L^2}^2,
\end{align}
indicating a possibly destabilizing effect of the divergence in equation \eqref{Burgers_twist_curl_curl}, which is reminiscent of the detabilziing effect of the divergence in the 2D KSE equations observed in \cite{Larios_Martinez_2024}.  We do not pursue these matters further here, although preliminary computational tests by the author (not shown here) indicate that solutions to \eqref{Burgers_twist_curl_curl} grow so rapidly that the simulated solution becomes infinite in only a few times steps.
\subsection{A KdV-Like equation}
It is far from clear how one might generalize  the Korteweg-de Vries ($KdV$) dispersive equation $u_t+uu_x=\alpha u_{xxx}$ to a higher-dimensional setting (see, e.g. the Kadomtsev-Petviashvili equation \cite{Kadomtsev_1970_stability,Petviashvili_1976_soliton}, or the Zakharov-Kuznetsov equation 
\cite{Zakharov_Kuznetsov_1974_3Dsolitons} for two possible generalizations).  Among the many difficulties, one is that in typical higher-dimensional generalizations, the nonlinear term does not vanish in $L^2$ energy estimates, unlike in the 1D case.  Therefore, one might consider the equation
\begin{align}\label{KdV_like_eqn}
\bu_t+(\nabla\times\bu)\times\bu &= \alpha\triangle\nabla\times\bu, \qquad \alpha\in\nR.
\end{align}
Clearly, the nonlinear term vanishes in $L^2$ energy estimates, but there seems to be more difficulty in handling the third-order term.  Of course, there is no claim that such an equation enjoys the many nice properties of the KdV equation, such as complete integrability.  However, if we look for a solution $\vect{b}$ which is a Beltrami flow, i.e., $\vect{b}$ satisfies $\nabla\times\vect{b}=\lambda\vect{b}$ for some $\lambda\in\nR$, then $\vect{b}$ is automatically divergence-free, so $\nabla\times(\nabla\times\vect{b}) = -\triangle\vect{b}$, then the equation simplifies to $\vect{b}_t = -\alpha\lambda^3\vect{b}$, which grows or decays depending on the sign of $\alpha\lambda$.
We also note that the Fourier modes of linearized equation satisfy $\partial_t\widehat{\bu}_\bk = -i\alpha|\bk|^2\bk\times\widehat{\bu}_\bk$, and hence are stationary for modes along the $\bk$ direction, and dispersive for modes perpendicular to the $\bk$ direction (i.e., divergence-free modes are dispersive). 
\subsection{Variations on the nonlinearity}
The 1D Kuramoto-Sivashinsky equations are known to be stabilized by the nonlinear term (see, e.g., \cite{Celik_Olson_Titi_2019} and the references therein); namely, in a mechanism first described  by E. Titi in the 1990's \cite{EdrissTitiPrivateCommunication} (see also \cite{Frisch_She_Thual_1986_JFM_viscoelastic}), the backward diffusion term in the KSE creates energy the low modes.  This energy then cascades to higher modes due to the effect of the nonlinear term, where is is dissipated by the fourth-order hyperdiffusion term.  It is unknown if such a mechanism holds for the 2D (or $n$D) KSE.  Following up on this idea, it is tempting to wonder if the nonlinear term can be ``enhanced'' in some way to cause an even greater cascade, while still maintaining the property that the nonlinearity is orthogonal to the flow.  Therefore, we propose the following equation, which has many nice properties.
\begin{align}\label{rotKSE_higher_order}
    \bu_t-((-\triangle)^{s/2}\bu)\times\bu + \lambda\triangle\bu + \triangle^2\bu &= \mathbf{0},
\end{align}
where $\lambda\geq0$ and $0<s\leq2$.  
In particular, it is straightforward to show global well-posedness of this equation in 3D periodic domains.  Indeed simple energy estimates show that, for initial data in $H^s$ and arbitrary $T>0$, there exists a unique solution in $L^\infty(0,T;H^s)\cap L^2(0,T;H^{1+s})$.  Lower-regularity solutions may also be possible.  Moreover, there are at least two energy balance laws, corresponding to taking inner products with $\bu$ and $(-\triangle)^{s/2}\bu$, reminiscent of the 2D NSE, and hence, one expects a forward energy cascade and a backwards enstrophy-like cascade (even in 3D). This backwards cascade combined with the backwards diffusion may be a source of destabilization of large scales when $\lambda$ is sufficiently large.  
For the sake of brevity, we delay the study of these questions to a forthcoming work.  Note that in 2D, the nonlinear term points out of the plane, so it may be better to consider a different nonlinear term instead, such as  $((-\triangle)^{s/2}(\nabla\times\bu))\times\bu$.
\subsection{Pedagogical Considerations}\label{sec_pedagogical}
We note that the rotational equations proposed in this work, in particular, ``Burgers with a twist'' \eqref{Burgers_twist}, ``Kuramoto-Sivashisnky with a twist'' \eqref{rotKSE_intro}, and the higher-order rotational equation \eqref{rotKSE_higher_order} (maybe ``Burgers with a twist and a side of extra derivatives''?) can serve as an early introduction to the mathematical analysis of Navier-Stokes-like equations.  For example, even in an advanced course in partial differential equations (PDEs), there is not always enough time to develop the well-posedness theory for the Navier-Stokes equations, since one needs to develop divergence-free spaces, the Leray-Helmholtz projection, the de Rham Theorem, and so on, in order to handle the pressure.  However, one can rather quickly develop for, e.g., \eqref{Burgers_twist}, a nearly analogous theory to that of the NSE after introducing Sobolev spaces (e.g., via Fourier series), avoiding all of the extra machinery that is needed to handle the pressure. Therefore, rotational equations of the forms proposed here may be a useful topic to include in time-limited courses.  Indeed, the author tried this with equation \eqref{Burgers_twist} in a one-semester advanced course on  PDEs in Fall 2023 at the University of Nebraska--Lincoln (Math 941), with successful results: after seeing a proof of the well-posedness of equation \eqref{Burgers_twist} in lectures, the students were able to prove the well-posedness of equation \eqref{rotKSE_higher_order} (with $s=2$) on a take-home assignment.
We also note that equation \eqref{Burgers_twist} serves as a possible counterexample to the often misused and misunderstood notion that ``\textit{a priori} estimates are sufficient'' to prove things like existence, uniqueness, etc.  Indeed, as pointed out in Remark \ref{remark_no_weak_Burgers_twist_sol}, one seemingly has all the \textit{a priori} estimates one would need to establish the existence of weak solutions.  However, this is not the case, since the convergence of the Galerkin solutions to a weak solution must still be shown, but for equation \eqref{Burgers_twist}, this requires strong convergence in $H^1$, while one only has weak convergence in this space.  This differs from the Navier-Stokes case, where strong convergence in $L^2$ is sufficient to establish the existence of weak solutions.  A student focusing only on \textit{a priori} estimates, without the corresponding rigorous convergence arguments, might overlook this subtle point.

\section{Conclusions}\label{sec_conclusions}
Given the equivalence of forms \eqref{NSE} and \eqref{NSE_vor} of the Navier-Stokes equations, there evidently no particular \textit{a priori} reason why the standard Burgers equation \eqref{Burgers} should be preferred over \eqref{Burgers_twist} 
as a simplified  phenomenological model for the Navier-Stokes equations.  Whereas \eqref{Burgers} emphasizes the advective character of the Navier-Stokes nonlinearity, \eqref{Burgers_twist}, \eqref{Burgers_twist} emphasizes its rotational character.  Moreover, \eqref{Burgers_twist} has the same energy balance at that of Navier-Stokes, has a global attractor, and has solutions which appear to be chaotic in simulations.  
Therefore, we proposed this as a model to study, and proved global well-posedness in the viscous case, and in the inviscid case proved short-time well posedness and exhibited finite-time blow up.  
Due to some subtleties in the model, we provided fully rigorous proofs, using Galerkin methods to establish short-time well-posedness of strong solutions in the viscous case, and then boot-strapping arguments to show the existence of classical solutions.  
With this level of smoothness, we were able to establish a maximum principle for the system, allowing for a proof of global well-posedness.   
We then established the existence of two forms of global attractor of a damped-driven version of the system: one attracts in the whole phase space $H^1$, but only in the $L^2$ norm, the other attracts in the $H^1$ and $L^\infty$ norms, but only on the restricted space $H^1\cap L^\infty$. Similar results can likely be established without damping in the bounded domain case with homogeneous Dirichlet boundary conditions, but to focus the discussion, we did not pursue that case here. 
We then proposed a rotational version of the 2D Kuramoto-Sivashinsky equations (which were the original idea for this project) and proved global well-posedness for the system.  
Next, we proposed a numerical scheme to handle the rotational nonlinearity, and used it in a simulation to exhibit some of the large-time dynamics of the system, which appear to be non-trivial and possibly chaotic.  Finally, we gave some consideration of related equations with rotational form that may be interesting to study, and some remarks on pedagogical applications of these ideas in the hope that they may be of use to others.

\section*{Acknowledgments}
 \noindent
The author would like to thank the Isaac Newton Institute for Mathematical Sciences, Cambridge, for support and warm hospitality during the 2022 programme ``Mathematical aspects of turbulence: where do we stand?'' where some of work on this manuscript was undertaken. This work was supported by EPSRC grant no. EP/R014604/1.
The author was partially supported by NSF grants DMS-2206762, DMS-2510494, CMMI-1953346, and USGS Grant No. G23AC00156-01.

\begin{scriptsize}

\end{scriptsize}
\end{document}